\documentclass[%
aip,
amsmath,amssymb,
reprint,%
]{revtex4-1}

\usepackage{Library/PaperPrerequesites}

\usepackage{xcolor}  
\usepackage{soul}  

\begin{document}
	
	
	\title{Nonautonomous Spectral Submanifolds for Model Reduction of Nonlinear Mechanical Systems under Parametric Resonance} 
	
	\author{Thomas Thurnher}
	\altaffiliation{Institute for Mechanical Systems, ETH Zürich, Leonhardstrasse 21, 8092 Zürich, Switzerland}
	\author{George Haller}
	\altaffiliation{Institute for Mechanical Systems, ETH Zürich, Leonhardstrasse 21, 8092 Zürich, Switzerland}
	\author{Shobhit Jain}
	\email{Shobhit.Jain@tudelft.nl}
	\altaffiliation{Delft Institute of Applied Mathematics, TU Delft, Mekelweg 4, 2628 CD Delft, The Netherlands}

	\date{\today}
	
	\begin{abstract}
		We use the recent theory of Spectral Submanifolds (SSM) for model reduction of nonlinear mechanical systems subject to parametric excitations. Specifically, we develop expressions for higher-order nonautonomous terms in the parameterization of SSMs and their reduced dynamics. We provide these results both for general first-order as well as second-order mechanical systems under periodic and quasiperiodic excitation using a multi-index based approach, thereby optimizing memory requirements and the computational procedure. We further provide theoretical results that simplify the SSM parametrization for general second-order dynamical systems. More practically, we show how the reduced dynamics on the SSM can be used to extract the resonance tongues and the forced response around the principal resonances in parametrically excited systems. In the case of two-dimensional SSMs, we formulate explicit expressions for computing the steady-state response as the zero-level set of a two-dimensional function for systems that are subject to external as well as parametric excitation. This allows us to parallelize the computation of the forced response over the range of excitation frequencies. We demonstrate our results on several examples of varying complexity, including finite-element type examples of mechanical systems. Furthermore, we provide an open-source implementation of all these results in the software package \texttt{SSMTool}.\end{abstract}
	
	\pacs{}
	
	\maketitle 

	\begin{quotation}
		As the complexity of engineering systems, such as MEMS devices, is growing, accurate reduced-order models (ROMs) are in high demand. Not only do such models promise quick routines for system design and analysis, but they can also deliver specifically desired information that might be inaccessible from a full system simulation. Reduction of the full system to Spectral Submanifolds (SSM) offers a rigorous mathematical procedure to derive such ROMs even for strongly nonlinear systems. In this paper, we develop a methodology for computing nonautonomous approximations to these SSMs up to arbitrary order, generalizing, improving and expanding available results in the literature. Using the exact ROM provided by the SSM-reduced dynamics, we analyze instabilities that arise due to parametric resonance and obtain the resulting steady states for a number of examples. Our results are also implemented in an open-source software library.
	\end{quotation}

	\section{Introduction}
	Parametric excitation is a widely observed phenomenon in mechanical systems that arises from time-dependence of parameters in a dynamical system. This time dependence is either due to external modulation or inherent nonlinear couplings \cite{Champneys2013}. Perhaps the simplest example is the periodic modulation of the length of a pendulum that may even destabilize its stable fixed point. This phenomenon is known as parametric resonance and occurs if the modulation frequency and the eigenfrequency of the pendulum assume a rational relationship.
	
	Due to its destabilizing effect, parametric resonance is often avoided in the operational regime of engineering systems~\cite{Sanches2012, Coleman1956}. However, it is also possible to take advantage of the resulting instability, which leads to exponential growth of energy and oscillation amplitudes. In rotor dynamics, for instance, parametric resonance may be used to virtually increase damping~\cite{Ecker2011}. Nonlinear energy harvesters show increased performance when driven in the regime of parametric resonance compared to systems using conventional resonance~\cite{Tehrani2015}. Micro-Electromechanical systems (MEMS) also frequently make use of parametric resonance. Examples include low-noise parametric amplifiers \cite{Rhoads2006_MEMS}, mass-sensing devices \cite{Li2014}, signal amplifiers \cite{Welte2013}, gyroscopes \cite{Polunin2017, Sharma2012}, picoprojectors, and microscanners \cite{Frangi2017}. Parametric resonance in combination with nonlinearites can also enable operation in a wider frequency range \cite{Rhoads2009}. 
	
	Among various types of parametric resonances, the \textit{principal parametric resonance}, where the external excitation frequency $\Omega$ and a system's modal frequency $\omega_i$ obey a relation of the form $\Omega = 2 \omega_i$, is of pronounced practical interest. Indeed, due to the presence of damping, higher-order parametric resonances become effectively inaccessible in most engineering structures \cite{Verhulst2002, Frangi2017, Verhulst2009, Champneys2013}. Nonetheless, high-order parametric resonances have been observed in specially engineered systems \cite{Jia2016}. 
	\subsection*{Forced response curves}
	\textit{Forced response curves} (FRC) visualize the relationship between a system's response amplitude and its excitation frequency. FRCs provide insights into the nonlinear characteristics of a system when compared with the linear frequency response, especially near resonance. Various techniques have been used to analyze parametrically excited systems. For instance, the Poincaré-Lindstedt method assumes the dependence of the period of the solution on a small excitation amplitude and expresses the periodic solution as an asymptotic series for weakly excited systems \cite{Verhulst1996}. The method of multiple scales introduces different temporal variables for different timescales. These are then used in perturbation expansions and can be applied to parametrically excited systems as well \cite{Warminski2005}. Other common methods found in the parametric excitation literature are averaging \cite{Rhoads2006_MEMS, Rhoads2006_Microbeam, Li2020, Polunin2017, Zaghari2016} and the analytical construction of approximate solutions, for instance, using a harmonic series \cite{Szabelski1995, Szabelski1997, Sorokin2015VibrationModulation, Aghamohammadi2022DynamicExcitations} or direct numerical simulation \cite{Warminski2003}. 
	
	For a systematic computation of FRCs, numerical continuation packages, such as AUTO~\cite{AUTO},  \textsc{coco}~\cite{Dankowicz2013}, or the harmonic balance-based NLvib\cite{Krack2019}, may be used. These methods are based on collocation or spectral discretization and typically suffer from the curse of dimensionality. Hence, their direct applications to high-dimensional finite element models are generally unfeasible \cite{Jain2021}.
	\subsection*{Stability Diagrams}
	Establishing the stability type of trivial fixed points is the first important step in the study of parametrically excited systems. The regions of instability and stability of these fixed points are plotted in an excitation amplitude vs. frequency diagram, often referred to as \textit{Strutt} or \textit{Ince Strutt} diagram. In these stability diagrams, the regions of instability are often called \textit{Floquet tongues} that extend in the direction of decreasing forcing amplitude as the system's damping decreases. The simplest technique for constructing stability diagrams is numerical time integration over a grid of amplitude and frequency parameters~\cite{Welte2013}. More sophisticated approaches employ Floquet theory and range from the method of infinte determinants in the conservative case \cite{Hansen1985} to spectral methods for obtaining the eigenvalues of the monodromy matrix~\cite{Lindh1970, Hansen1985, Welte2013}. In the extended phase space, a bifurcation of the trivial periodic orbits occurs along the stability boundary, which enables its computation using continuation techniques~\cite{Champneys2013}. 
	\subsection*{Model Reduction}
	The methods discussed above are well suited for computing FRCs and stability diagrams in low-dimensional examples under parametric excitation. For the analysis of realistic high-dimensional systems, however, reduced-order models (ROM) are necessary.  While projection onto modal subspaces is a powerful ROM technique for linear systems, in nonlinear systems, those subspaces lose their invariance. The dynamics and the trajectories projected on them thus no longer correspond to those of the full dynamical system. 
	
	In constrast, nonlinear normal modes (NNM)~\cite{Shaw1991} are invariant manifolds that perturb from linear modal subspaces and serve as the nonlinear analogues of those subspaces. Using projections to NNMs, Warminski et al.~\cite{Warminski2012} analyzed parametrically and externally excited mechanical systems. Sinha and coworkers~\cite{Sinha2005, SinhaManifolds} employed Lyapunov-Floquet transformations to parametrically forced systems to  obtain systems with autonomous linear part before computing NNMs. However, this method is feasible only for low-dimensional systems because of its reliance on the initial transformation. Furthermore, the existence and uniqueness of such NNMs is a priori unclear~\cite{ssmexist}.
	
	The theory of \textit{Spectral Submanifolds} (SSMs) \cite{ssmexist} has clarified the mathematical preliminaries required for rigorous model reduction of nonlinear mechanical systems. Specifically, an SSM is the unique smoothest invariant manifold that perturbs from a modal subspace under the addition of nonlinearities and small-amplitude parametric or external forcing. The existence and uniqueness of an SSM are guaranteed under appropriate non-resonance conditions \cite{ssmexist}. SSMs associated with slow eigenspaces (slow SSMs) attract full system trajectories. Hence, the reduced dynamics on an SSM serves as an exact reduced-order model for the full system. 
	
	In the recent years, significant advance has been made towards the automated computation of SSMs. In a first implementation, Ponsioen et al.~\cite{SSMTool} proposed a routine for automated computation of two-dimensional SSMs in modal coordinates. More recent efforts have focused on multidimensional SSM computation in physical coordinates, using only the eigenvectors associated with the master modal subspace~\cite{Jain2021}. This alleviated the computational burden imposed by a modal transformation and enabled SSM computations for high-dimensional finite element models. 
	
	In particular, for systems without internal resonances, the reduced dynamics on two-dimensional SSMs have been used to compute the forced response of periodically forced nonlinear systems without any numerical simulation, simply by finding the zero-level set of a function~\cite{isolatedresponsesten}. Coupled with the automated SSM computation procedure \cite{Jain2021}, this makes the extraction of FRCs very fast for high-dimensional finite element problems as well. If internal resonances are present, SSMs of higher dimension need to be computed over the linear modes involved in the resonance~\cite{MingwuPart1, MingwuPart2}. More recently, SSMs have also been constructed directly from individual trajectory data~\cite{Cenedese2022Data-drivenSystems}, used in control theory\cite{Mahlkneckt2022,Alora2022Data-DrivenRobots} and for systems with algebraic constraints\cite{Li2022ModelSubmanifolds}. All these developments, however, used only a leading-order approximation for the nonautonomous terms in the SSM and its reduced dynamics.
	
	To capture parametric effects, a high-order approximation of non-autonomous SSMs and their reduced dynamics is necessary. Ponsioen et al.~\cite{Ponsioen2020} have already presented such higher-order computations of the non-autonomous SSM, albeit in modal coordinates. For efficient memory management and for avoiding redundant computations, they employ a multi-index notation in their numerical algorithm. More recently, Opreni et al. \cite{Opreni2022} have developed expressions for the computation of nonautonomous SSMs in physical coordinates. Their expressions, however, are limited to systems with proportional damping and geometric nonlinearities up to cubic order. They use these expressions to analyze several mechanical systems subject to external periodic excitation (but no parametric forcing).
	\subsection*{Our contributions}
	Here we develop an automated algorithm for computing arbitrarily high-order approximations for non-autonomous SSMs and their reduced dynamics for mechanical systems in physical coordinates with general damping and geometric/material nonlinearities of arbitrary polynomial order. Furthermore, we develop separate expressions for general first-order dynamical systems in the Appendix. While our applications focus on periodically excited systems, our expressions and their implementations are developed for the general quasiperiodic case. In contrast to the prior approaches~\cite{Jain2021, Vizzaccaro2022, Opreni2022} that use a tensor-based notation, we use a multiindex notation, which minimizes the memory requirements and the number of equations that need to be solved at each polynomial degree. Indeed, we argue that for two-dimensional SSMs, the overhead of equations resulting from the use of tensors increases exponentially with the order of the approximation. 
	
	In prior work~\cite{Vizzaccaro2022, Opreni2022}, a special structure in the invariance equations associated with second-order mechanical systems containing up to cubic-order geometric nonlinearity and proportional damping was used to reduce the number of equations in SSM computation. We now provide a theoretical foundation and physical intuition for this structure, linked to SSM theory and computations for general second-order dynamical systems. 
	
	For the case of two-dimensional periodic SSMs, we develop the following results. For the first time, we demonstrate how SSM theory can be used to obtain stability diagrams for the principal parametric resonance of finite element examples directly from the ROM. Additionally, we analyze the forced response arising from these instabilities using the reduced dynamics on the SSM. We show how the FRCs of mechanical systems subject to external as well as parametric periodic excitation can be obtained using an analytic transformation of the reduced dynamics on the SSM, without the need for any numerical continuation. This result allows for parallelization of the FRC computation over any desired frequency range. Finally, we provide an open-source implementation of our computational routines and numerical examples in the MATLAB-based software package SSMTool 2.4~\cite{SSMTool2}.
	
	The remainder of the paper is structured as follows. We discuss the general class of mechanical systems we treat in Section~\ref{sec:Setup}. Section \ref{sec:SSMs} provides the basic results from SSM theory relevant for our computations. In Section~\ref{sec:CompSSM}, we present the autonomous and nonautonomous invariance equations associated to the SSM, and develop specific resonance conditions which determine the normal form of the SSM-reduced dynamics. We provide an automated routine for the extraction of stability diagrams in Section~\ref{sec:SDSSM} along with numerical examples. Furthermore, we develop explicit expressions that enable analytical extraction of FRCs from two-dimensional, nonautonomous SSMs along with examples in Section~\ref{sec:FRCSSM}. Finally, we draw conclusions in Section~\ref{sec:conc}.
	
	\section{Setup}
	\label{sec:Setup}
	
	We consider a quasiperiodic forced nonlinear mechanical system of dimension $n$. 
	\begin{align}
		\label{eq:DS}
		\bm{M\ddot{y} + C}&\bm{\dot{y} + Ky + \bm{f}(y,\dot{y})} = \epsilon \bm{g}(\bm\Omega t,\bm{y},\bm{\dot{y}})
		, \ \ 0\leq\epsilon \ll 1,
	\end{align}
	with $\bm{y} \in \mathbb{R}^n$ a generalized displacement; and $\bm{M}, \bm{C}, \bm{K} \in \mathbb{R}^{n\times n}$ denoting the mass, damping and stiffness matrices, respectively. The function $\bm{f}(\bm{y},\bm{\dot{y}})$ is a nonlinear function that is $r$ times continuously differentiable  with $r\ge 1$. The mechanical system is subject to periodic parametric excitation via $\bm{f}_\text{param}(\bm\Omega t,\bm{y},\bm{\dot{y}})$ and possibly to external periodic forcing~$\bm{f}_{\text{ext}}(\bm\Omega t) $ with the frequency vector $\bm \Omega \in\mathbb{R}^K$ for some $K\ge 0$.  Here $\bm{f}_\text{param}$ is a general function of  $\bm{y}$ and $ \bm{\dot{y}}$ and is taken to be $C^r$ in these variables. The two types of excitation are included in the general excitation function 
	\begin{equation}\label{eq:forcing}
		\bm{g}(\bm\Omega t,\bm{y},\bm{\dot{y}}) = \bm{f}_{\text{ext}}(\bm\Omega t) + \bm{f}_\text{param}(\bm\Omega t,\bm{y},\bm{\dot{y}}),
	\end{equation}
	where $\bm g$ will be autonomous for $K = 0$, periodic in $t$ for $K = 1$, and quasiperiodic in $t$ for $K > 1$ with $K$ rationally incommensurate frequencies. 
	
	The second-order system~\eqref{eq:DS} can be transformed to the first-order form
	\begin{align} \label{eq:DS_FirstOrder}
		\bm{B\dot{z} = Az + F(z) }+ \epsilon \bm{ G}(\bm\Omega t, \bm{z}),
	\end{align}
	where
	\begin{align} 
		\label{eq:choiceDS}
		\bm{z} =  \begin{bmatrix} \bm{y}  \\ \dot{\bm{y}} \end{bmatrix}, 
		\bm{A} =  \begin{bmatrix} -\bm{K} & \bm{0} \\ \bm{0} & \bm{M} \end{bmatrix}, 
		\bm{B} =  \begin{bmatrix} \bm{C} & \bm{M} \\ \bm{M} & \bm{0} \end{bmatrix}, 
		\\
		\bm{ F(z)} =  \begin{bmatrix}- \bm{f(y,\dot{y})} \\ \bm{0} \end{bmatrix}, 
		\bm{ G}(\bm\Omega t, \bm{z}) =  \begin{bmatrix} \bm{g}(\bm\Omega t,\bm{y},\bm{\dot{y}}) \\ \bm{0}\end{bmatrix},
	\end{align}
	with $\bm{z} \in \mathbb{R}^N, N = 2n$. This choice for the first-order form is not unique\cite{Jain2021} but all such forms can be used for the computation of SSMs. The autonomous linear part of system~\eqref{eq:DS_FirstOrder},
	\begin{align}\label{eq:DS_lin}
		\bm{B\dot{z} = Az},
	\end{align}
	has a fixed point at $\bm{z}=\bm{0}$, which we assume to be hyperbolic and stable. We focus on an even subset of $M << N $ eigenvalues and eigenvectors of system~\eqref{eq:DS_lin}, the choice of which is discussed below. Note that the task of finding specific subsets of eigenvalues and eigenvectors can be carried out efficiently with iterative algorithms, as presented, for instance, by Golub and Van Loan \cite{Golub2013MatrixComputations}. 
	
	The reduced left and right eigenvalue problems are given by
	\begin{align}\label{eq:LeftEigenvalue}
		(\bm{A}- \lambda_j \bm{B})\bm{v}_j = \bm{0},
		\qquad
		\bm{u}^*_j(\bm{A}- \lambda_j \bm{B}) = \bm{0}
		\\
		(\bm{A}- \conj{\lambda_j} \bm{B})\conj{\bm{v}_j} = \bm{0},
		\qquad
		\conj{\bm{u}^*_j}(\bm{A}- \conj{\lambda_j} \bm{B}) = \bm{0}
	\end{align}
	for $j=1,...,M/2$. We arrange these eigenvectors in the increasing order of magnitudes of the real parts of the associated eigenvalues, i.e.,
	\begin{align}
		\text{Re}(\lambda_{M/2}) \leq \text{Re}(\lambda_{M/2-1}) \le ... \leq  \text{Re}(\lambda_{1})  < 0.
	\end{align}
	The choice of $\bm A$ and $\bm B$ in eq.~\eqref{eq:choiceDS} results in displacement and velocity variables that are are inherently related, such that
	\begin{align} 
		\bm{v}_{j} =          \begin{bmatrix}
			\bm \phi_j 
			\\
			\bm \phi_j \lambda_{j}
		\end{bmatrix}
		, \
		\bm{u}_{j} = 
		\begin{bmatrix}
			\bm \theta_j 
			\\
			\bm \theta_j \conj{\lambda_{j}}
		\end{bmatrix},
		\\
		\conj{\bm{v}_{j}} =          \begin{bmatrix}
			\conj{\bm \phi_j} 
			\\
			\conj{\bm \phi_j \lambda_{j}}
		\end{bmatrix}
		, \
		\conj{\bm{u}_{j}} = 
		\begin{bmatrix}
			\conj{\bm \theta_j }
			\\
			\conj{\bm \theta_j }\lambda_{j}
		\end{bmatrix},
	\end{align}
	where the vectors $\bm \phi_j, \bm \theta_j \in\mathbb{C}^n$ are the right and left eigenvectors associated with the quadratic eigenvalue problem
	\begin{align}\label{eq:EigenprobSecondOrder}
		(\lambda_j^2 \bm{M} + \lambda_j \bm{C} + \bm{K}) \bm \phi_j =  \bm{0} , \quad 
		\bm \theta_j^*  (\lambda_j^2 \bm{M} + \lambda_j \bm{C} + \bm{K})   =  \bm{0}
		\\
		(\conj{\lambda_j^2} \bm{M} + \conj{\lambda_j} \bm{C} + \bm{K}) \conj{\bm \phi_j} =  \bm{0} , \quad 
		\conj{\bm \theta_j^*} ( \conj{\lambda_j^2} \bm{M} + \conj{\lambda_j} \bm{C} + \bm{K})   =  \bm{0}
	\end{align}
	for $j=1,...,M/2$. 
	
	Note that the above procedure is valid for any general damping matrix $\bm{C}$ and reduces to the computation of conservative eigenmodes when $\bm{C}$ is simultaneously diagonalizable with the mass and stiffness matrices, e.g., in the case of proportional damping. 
	
	\section{Non-autonomous invariant manifolds and model reduction}\label{sec:SSMs}
	
	We now introduce the methodology for computing the SSMs on which the ROM of the parametrically excited system \eqref{eq:DS} is constructed. The complex-conjugate pairs of right eigenvectors span spectral the subspaces $E_j = \text{span}\{ \bm{v}_{j} , \conj{\bm{v}_{j}} \}$, which are invariant under the linear flow of system~\eqref{eq:DS_lin}. To analyze the linear system response via a ROM, we select a $M$-dimensional spectral subspace ${\cal E} = E_{1} \oplus E_{2} \oplus ... \oplus E_{{M/2}}$, the \textit{master spectral subspace}. The modes $E_j$ are chosen such that the corresponding spectrum
	\begin{align}
		\text{Spec} (\mathcal{E})  =  \{\lambda_1, \conj{\lambda_1}  ..., \lambda_{M/2}, \conj{\lambda_{M/2}} \}
	\end{align}
	contains pairs of eigenvalues that potentially assume near-resonance relationships with either the forcing frequencies or among each other. In the absence of any (near) resonances, the subspace $\mathcal{E}$ typically contains the slowest modes of the system. The corresponding eigenvectors are normalised such that 
	\begin{align}\label{eq:Normalisation}
		(\bm{u}_i)^*\bm{B}\bm{v}_j = \delta_{ij}
	\end{align}
	which implies 
	$ \bm \theta_i^* \bm{C} \bm \phi_j  
	+ \lambda_i \bm \theta_i^* \bm{M} \bm \phi_j  
	+ \lambda_j \bm \theta_j^* \bm{M} \bm \phi_i  
	= \delta_{ij}$. Note that this normalization choice is different from the commonly adopted mass normalization for proportionally damped systems, i.e., $\bm \phi_i^* \bm{M} \bm \phi_j = \delta_{ij}$.
	
	Under the addition of nonlinearities to system~\eqref{eq:DS_lin}, the subspace $\mathcal{E}$ does not remain invariant. Model reduction onto $\mathcal{E}$ then looses its justification as the trajectories projected onto $\mathcal{E}$ no longer correspond to the trajectories of the full dynamical system. A mathematically rigorous alternative is to construct an invariant manifold that acts as a nonlinear continuation of $\mathcal{E}$. \textit{Spectral Submanifolds} (SSM)~\cite{ssmexist} are such continuations with their existence conditioned on a set of resonance conditions~\cite{Ponsioen2020}. Specifically, an \textit{autonomous} SSM, $\mathcal{W}(\mathcal{E})$, is the unique smoothest invariant manifold that perturbs smoothly from the spectral subspace $\mathcal{E}$ under the addition of the nonlinear autonomous terms $\bm{f(x,\dot{x})}$ to the linearized dynamical system \eqref{eq:DS_lin}. The dynamics on such an SSM are called its \textit{ reduced dynamics} and serve as an exact ROM of the full dynamical system~\eqref{eq:DS} for $\varepsilon=0$.
	
	Under the addition of quasiperiodic excitation (i.e., for $\varepsilon>0$, small), this autonomous manifold starts to oscillate and deform, becoming a \textit{non-autonomous} SSM, $\mathcal{W}_{\bm \gamma_\epsilon}(\mathcal{E})$. Instead of being anchored to the trivial fixed point $\bm{z = 0}$, the quasiperiodic SSM is now attached to $\mathcal{O}(\varepsilon)~C^r$-close invariant torus $\bm \gamma_\epsilon$ into which this fixed point perturbs\cite{Guckenheimer1983}. The tangency of the autonomous SSM to $\mathcal{E}$ implies that the nonautonomous SSM perturbs from the spectral subbundle $\bm \gamma_\epsilon  \times \mathcal{E}$.  
	Now, we will use $$\bm{\phi}=\boldsymbol{\Omega} t \in{\mathbb T} ^K$$ to denote the phase of the quasiperiodic excitation term. For computing $\mathcal{W}_{\bm \gamma_\epsilon}(\mathcal{E})$ and its reduced dynamics, the following results are useful \cite{ssmexist}.
	\begin{enumerate}
		\item $\mathcal{W}_{ \bm \gamma_\epsilon}(\mathcal{E})$ can be described by a time dependent parametrization 
		\begin{align}
			\bm{W}_\epsilon(\bm{p},\bm\phi): \mathcal{U} = U \times \mathbb{T}^K \mapsto \mathbb{R}^{N},\quad U \in \mathbb{C}^M,
		\end{align}
		from an open set $\mathcal{U}$ onto the phase space of the full system \eqref{eq:DS_FirstOrder}.
		\item The reduced vector field $\bm{R}_\epsilon ( \bm{p}, \bm\phi): \mathcal{U} \times \mathbb{T}^K \mapsto \mathcal{U}$ on $\mathcal{W}_{\bm \gamma_\epsilon}(\mathcal{E})$ satisfies the \textit{invariance equation}
		\begin{align}\label{eq:InvEq}
			\bm{B} ((\partial_{\bm p} {\bm W}_\epsilon) \bm{R}_\epsilon + (\partial_{\bm\phi} \bm{W}_\epsilon)\bm\Omega )
			= 
			\bm{A}\bm{W}_\epsilon + \bm{F}\circ \bm{W}_\epsilon 
			+ \epsilon \bm{G}(\bm\phi, \bm{W}_\epsilon)
		\end{align}
		and generates the reduced dynamics on $\mathcal{W}_{\bm \gamma_\epsilon}(\mathcal{E})$ via the ODE
		\begin{align}
			\label{eq:RedDyn}
			\dot{\bm{p}} = \bm{R}_\epsilon (\bm{p},\bm\phi),\quad  \dot{\bm\phi} = \bm\Omega.
		\end{align}
	\end{enumerate}
	If for a given dynamical system we wish to obtain a mathematically exact ROM, it is therefore our task to try and construct representations of the invariant manifold parametrization $\bm{W}_\epsilon$ and the reduced vector field $\bm{R}_\epsilon$. Once this has been achieved, the low-dimensional reduced dynamics can be analyzed in detail. Due to the invariance of the SSM $\mathcal{W}_{\bm \gamma_\epsilon}(\mathcal{E})$, any periodic orbit, invariant torus, and bifurcation that is observed in eq.~\eqref{eq:RedDyn} also exists in the full system and can be mapped back using the SSM parametrization $\bm{W}_\epsilon$ to the full system~\eqref{eq:DS}.
	
	\section{SSM Computation}
	\label{sec:CompSSM}
	
	In this section, we address how the SSM parametrization $\bm{W}_\epsilon$ and the reduced vector field $\bm{R}_\epsilon$ can be efficiently computed up to arbitrary orders of accuracy for general mechanical systems. We will merely require knowledge of the eigenvectors and eigenvalues of the underlying master modal subspace over which the SSM is constructed. All computations will be carried out in physical coordinates, without the need of a modal transformation. This will enable the treatment of large dynamical systems arising, e.g., from finite-element models.
	
	To eliminate the redundancy that would be introduced by a tensor-based notation for the expansions, we compute the SSM parametrization and the reduced dynamics with the use of multi-index notation (see Appendix \ref{App:TensorRedundancy} and Ref. \onlinecite{Ponsioen2020}). This optimizes the memory requirements for storing and manipulating the expansion coefficients and reduces the number of computations to be performed. We will present results for both the autonomous and nonautonomous cases. The general treatment for first-order ODEs that do not originate from an underlying second-order mechanical system is given in Appendix~\ref{app_sec:SSMComp_1storder}. The complete algorithm and the implementation of the examples treated here are openly available in the MATLAB-based software package SSMTool 2.4~\cite{SSMTool2}. 
	
	In summary, our treatment here extends the available results~\cite{Ponsioen2020,Jain2021,Vizzaccaro2022,Opreni2022} on automated computation of SSMs. Specifically, we combine the advantages of a multi-index approach \cite{Ponsioen2020}, yet use physical coordinates~\cite{Jain2021}. We compute the nonautonomous SSM up to arbitrary orders of approximation and exploit a special structure inherent in phase space for second-order systems~\cite{Vizzaccaro2022, Opreni2022}. We also provide the theoretical foundation for this special structure (see Lemmas~\ref{lem:HomAutVel},~\ref{Lem:HomNonautVel}), allowing extension of its advantages to any second-order system. At the same time, we consider more general systems, including direct parametric excitation, and geometric and system nonlinearities of arbitrary polynomial order including nonlinear damping.
	
	\subsection{SSM parametrization}
	For a second-order system, the phase space variables $\bm{z} =  (\bm{y}, \dot{\bm{y}})^T$ are inherently connected as the velocity components are simply the time derivative of the displacement variables. Naturally, this relationship carries over to the SSM parametrization ${\bm W}_\epsilon(\bm p,\bm\phi)$, which maps to the phase space. This allows us to express the SSM parametrization as
	\begin{align}
		\label{eq:SSM_param_struct}
		{\bm W}_\epsilon(\bm p,\bm\phi) 
		= 
		\begin{bmatrix}
			{\bm w}_\epsilon(\bm p,\bm\phi) \\
			\dot{ {\bm w}}_\epsilon(\bm p,\bm\phi)
		\end{bmatrix},
	\end{align}
	where ${\bm w}_\epsilon$ and $\dot{ {\bm w}}_\epsilon$ denote the parametrizations for the displacement and the velocity variables. The smooth dependence of the SSM $\mathcal{W}_{\bm \gamma_\epsilon}(\mathcal{E})$ on $\epsilon$ allows us to further expand these parametrizations in $\epsilon$ as
	\begin{align}
		\label{eq:ExpandW_e1}
		{\bm W}_\epsilon(\bm p,\bm\phi) &= \bm W (\bm p ) + \epsilon \bm{X} (\bm p,\bm \phi) + \cdots\\ \label{eq:ExpandW_e}
		&= \begin{bmatrix}
			\bm w (\bm p )\\
			\dot{\bm w} (\bm p )
		\end{bmatrix} 
		+ \epsilon \begin{bmatrix}
			\bm{x} (\bm p,\bm \phi)\\
			\dot{\bm x} (\bm p,\bm \phi)
		\end{bmatrix}
		+ \dots.
	\end{align}
	
	%
	Similarly, we expand the reduced dynamics parametrization as 
	\begin{align}\label{eq:ExpandR_e}
		{\bm R}_\epsilon \bm(p,\bm\phi) =
		\bm R ( \bm p )+ \epsilon \bm S (\bm p,\bm \phi) + O(\epsilon^2).
	\end{align}
	We will show that, in order to compute the SSM for second-order systems, it suffices to solve the invariance equation for the displacement parametrization ${\bm w}_\epsilon$ only. Subsequently, the velocity parametrization $\dot{ {\bm w}}_\epsilon$ can be analytically reconstructed by differentiating ${\bm w}_\epsilon$ with respect to time. 
	
	
	\subsection{Autonomous SSM ($\varepsilon=0$)}
	In order to compute autonomous SSMs, the \textit{autonomous invariance equation} 
	\begin{align}\label{eq:InvEq_aut}
		\bm{B} (\text{D}\bm{W}(\bm{p}) ) \bm{R}(\bm{p}) 
		= 
		\bm{A}\bm{W}(\bm{p}) + \bm{F} (\bm{W}(\bm{p}))
	\end{align}
	must be solved. This equation is obtained by substituting the expansions~\eqref{eq:ExpandW_e} and \eqref{eq:ExpandR_e} into the invariance equation~\eqref{eq:InvEq} and collecting terms at $\mathcal{O}(\epsilon^0)$. As we have noted, expansions of the SSM parametrization and the reduced dynamics will be written using multi-indices. A multi-index $\bm{m}$ is a vector $\bm{m} \in \mathbb{N}^M$ for which addition, subtraction and other operations are defined elementwise. With the help of multi-indices, multivariate monomials can be written as $ \bm{p}^{\bm m} = p_1^{m_1}\dots p_M^{m_M}$. 
	
	The autonomous parametrizations of the SSM are expanded using these multi-indices
	\begin{align}\label{eq:AutSSMParam}
		\bm{w}(\bm{p}) 
		&=  \sum_{\bm{m}\in \mathbb{N}^M} \bm{w}_{\bm{m}} \bm{p}^{\bm m},
		\\ \label{eq:AutSSMParamVel}
		\dot{\bm{w}}(\bm{p}) 
		&=  \sum_{\bm{m}\in \mathbb{N}^M} \dot{\bm{w}}_{\bm{m}} \bm{p}^{\bm m}
	\end{align}
	where $\bm{w}_{\bm m} , \bm{\dot{w}}_{\bm m} \in \mathbb{C}^n$. As we elaborate in Appendix~\ref{App:TensorRedundancy}, when compared to tensor-based approaches\cite{Jain2021, Vizzaccaro2022, Opreni2022}, this notation results in a drastic reduction in the number of systems of equations that need to be solved, depending on the SSM dimension and the order of approximation. For instance, in the case of two-dimensional SSMs, we observe an exponentially lower number of equations for multi-index notation relative to the tensor-notation as the order of approximation increases. 
	
	The autonomous reduced vector field  $\bm{R}\bm{(p)}$ is also expanded using multi-index notation as
	\begin{align}\label{eq:AutRDParam}
		\bm{R}(\bm{p}) 
		&=  \sum_{\bm{m}\in \mathbb{N}^M} \bm{R}_{\bm{m}} \bm{p}^{\bm m},
		\ \bm{R}_{\bm{m}} = 
		\begin{bmatrix}
			R^1_{\bm{m}}\\ \vdots \\  R^{M}_{\bm{m}}
		\end{bmatrix} 
	\end{align}
	Substituting the expansions~\eqref{eq:AutSSMParam},\eqref{eq:AutSSMParamVel} and \eqref{eq:AutRDParam} into the autonomous invariance equation \eqref{eq:InvEq_aut} results in linear system of equations that can be solved recursively at each order for the unknown coefficients $\bm{W}_{\bm{m}}, \bm{R}_{\bm{m}}$. These equations are decoupled for distinct multi-indices at each order. Furthermore, the splitting~\eqref{eq:AutSSMParam} of the parametrization yields a single $n$-dimensional system of linear equations that can be solved for the displacement coefficients $\bm{w}_{\bm{m}}$ according to the following statement.
	\begin{lemma} \label{lem:HomAut}
		The autonomous invariance equation~\eqref{eq:InvEq_aut} for a multi-index $\bm$ can be written as 
		\begin{align} \label{eq:Hom_Aut_Second}
			\underbrace{
				\bigg(
				\bm{K}  
				+ \Lambda_{\bm m} \bm{C} 
				+ \Lambda_{\bm m}^2 \bm{M}
				\bigg)
			}_{:= \bm{L}_{\bm m}}
			\bm{w}_{\bm m}
			=
			\bm{D}_{\bm m} 	\bm{R}_{\bm m}
			+ 
			\bm{C_m},
		\end{align}
		where $\bm{C}_{\bm m}\in \mathbb{C}^{n}$, $\bm{D}_{\bm m}\in \mathbb{C}^{n\times M}$ are defined in Appendix~\ref{app_sec:SSMComp_2ndorder}, and $\Lambda_{\bm m} = \bm \Lambda \cdot \bm{m}$ with $\bm{\Lambda} = (\lambda_1, ..., \conj{\lambda}_{M/2}) $.
	\end{lemma}
	\begin{proof}
		See Appendix~\ref{app_sec:SSMComp_2ndorder}.
	\end{proof}

	In system~\eqref{eq:Hom_Aut_Second}, the matrix $\bm{D_m}$ is computed solely based on the eigenvalues and eigenvectors associated to the master subspace, whereas $\bm{C_m}$ contains contributions from the lower-order terms in the expansion of the SSM and its reduced dynamics. Based on Lemma~\ref{lem:HomAut}, the system~\eqref{eq:Hom_Aut_Second} can be solved recursively up to any arbitrary order for the autonomous coefficients $\bm{w}_{\bm m}, \bm{R}_{\bm m}$ of the SSM parametrization and the reduced dynamics. At any order $m$, we have $z_m = \binom{m+M-1}{M-1}$ distinct multi-indices in the expansions~\eqref{eq:AutSSMParam}, \eqref{eq:AutRDParam}. Thus, at each order, we need to solve $z_m$ systems of $n$-dimensional linear equations. Indeed, for real dynamical systems, symmetries allow for further reduction in the number of equations, as we elaborate in Appendix~\ref{app_sec:SSMComp_1storder}. 
	
	\begin{remark}
		In case of near-resonances of the form
		\begin{align} \label{eq:ResCond_aut}
			\lambda_i \approx \bm{\Lambda}_{\bm{m}},
		\end{align}
		the coefficient matrix $\bm L_{\bm m}$ will be nearly singular~\cite{Jain2021}.
		If the resonance condition~\eqref{eq:ResCond_aut} is fulfilled for any eigenvalue $\lambda_i\in\text{Spec} (\mathcal{E})$ then a nonzero choice for $\bm{R}_{\bm m}$ is necessary to solve the linear system~\eqref{eq:Hom_Aut_Second}, where different choices lead to different styles of parametrization \cite{Haro2016,Jain2021}. 
	\end{remark}
	
	To understand the resonance condition~\eqref{eq:ResCond_aut}, consider a two-dimensional SSM tangent to a modal subspace spanned by lightly damped (complex-conjugate) modes with associated eigenvalues $\lambda$ and $\conj{\lambda}$ satisfying
	\begin{align}
		\lambda \approx l \lambda + (l-1) \conj{\lambda}
	\end{align}
	for all small enough $l \in \mathbb{N}^+$. This relationship holds for the pair $(\lambda,\conj\lambda)$ if its real part is small, i.e., underlying physical system is lightly damped. In this case, for any multi-index of the form $\bm{m} =  \begin{bmatrix} l,  & l-1 \end{bmatrix} $, the coefficient matrix $\bm{L}_{\bm m}$ in system~\eqref{eq:Hom_Aut_Second} turn out to be nearly singular.

	Finally, once we have computed the coefficients $\bm w_{\bm m}$ for the displacement parametrization $\bm{w(p)}$, the coefficients for the velocity parametrization can be directly obtained using the following statement.
	\begin{lemma} \label{lem:HomAutVel}
		The coefficients of the velocity parametrization corresponding to a multi-index $\bm{m}$ can be obtained directly from the displacement parametrization as
		\begin{align}
			\dot{\bm w}_{\bm m} = [\text{D}\bm w(\bm p) \bm R (\bm p)]_{\bm m} 
		\end{align}
	\end{lemma}
	\begin{proof}
		See Appendix~\ref{app_sec:SSMComp_2ndorder}.
	\end{proof}

	\subsection{Non-autonomous SSM ($\epsilon>0$)}
	In the non-autonomous setting, inserting the expansions~\eqref{eq:ExpandW_e},\eqref{eq:ExpandR_e} into the invariance equation~\eqref{eq:InvEq}, and collecting terms of 
	$\mathcal{O}(\epsilon)$ leads to the \textit{non-autonomous invariance equation}
	{\small
		\begin{align} \nonumber
			\bm{B}
			\bigg( 
			(\text{D}\bm{W}(\bm p) )\bm{S}(\bm p,\bm \phi)    
			+
			[\partial_{\bm p} \bm{X} (\bm p,\bm \phi)]  \bm{R} (\bm{p}) 
			+
			[\partial_{\bm\phi} \bm{X } (\bm p,\bm \phi)]  \bm\Omega 
			\bigg)
			=
			\\\label{eq:InvEq_nonaut}
			\bm{A}\bm{X}(\bm p,\bm \phi)  
			+	
			\big[\text{D}\bm{F} (\bm{W}(\bm p)) \big]\bm{X} (\bm p,\bm \phi)
			+
			\bm{G}( \bm{W}(\bm p,\bm \phi),\bm\phi), 
		\end{align}
	}
	which needs to be solved for the non-autonomous terms $\bm{X}$ and $\bm{S}$ in the parametrizations for the SSM and its reduced dynamics. As in the autonomous case, we now expand the nonlinear functions in Taylor series using multi-indices. The coefficients of these series, however, depend quasiperiodically on time. Hence, we futher expand these coefficients in temporal Fourier series to obtain
	\begin{align}\label{eq:expandNonAutSSM}
		\bm{x}(\bm{p},\bm\phi) 
		&=  \sum_{\bm{m}\in \mathbb{N}^M} \sum_{\bm\kappa \in \mathbb{Z}^k} \bm{x}_{\bm{m},\bm\kappa } e^{i \langle \bm\kappa , \bm\phi  \rangle} \bm{p}^{\bm m},
		\\ \label{eq:expandNonAutSSMVel}
		\dot{\bm{x}}(\bm{p},\bm\phi) 
		&=  \sum_{\bm{m}\in \mathbb{N}^M} \sum_{\bm\kappa \in \mathbb{Z}^k} \dot{\bm{x}}_{\bm{m},\bm\kappa } e^{i \langle \bm\kappa , \bm\phi  \rangle} \bm{p}^{\bm m}.
	\end{align}
	The reduced dynamics are similarly expanded in a Taylor-Fourier Series as
	\begin{align}
		\bm{S}(\bm{p},\bm\phi) 
		&=  \sum_{\bm{m}\in \mathbb{N}^M}     \sum_{\bm\kappa \in \mathbb{Z}^k} \bm{S}_{\bm{m},\bm\kappa } e^{i \langle \bm\kappa , \bm\phi  \rangle} \bm{p}^{\bm m} 
		, & \
		\bm{S}_{\bm{m},\bm\kappa}  
		&=  \begin{bmatrix}
			S^1_{\bm{m},\bm\kappa}\\ \vdots \\  S^{M}_{\bm{m},\bm\kappa}
		\end{bmatrix} .
		\label{eq:expandNonAutRD}
	\end{align}
	The procedure of determining these coefficients recursively is similar to that of the autonomous case. The contributions to distinct harmonics and to different multi-indices can be considered independently as the corresponding equations decouple. For second-order systems, the computation of these coefficients can again be carried out by first obtaining the displacement parametrization $\bm{x}(\bm{p},\bm\phi) $ via a reduced invariance equation of dimension $n$, followed by computing the velocity parametrization $\dot{\bm{x}}(\bm{p},\bm\phi) $ directly from the time derivative of the displacement terms, as we detail below.
	\begin{lemma} \label{Lem:HomNonaut}
		The non-autonomous invariance equation in terms of a multi-index $\bm{m}$ and harmonic $\bm \kappa$ is given by
		\begin{align}
			\label{eq:HomNonAut_Second}
			\underbrace{
				\big(
				\bm{K} 
				+
				\Lambda_{\bm{m}, \bm \kappa} \bm{C} 
				+
				\Lambda_{\bm{m}, \bm \kappa}^2  \bm{M}
				\big)
			}_{\bm L_{\bm{m},\bm \kappa}}
			\bm{x}_{\bm{m}, \bm \kappa }
			=
			\bm{D}_{\bm{m},\bm\kappa}
			\bm{S}_{\bm{m},\bm\kappa}
			+ 
			\bm{C}_{\bm{m}, \bm \kappa},
		\end{align}
		$\bm{C}_{\bm{m},\bm\kappa}\in \mathbb{C}^{n}$, $\bm{D}_{\bm{m},\bm\kappa}\in \mathbb{C}^{n\times M}$ are defined in Appendix~\ref{app_sec:SSMComp_2ndorder}, and $\Lambda_{\bm{m}, \bm \kappa} := \bm \Lambda \cdot \bm{m} + i \bm \Omega \cdot \bm \kappa $. 
	\end{lemma}
	\begin{proof}
		See Appendix~\ref{app_sec:SSMComp_2ndorder}.
	\end{proof}
	In system~\eqref{eq:HomNonAut_Second}, the matrix $\bm{D}_{\bm{m},\bm\kappa}$ depends on the eigenvalues and eigenvectors of the master subspace as well as on the forcing frequency $\bm \Omega$. The vector $\bm{C}_{\bm{m}, \bm \kappa}$ contains contributions from lower-order terms in the parametrization of the SSM and its reduced dynamics. The invariance system~\eqref{eq:HomNonAut_Second} contains decoupled equations for each pair of multi-indices and harmonics $(\bm{m}, \bm \kappa)$. Thus, for $z_m$ distinct multi-indices and $z_K$ distinct harmonics, we must solve $z_m z_K$ systems of $n-$dimensional linear equations. 
	
	
	\begin{remark}
		In case of near-resonances of the form
		\begin{align} \label{eq:ResCond_nonaut}
			\lambda_i \approx \bm \Lambda \cdot \bm{m} + i \bm \Omega \cdot \bm \kappa 
		\end{align}
		the coefficient matrix $\bm L_{\bm{m},\bm \kappa}$ will be nearly singular~\cite{Jain2021}.
		If the resonance condition \eqref{eq:ResCond_aut} is fulfilled for any eigenvalue $\lambda_i\in\text{Spec} (\mathcal{E})$ then a nonzero choice for $\bm{S}_{\bm{m},\bm\kappa}$ is necessary to solve the linear system~\eqref{eq:Hom_Aut_Second}, where different choices lead to different styles of parametrization \cite{Haro2016,Jain2021}.
	\end{remark}
	As an example of the resonance condition~\eqref{eq:ResCond_nonaut}, consider a two-dimensional periodic SSM. At the leading order, the resonance condition then reduces to $\lambda \approx i\Omega$, which is a simple resonance created by the external periodic excitation. Higher-order terms lead to more complicated resonances which are able to capture the effects of parametric resonance. For instance, the a principal parametric resonance of the form $\Omega \approx 2 \text{Im}(\lambda)$ results in
	\begin{align}
		\lambda \approx \conj{\lambda} + i\Omega
	\end{align}
	in terms of relationship~\eqref{eq:ResCond_nonaut}.
	Finally, once we have computed the coefficients $\bm x_{\bm m,\kappa}$ for the displacement parametrization $\bm{x(p,\phi)}$, the coefficients for the nonautonomous velocity parametrization can be obtained directly using the following statement (cf.~Lemma~\ref{lem:HomAutVel}).
	
	\begin{lemma} \label{Lem:HomNonautVel}
		The coefficients of the nonautonomous velocity parametrization $\dot{\bm{x}}_{\bm{m},\bm\kappa }$ for a multi-index $\bm{m}$ and harmonic $\bm{\kappa}$ can be computed directly from the displacement parametrization as 
		
		\begin{align}
			\dot{\bm{x}}_{\bm{m},\bm\kappa } = [\text{D}\bm w(\bm p) \cdot \bm S +  \partial_{\bm p}\bm x \cdot \bm R + \partial_{\bm \phi}\bm x\cdot \bm \Omega ]_{\bm m,\bm \kappa}
		\end{align}
	\end{lemma}
	
	\begin{proof}
		See Appendix~\ref{app_sec:SSMComp_2ndorder}
	\end{proof}
	
	Prior works~\cite{Vizzaccaro2022,Opreni2022} identified a special structure in the invariance equations of second-order mechanical systems with proportional damping and up to cubic-order geometric nonlinearities. This special structure led to reduced computational costs by halving the number of invariance equations relative to a first-order setting. We have shown that such a reduction in the number of equations arises from the underlying structure~\eqref{eq:SSM_param_struct} of the phase space variables in any second-order system. Indeed, according to Lemmas~\ref{lem:HomAut}-\ref{Lem:HomNonautVel}, these computational advantages apply to general second-order systems with nonlinearities of arbitrary polynomial order, including nonlinear damping.
	\section{Stability Diagrams and SSM}\label{sec:SDSSM}
	\subsection{Mathematical Foundation}\label{sec:MathSD}

	In this section, we consider mechanical systems subject to periodic parametric excitation, which means that $\bm{f}^{ext} = \bm{0}$ in eq.~\eqref{eq:forcing}. For such systems, $\bm{z=0}$ remains a fixed point for system~\eqref{eq:DS_FirstOrder} even for $\epsilon>0$. This fixed point may, however, be destabilized by resonances between the eigenfrequencies of the linear system and the parametric excitation frequency. 
	
	In order to study the stability of this fixed point for $\epsilon>0$, the dynamical system \eqref{eq:DS_FirstOrder} can be extended to an autonomous system of variables $(\bm{z}, \tau) \in \mathbb{R}^N \times S^1$ such that the trivial fixed point can be interpreted as the periodic orbit $(\bm{z}, \tau ) = (\bm{0}, t \ \text{mod} \ 2\pi )$\cite{Champneys2013}. Changes in the stability of this periodic orbit result in bifurcations with respect to parameters. Tools such as the \texttt{po} toolbox of the continuation package \textsc{coco}~\cite{Dankowicz2013} can detect and continue families of such bifurcations using Floquet theory. We refer to Dankowicz and Schilder~\cite{Dankowicz2013}) for a detailed account. Numerical continuation for finding an initial bifurcation that characterizes the stability boundary may be carried out by varying either the forcing frequency or the amplitude, as shown in Figure~\ref{fig:StabDiag}.
	\begin{figure}[ht]
		\centering
		\includegraphics[width=0.4\textwidth]{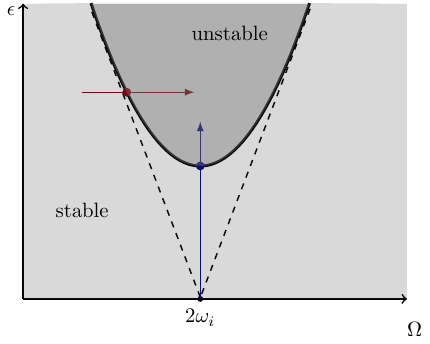}
		\caption{Sketch of a stability diagram at the principal resonance. The thin dashed line corresponds to the boundary of stable and unstable regions in the conservative limit. When damping is added, the tongue starts to lift up, bounded by the thick black line. The numerical continuation in the variables $\epsilon$ (blue arrow) and $\Omega$ (red arrow), respectively, leads to the detection of a bifurcation indicated by the blue and red dots. Choosing these bifurcations as a starting point, the family of bifurcations on the boundary can then be found via further continuation.}
		\label{fig:StabDiag}
	\end{figure}
	We employ SSM computation in order to efficiently obtain the \emph{resonance tongue} of the principal resonance, which marks the stability boundary in Figure~\ref{fig:StabDiag}. The SSM is constructed over the spectral subspace corresponding to the resonant eigenfrequency $\omega_i$. Consequently, the $M$-dimensional ROM  
	\begin{align}
		\dot{\bm{p}} &= \bm{R}_\epsilon(\bm{p}, \Omega t)
		\\ \label{eq:red_dyn_expansion}
		&=
		\sum_{\bm{m}\in \mathbb{N}^M} \bm{R}_{\bm{m}} \bm{p}^{\bm m}
		+
		\epsilon
		\sum_{\substack{\bm{m}\in \mathbb{N}^M \\ m \geq 1}} 
		\sum_{\kappa \in \mathbb{Z}} \bm{S}_{\bm{m},\kappa } e^{i \kappa \Omega t }\bm{p}^{\bm m}
	\end{align}
	needs to be analyzed for bifurcations using the \texttt{po} toolbox of \textsc{coco}\cite{Dankowicz2013}. Details on the corresponding continuation problem are recounted in Appenidx \ref{App:COCO}. As the forcing frequency changes, the ROM has to be updated, which highlights the need for a fast computational algorithm to calculate the SSM. The reduced dynamics are of low dimension, so continuation is a suitable method, as it is fast on small systems and provides key insights into the bifurcation behavior. The proposed approach thus combines the advantages of reduced-order modeling and numerical continuation. This procedure is now automated and available in SSMTool 2.4, which we use to analyze the following examples. Unless stated otherwise, the computations are performed using Matlab 2021a on a Macbook Air with a 1,4 GHz Intel Core i5 processor and 4GB of total RAM.
	
	\begin{remark}
		Other resonances with nonmaster modes might also occur close to this principal resonance and lead to resonance tongues that interact with the principal resonance tongue. Naturally, these other resonances are not detected by the ROM as it is ignorant of the spectrum of the other modes. Such near-resonance relationships with the master mode are called internal resonances. For such internally resonant systems, the existence and uniqueness of the underlying SSM are guaranteed if these resonant modes are also included in the master subspace, leading to a higher-dimensional SSM. We refer to Li et al. \cite{MingwuPart1, MingwuPart2} for the details of SSM theory and computation for internally resonant systems.
	\end{remark}
	
	\subsection{Examples}
	\label{sec:ExamplesSD}
	\subsubsection{Mathieu equation}\label{sec:ExamplesSD_Mathieu}
	The most frequently discussed system in the parametric excitation literature is the Mathieu equation
	\begin{align}\label{eq:MathieuEquation}
		m \Ddot{y} + c \Dot{y} + \omega_0^2(1 + \epsilon \cos{ \Omega t})y = 0, \quad 0\ll\epsilon\ll 1,    
	\end{align}
	whose stability is well understood. Various studies have appeared for the damped and undamped cases\cite{Verhulst2009}, some with additional nonlinearities to system~\cite{Rhoads2006_MEMS} and others in higher dimensions \cite{Hansen1985}. Here, we extend the system~\eqref{eq:MathieuEquation} to a set of damped oscillators coupled via springs with parametrically varying stiffness. Adding stiffness nonlinearities then results in the mechanical system displayed in FIG.~\ref{fig:CoupledMathieuModel}. 
	\begin{figure}[htbp]
		\centering
		\includegraphics[width=0.48\textwidth]{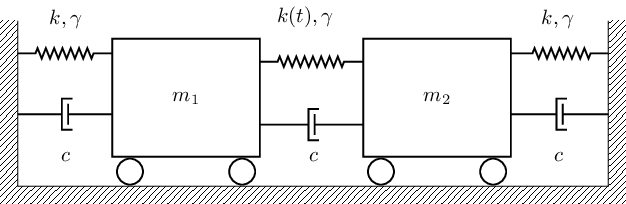}
		\caption[]{ Schematic model of two oscillators described by two coupled Mathieu equations. The coupling springs have time-varying stiffness coefficients $k(t) = k(1+ \epsilon \cos (\Omega t))$. The induced parametric excitation may destabilise the trivial fixed point if chosen in subharmonic resonance with an eigenfrequency of the linearized dynamical system.}
		\label{fig:CoupledMathieuModel}
	\end{figure}
	The resulting generalised Mathieu equation is given by
	\begin{align}\label{eq:ndMathieu}
		\bm{M} \bm{\ddot{y}} + \bm{C} \bm{\dot{y}}  + \bigg( \bm{K} + \epsilon \bm{Q}  \cos(\Omega t) \bigg)   \bm{y}  +\bm{f}(\bm{y}) =\bm{0}.
	\end{align}
	In the case considered here, these matrices and the nonlinearity are explicitly given as
	\begin{align}
		\small
		\bm{M}&=\left[\begin{array}{cc}m_1 & 0 \\ 0 & m_2 \end{array}\right],~
		\bm{C}= \left[\begin{array}{cc} 2c & -c \\ -c & 2c \end{array}\right],~
		\bm{K}=\left[\begin{array}{cc}2k & -k \\ -k & 2k \end{array}\right],
		\\
		\bm{Q}&= \left[\begin{array}{cc} k & -k \\ -k &  k \end{array}\right],~
		\bm{f}({y_1, y_2})= \kappa \left[\begin{array}{cc} -y_1^3 -(y_1-y_2)^3 \\  -y_2^3 +(y_1-y_2)^3   \end{array}\right].
	\end{align}
	We choose non-dimensional parameters $m_i =1, k=1, c = 0.05$ and $\kappa = 0.1$, for which we obtain the eigenvalues of the $\epsilon=0$ system in the form $\lambda_{1},\conj{\lambda_1} \approx -0.0250 \pm 0.9997i$ and $\lambda_{2},\conj{\lambda_2} \approx  -0.0500 \pm 1.7304i$. We seek to analyze the nonlinear response under the principal parametric resonance $\Omega \approx 2\omega_2$ by constructing the non-autonomous SSM corresponding to the spectral subspace $\mathcal{E} = \text{span}(\bm{v}_2,\conj{\bm{v}_2})$. Using numerical continuation in the reduced dynamics, we will detect the corresponding resonance tongue. 
	
	First, we choose an initial parameter set $(\epsilon_i, \Omega_i )$, then either of the two parameters is released, and the trivial $2\pi/\Omega_i$-periodic response is continued in search of period-doubling bifurcations. The detected bifurcation point is then continued in the parameter space, yielding the stability diagram in FIG.~\ref{fig:StabDiag_2DMathieu_SD}. The entire procedure is automated, constantly updating the coefficients of the nonautonomous reduced dynamics when the excitation frequency gets changed.
	\begin{figure}[t]
		\centering
		\subfloat[]{
			\includegraphics[width=0.4\textwidth]{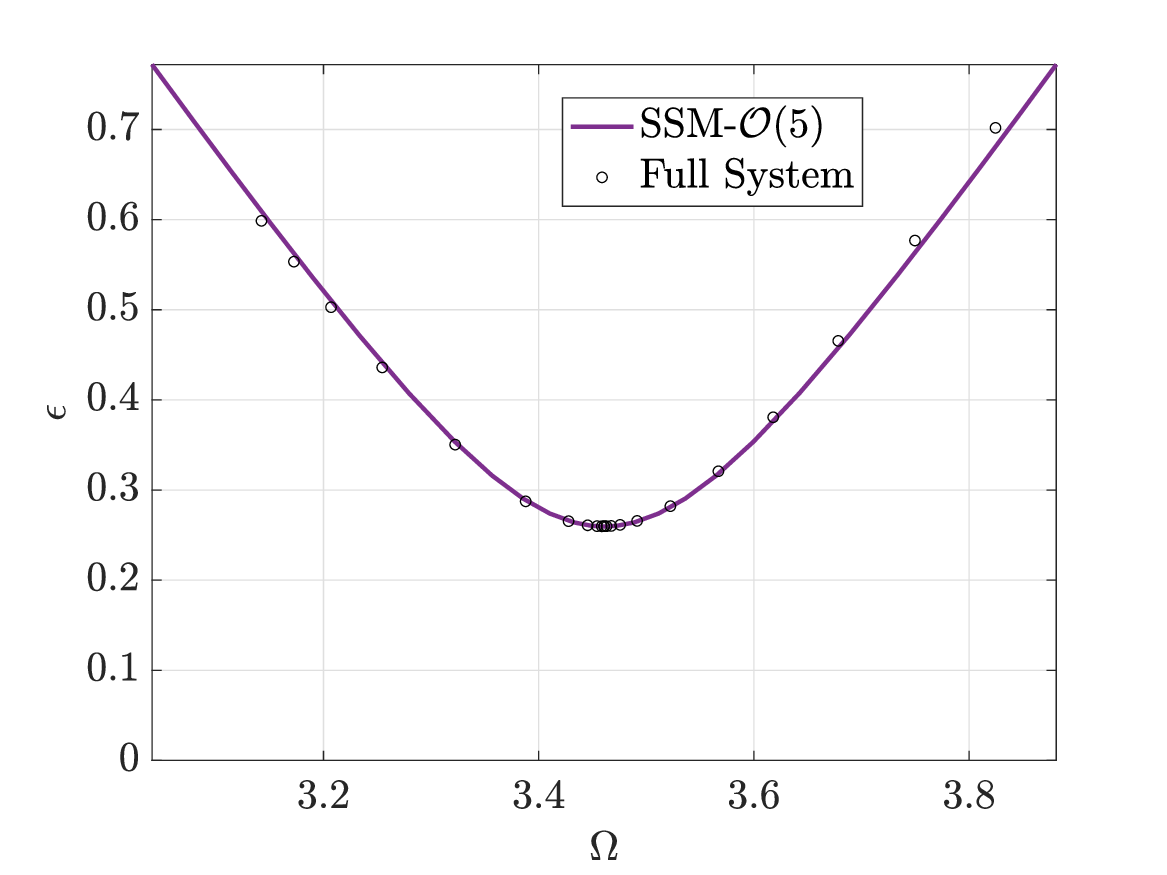}
			\label{fig:StabDiag_2DMathieu_SD}
		}
		\hfill
		\subfloat[]{
			\includegraphics[width=0.4\textwidth]{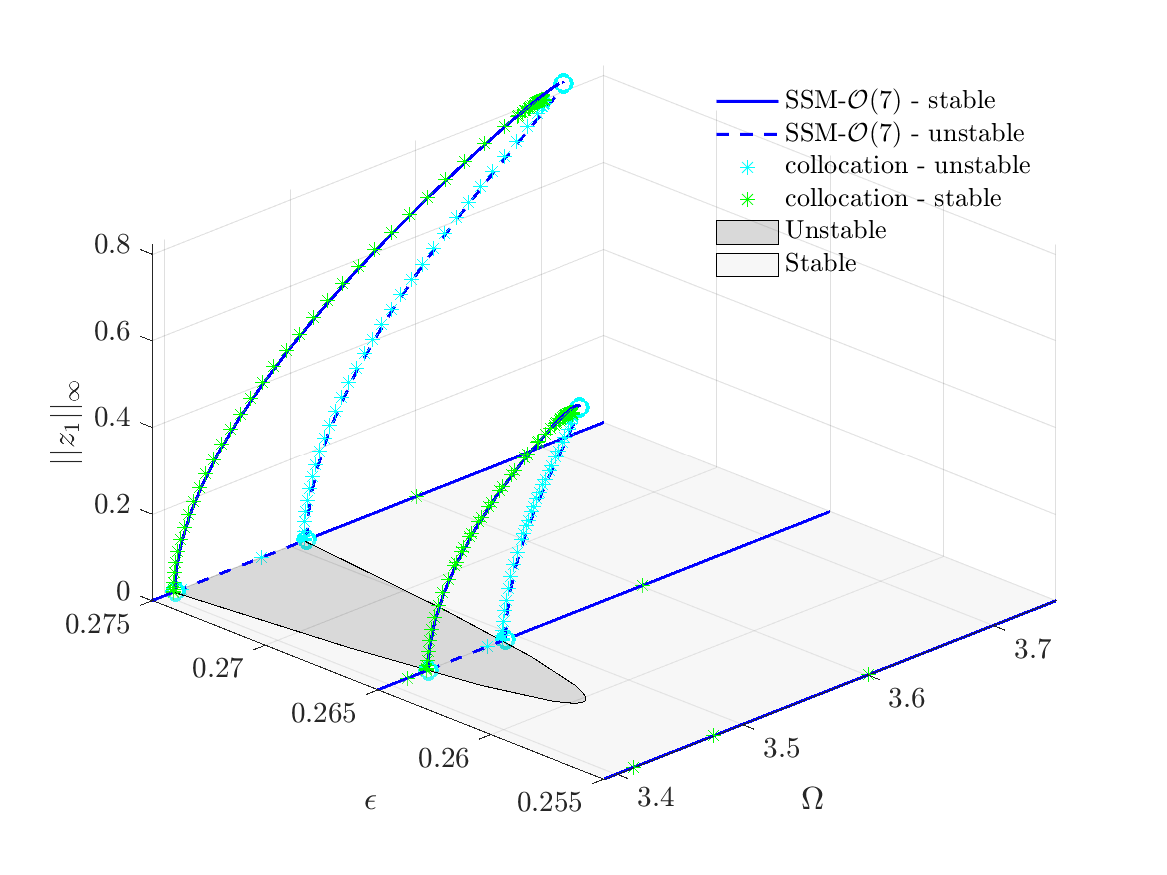}
			\label{fig:StabDiag_2DMathieu_FRC}
		}
		\caption[]{Stability diagram and forced response of the two coupled Mathieu oscillators shown in FIG.~\ref{fig:CoupledMathieuModel}. (a) Stability diagram of the second mode obtained using an SSM-based ROM, which locally approximates the resonance tongue in a region around the principal resonance where $\Omega \approx 2\omega_2$. (b) The $||z_1||_\infty = 0$ plane corresponds to the trivial response of the full system. On this plane, the resonance tongue is displayed in gray shades. From its boundary, orbits with period $4 \pi / \Omega$ emerge that show hardening behavior due to the nonlinearity. }
		\label{fig:StabDiag_2DMathieu}
	\end{figure}
	To obtain a more complete picture of how these regions of stability influence the system's dynamics, we consider the non-trivial forced response of system \eqref{eq:ndMathieu} which emanates from the instability. The orbits with period $T = \frac{4\pi}{ \Omega} \approx \frac{2 \pi}{\omega_0}$ bifurcate out of the trivial response at the stability boundary. The amplitude of these response curves increases smoothly with increasing excitation amplitude $\epsilon$.  The reduced dynamics on the SSM serves as a ROM to explore periodic orbit families via numerical continuation. The resulting stability diagram, as well as the forced response, can be seen in FIG. \ref{fig:StabDiag_2DMathieu_FRC}. The resonance tongue and the forced response are accurately reproduced using the ROM, as we verify using the \texttt{po}-toolbox of \textsc{coco} on the full dynamical system. The parameters used for continuation of the full and reduced models are noted in Appendix~\ref{App:COCO}. 
	
	The computation time for the stability diagram using the full system was 22 seconds. The same computation using the ROM, including the continual update of the SSM parametrization, takes 1 minute and 39 seconds. Similarly, for the FRCs displayed in FIG.~\ref{fig:StabDiag_2DMathieu_FRC} the full system analysis took 1 minute and 57 seconds, while using the ROM amounted to 3 minutes and 43 seconds. Therefore, because the full dynamical system is low-dimensional, the advantage of using a ROM is outweighed by the computational cost of constructing the reduced dynamics on the time-dependent SSM for each forcing frequency.
	\subsubsection{Bernoulli beam}\label{sec:ExamplesSD_BB}
	\begin{figure}[t]
		\centering
		\subfloat[]{
			\includegraphics[width=0.5\textwidth]{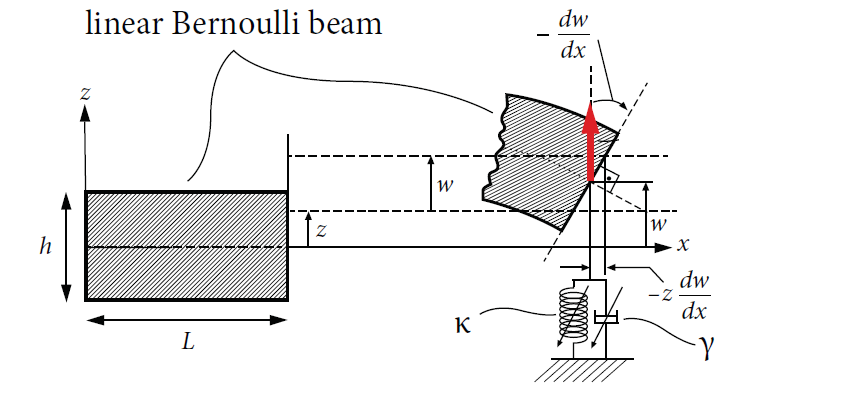}
		}
		\hfill
		\subfloat[]{
			\includegraphics[width=0.5\textwidth]{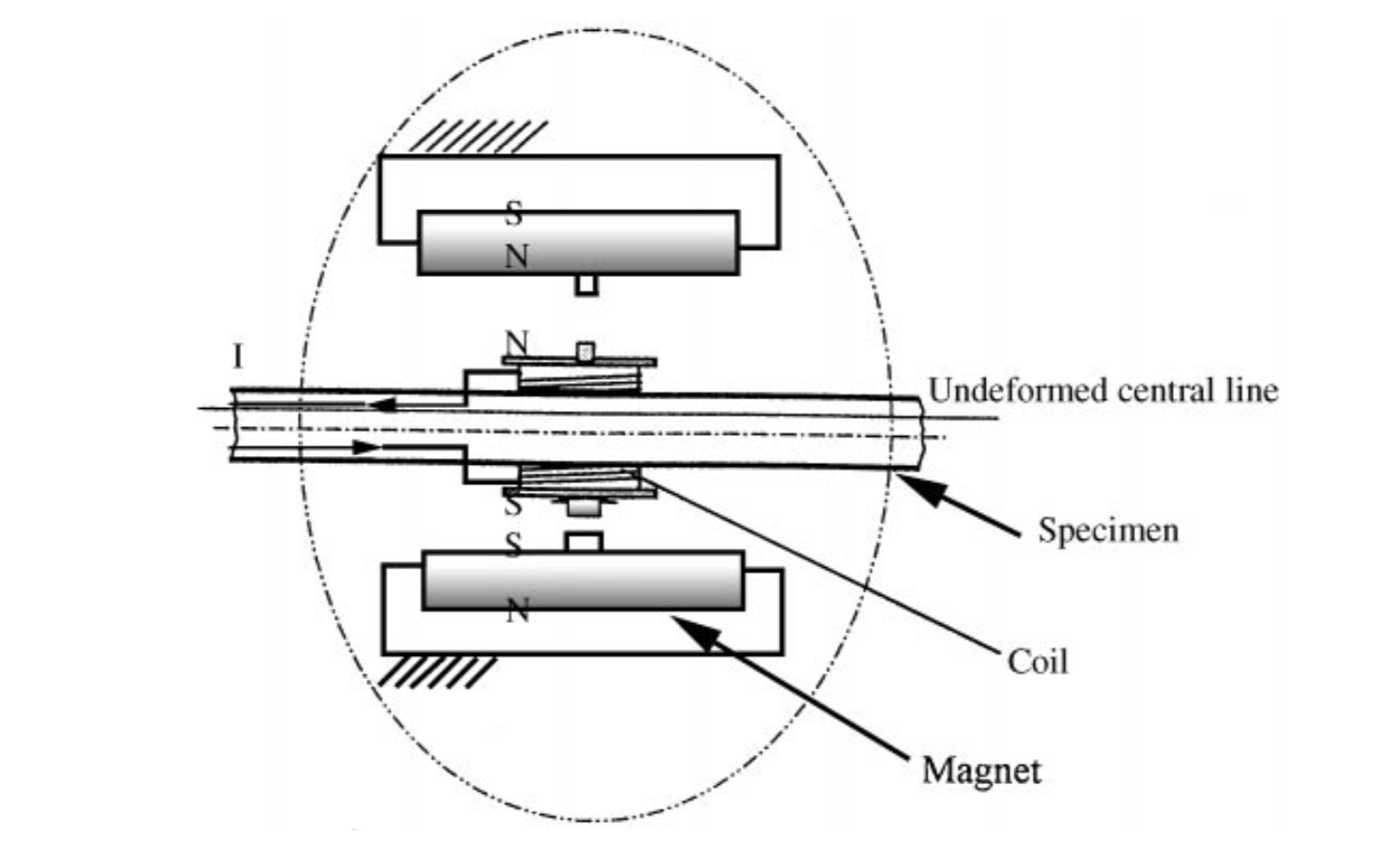}
		}
		\caption{(a) Model of a Bernoulli beam of length $L$ and height $h$\cite{isolatedresponsesten}. The beam is aligned with the $x$ axis, cantilevered on one side, and parametrically excited on the other, as indicated by the red arrow. A nonlinear (cubic) damper and a cubic spring are attached to the free end of the beam. (b) Schematic depiction of a possible experimental setup for parametrically exciting the beam. Two coils are used as tunable magnets via the induced current $I$. Depending on its magnitude and direction, the effective attraction and repulsion between the beam and the adjacent magnets can be controlled. This results in a time varying stiffness~\cite{Chen2001}.}
		\label{fig:BB_model}
	\end{figure}
	We now consider the nonlinear Bernoulli beam shown in FIG.~\ref{fig:BB_model} and treated in Ref.~\onlinecite{isolatedresponsesten}. It is cantilevered on one side and subject to a cubic spring, a cubic damper, and parametric excitation at its free end. A similar example has been experimentally studied by Chen and Yeh\cite{Chen2001}, who implemented parametric excitation through a pair of coils fixed to the beam. Nearby magnets mounted on a support were then used to exert a force on these coils when a current is applied.
	
	We discretize the beam into 5 finite elements, resulting in $n=10$ degrees of freedom. Damping is assumed to be proportional and of the form $\bm{C} = \sigma ( \alpha \bm{M} + \beta \bm{K})$. Here, the cubic nonlinearity $\bm{f}(\bm{y},\dot{\bm{y}})$ includes the nonlinear damping coefficient $\gamma$ and the nonlinear stiffness $\kappa$. The parameters and beam properties are set according to TABLE \ref{tab:BB_model}, and the parametric excitation amplitude is set to 1 Newton. This system is driven in the regime of parametric resonance by targeting the principal resonance of the mode associated with the leading pair of eigenvalues
	\begin{align}
		\lambda_{1},\conj{\lambda_1} \approx -0.1238 \pm 6.9995 i
	\end{align}
	The prinicipal resonance is thus obtained for forcing frequency values $\Omega \approx  13.999 \ \text{rad}/s$. We use the reduced dynamics on the corresponding SSM to construct the stability diagram of the principal resonance of the selected mode. The resulting resonance tongue is shown in FIG. \ref{fig:BB_SD} for various values of the damping coefficient $\sigma$. The forced response obtained from SSM theory for a set of parametric excitation amplitudes with SSM theory is shown in FIG.\ref{fig:BB_FRC}. We verified these reduced results against the full system using \textsc{coco}, for which the relevant simulation parameters can be found in Appendix \ref{App:COCO}. The computation times for the stability diagram and forced responses are reported in TABLE \ref{tab:CompTimeBB_SD} and TABLE \ref{tab:CompTimeBB_FRC}, which show that the computational savings arising from the use of SSM-based ROM in this higher-dimensional example are significant, unlike in our first example.
	\begin{table}
		\parbox{\linewidth}{
			\centering
			\caption{Physical parameters used for the model of the Bernoulli beam.}
			\begin{tabular}{c c c } 
				\hline
				Symbol & Meaning & Value \\ [0.5ex] 
				\hline\hline
				$L$ & Length of beam & 2.7 (m) \\
				$h$ & Height of beam & 10 (mm)\\
				$b$ & Width of beam  & 10 (mm) \\
				$E$ & Young's modulus & 45 $\times 10^6$ (kPa) \\
				$I$ & Area moment of inertia & 833.3 (mm$^4$) \\
				$\alpha$ & Structural damping parameter & 1.25 $\times$ $10^{-4}$ (s$^{-1}$) \\ 
				$\beta$ & Structural damping parameter & 2.5 $\times$ $10^{-4}$ (s) \\ 
				$\kappa$ & Cubic stiffness coefficient & 50 (N/m$^3$) \\ 
				$\gamma$ & Cubic damping coefficient & 0.01 (Ns/m$^3$) \\ 
				\hline
				\label{tab:BB_model}
			\end{tabular}
		}
	\end{table}
	%
	%
	
	\begin{table}
		\centering
		\caption{Computation time for for the stability diagrams in FIG. \ref{fig:BB_SD} with $n=10$ DOFs in the format hh:mm:ss.}
		\begin{tabular}{c c c} 
			\hline
			Damping ($\sigma$) & SSM-based ROM & {Full system} \\ [0.5ex] 
			\hline\hline
			10 & 00:01:10 & 00:02:07 \\
			20 & 00:00:35 & 00:01:42 \\
			30 & 00:00:30 & 00:01:31 \\
			40 & 00:00:37 & 00:01:21 \\
			\hline
			\label{tab:CompTimeBB_SD}
		\end{tabular}
	\end{table}
	\begin{table}
		\centering
		\caption{Computation time for the FRCs in FIG. \ref{fig:BB_FRC} with $n=20$ DOFs in format hh:mm:ss.}
		\begin{tabular}{c c c} 
			\hline
			Damping ($\sigma$) & SSM-based ROM & Full system \\ [0.5ex] 
			\hline\hline
			10 & 00:14:58 & 06:12:31 \\
			\hline
			\label{tab:CompTimeBB_FRC}
		\end{tabular}
	\end{table}
	\begin{figure}
		\centering
		\subfloat[]{
			\includegraphics[width=0.4\textwidth]{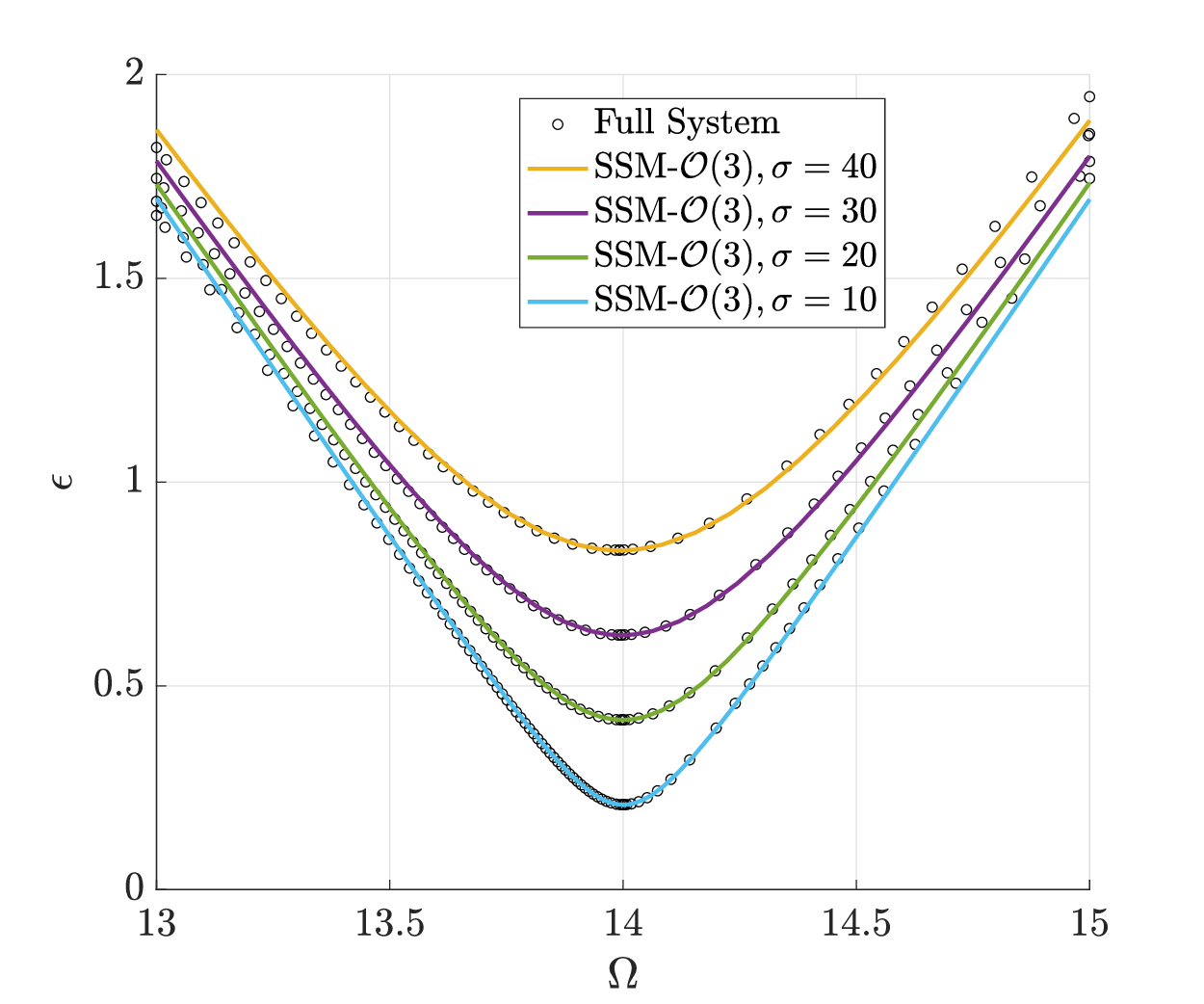}
			\label{fig:BB_SD}
		}
		\hfill
		\subfloat[]{
			\includegraphics[width=0.4\textwidth]{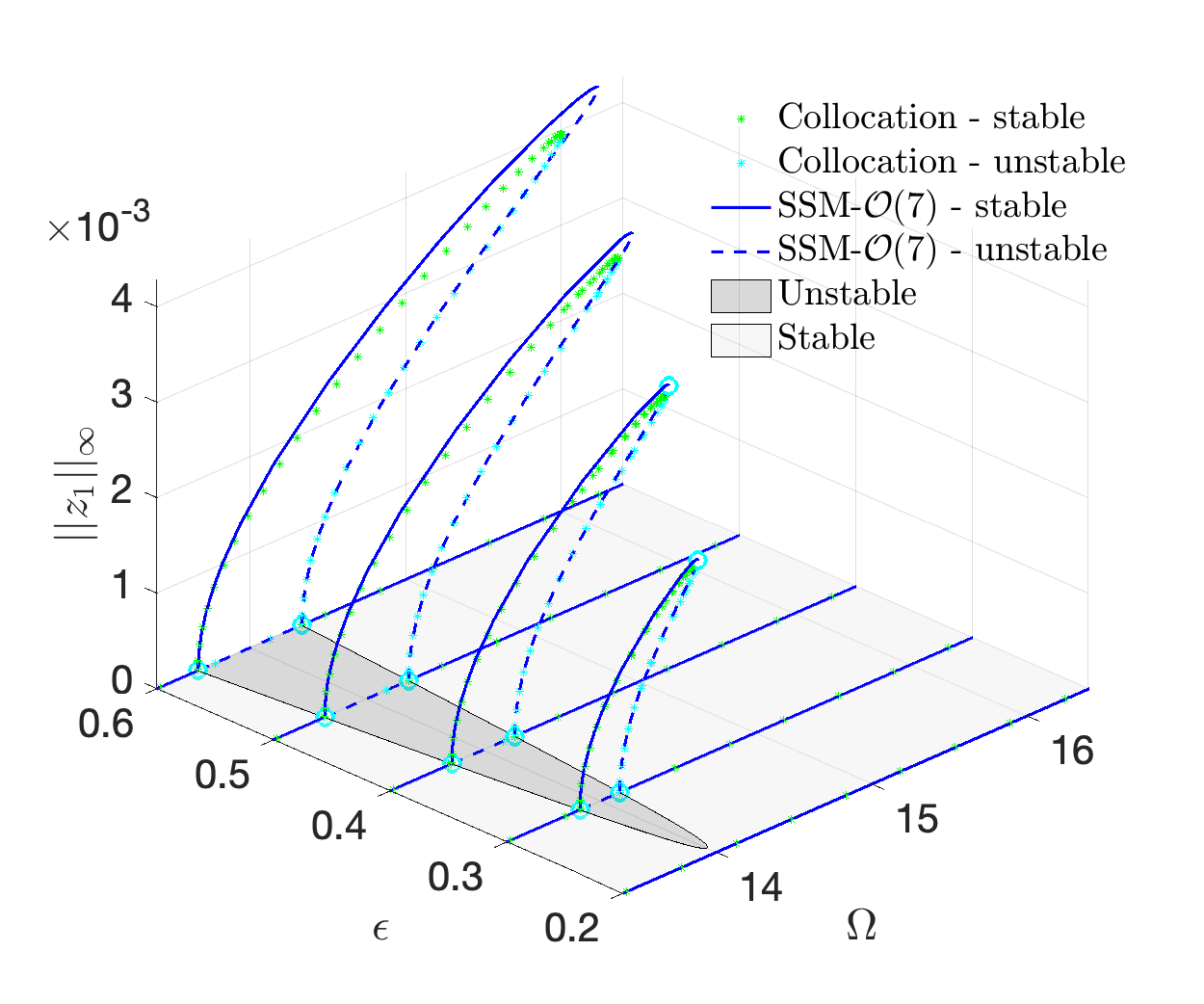}
			\label{fig:BB_FRC}
		}
		\caption[]{ (a) Stability diagram for the principal resonance of the first mode of the Bernoulli beam for 5 discrete elements with $n=10$ degrees of freedom. The solid line denotes the border of the resonance tongues as obtained using the dynamics of the full system and has been obtained using \texttt{po}. The reduced model of these dynamics provided by a cubic truncation of the SSM accurately locates the position and shape of the tongue for different values of the damping parameter $\sigma$. (b) Forced response bifurcating out of the trivial response at the boundary of the stability diagram for $n=20$ DOFs. The amplitude of the response curve increases with growing parametric excitation amplitude $\epsilon$.}
		\label{fig:BB}
	\end{figure}
	\subsubsection{Prismatic Beam}\label{sec:ExamplesSD_PB}
	Next we analyze the dynamics of a hinged-clamped prismatic beam. Extending the initial treatment of external excitation by Nayfeh et al. \cite{Nayfeh1974NonlinearElements}, Li et al~\cite{MingwuPart1} have recently applied SSM-theory to the case of a 1:3 internal resonance between the first two bending modes of this beam model. Their implementation serves as the basis of this example but we choose the nondimensionalized length in terms of the characteristic length here as $l= 1.7$ to avoid internal resonances. Introducing nondimensional parameters and coordinates, Nayfeh et al.\cite{Nayfeh1974NonlinearElements} derive nonlinear partial equations for the transverse displacement $w(y,t)$ of the beam. Under axial loading at the hinged tip, these equations must be modified to \cite{Barari2011NonlinearBeams}
	\begin{align}\label{eq:PDEPBParam}
		\frac{\partial^4 w}{\partial y^4} + \frac{\partial^2 w}{\partial t^2}
		=  \epsilon \left(H \frac{\partial^2w}{\partial y^2}  - p_a(t) \frac{\partial^2w}{\partial y^2}  \right),
	\end{align}
	where $p_a(t)$ describes the axial loading, $H$ accounts for the axial stretching forces and depends on $w$, and $c$ is the distributed damping parameter. The axial excitation effectively leads to a parametric excitation on the transverse displacements of the beam, as is evident in eq.~\eqref{eq:PDEPBParam}. The beam is clamped at $x = l$ and hinged at $y=0$, resulting in the boundary conditions
	\begin{align}
		w(0) = w(l) = 0, \quad
		w'(l)  = 0,\quad
		w''(0) = 0. 
	\end{align}
	After writing the displacement $w(y,t)$ as a modal expansion, we obtain a system of second-order ODEs for the evolution of the modal amplitudes $z_j(t)$.
	\begin{figure}[t]
		\centering
		\subfloat[]{
			\includegraphics[width=0.4\textwidth]{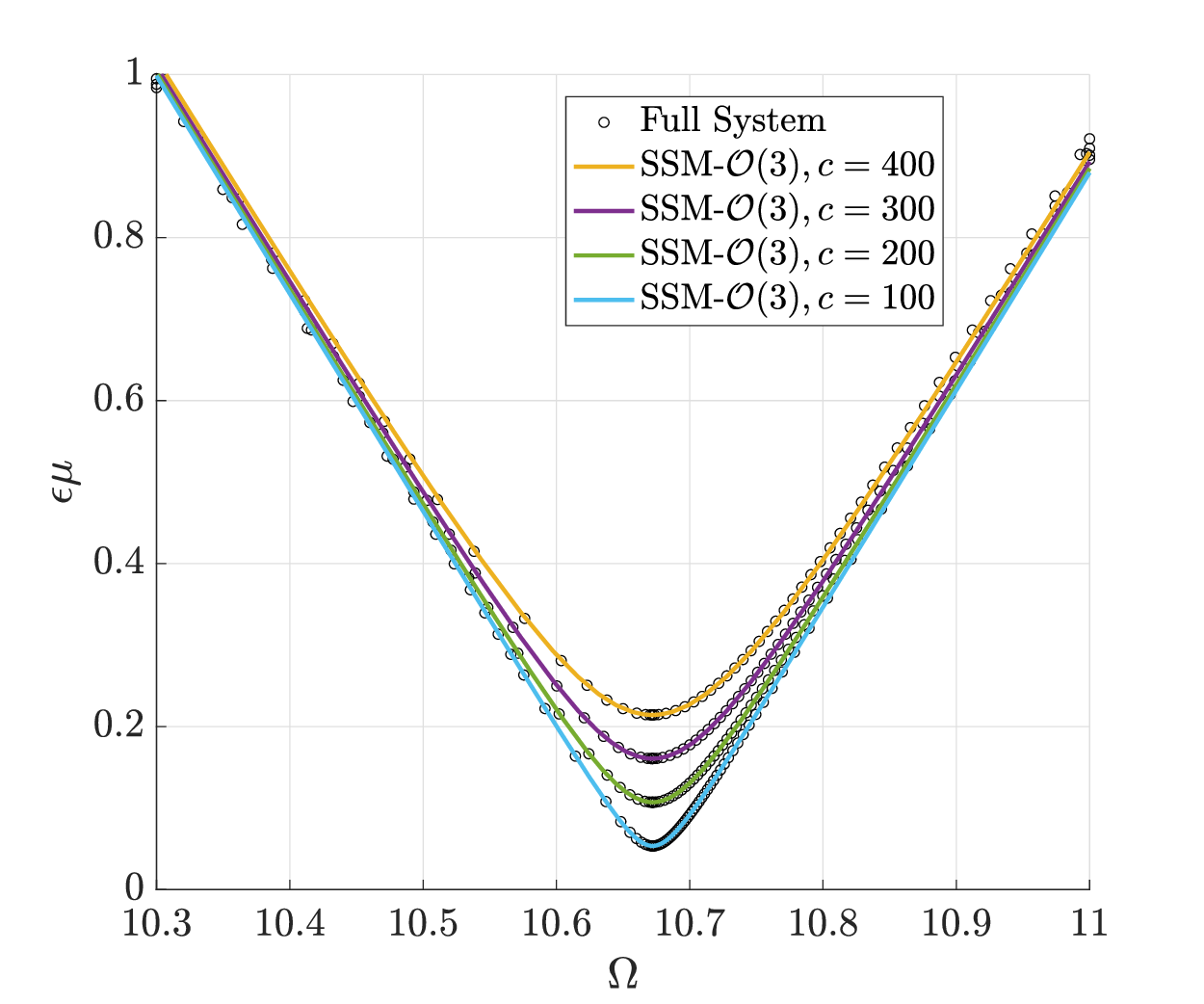}
			\label{fig:PB_SD}
		}
		\hfill
		\subfloat[]{
			\includegraphics[width=0.4\textwidth]{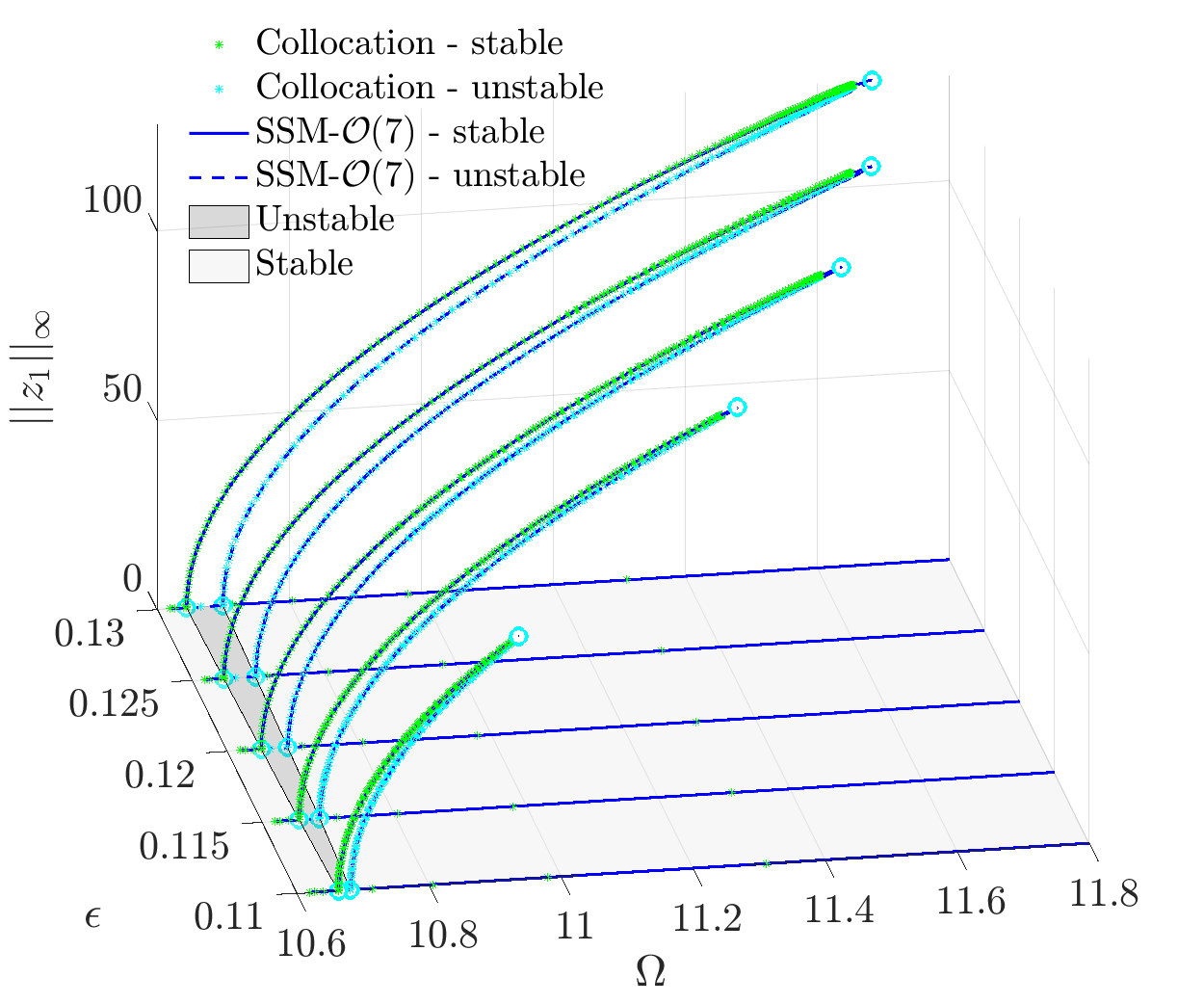}
			\label{fig:PB_FRC}
		}
		\caption[]{ (a) Stability diagram for the principal resonance of the first mode of the prismatic beam for a truncation with $n=10$ modal coordinates. The full system response has been obtained with the use of collocation, specifically \texttt{po}. The reduced model of these full dynamics given by the dynamics on the SSM can accurately predict the stability behaviour of the full system. (b) Nontrivial period $T = 4\pi / \Omega$ forced response for a number of excitation amplitudes for a fixed damping coefficient $c=200$. The ROM on the SSM accurately reproduces the full system response.}
		\label{fig:PB}
	\end{figure}
	\begin{align}
		\Ddot{z}_j +  2 \epsilon c \dot{z}_j + \omega_j^2 z_j = \sum_{i,k,s} \epsilon \alpha_{jiks} z_i z_k z_s + \epsilon \sum_i p_a(t) z_i a_{ji},  
	\end{align}
	with $j = 1,\dots,n$. A detailed derivation of these equations using Galerkin projection along with the definition of the coefficients $\alpha_{jiks}, a_{ji}$ are given in Appendix~\ref{App:PrismaticBeam}. For this particular system, the nonlinear internal forces appear premultiplied by $\epsilon$. Note, however, that this not a requirement for SSM-theory to be applicable. The shape of the SSMs depends smoothly on the magnitude of the nonlinearities, they thus also allow for the analysis of systems that include large nonlinearities. We take the axial forcing to be harmonic and of the form
	\begin{align}
		p_a(t) = \mu \cos (\Omega t).
	\end{align}
	Furthermore, we set $\epsilon = 1 \times 10^{-4}$. The modal expansion is truncated at $n=10$ modes which gives rise to a phase space dimensionality of $N=20$. 
	
	We wish to analyze the stability of the beam when axially excited at the principal resonance of its fundamental frequency. We construct the SSM over the first pair of modes with eigenvalues
	\begin{align}
		\lambda_1, \conj{\lambda_1} \approx -0.0200 \pm 5.3361i.
	\end{align}
	The instability of the trivial fixed point induced by the parametric resonance occurs near the excitation frequency values $\Omega \approx 10.6672  \ [rad/s]$. FIG. \ref{fig:PB} shows the corresponding stability diagram and the steady-state response computed using the SSM-based ROM. The full system's stability is again verified using the \texttt{po}-toolbox of \textsc{coco}. The parameters chosen for numerical continuation of the full and reduced system are documented in Appendix~\ref{App:COCO}.
	
	TABLE~\ref{tab:CompTimePB} shows the computation times for obtaining the FRCs as well as the stability diagram for a series of damping parameters. Again, the use of an SSM-based ROM brings a significant computational advantage over the full model. 
	\begin{table}
		\centering
		\caption{Computation time for the prismatic beam example in hh:mm:ss.}
		\begin{tabular}{c c c c c} 
			\hline
			FIG.& DOFs (n) & Damping ($c$) & SSM-based ROM  & Full system \\ [0.5ex] 
			\hline\hline
			
			\ref{fig:PB_SD} & 10 & 100 & 00:01:11 & 00:05:41 \\
			&   & 200 & 00:00:42 & 00:04:39 \\
			&   & 300 & 00:00:38 & 00:04:19 \\
			&   & 400 & 00:00:36 & 00:04:02 \\
			\hline
			\ref{fig:PB_FRC} & 10 & 200 & 02:21:09 & 21:03:03 \\
			\hline
			\label{tab:CompTimePB}
		\end{tabular}
	\end{table}
	\section{Forced response and Spectral submanifolds}
	\label{sec:FRCSSM}
	\subsection{Mathematical Foundation}\label{sec:MathFRC}
	For externally excited systems, SSMs have proven highly effective in capturing the forced response accurately at a low computational cost\cite{SSMTool, isolatedresponsesten, Ponsioen2020, Jain2021}. Next, we extend these results to obtain FRCs of systems subject to both external and parametric excitation. Our algorithm is completely parallelizable and can be carried out for each forcing frequency independently. We consider periodic external and parametric excitation and assume that their harmonic $\kappa_0$ is in 1:1 resonance with the master mode, i.e. 
	\begin{align} \label{eq:ResCond_FRC}
		\mathrm{i}\kappa_0 \Omega \approx \lambda_j
	\end{align}
	for some $j$. We further assume that no (near) internal resonances are present in the spectrum of the dynamical system. The periodic forced response and its stability can then be explicitly computed from the SSM-based ROM according to the following statement.
	\begin{lemma} \label{lem:FRC}
		Assume that the near-resonance condition~\eqref{eq:ResCond_FRC} holds for a mechanical system subject to periodic external and parametric excitation. We then obtain the following results for the dynamics on the two-dimensional SSM that serves as the non-autonomous, nonlinear continuation of the $j^{\mathrm{th}}$ linear mode:
		\begin{enumerate}[label=(\roman*),ref=(\roman*)]
			\item \label{lem:i_PolDyn} Polar dynamics:
			The reduced dynamics in polar coordinates $(\rho, \theta)$ is given by
			\begin{align} \label{eq:PolarDynamics}
				\begin{bmatrix}\dot{\rho} \\ \rho \dot{\psi}\end{bmatrix} &= \bm{r}(\rho,\psi,\Omega),\\   
				\dot{\phi} &=  \Omega,
			\end{align}
			where 
			\begin{align}
				\psi 				  =& \ \theta - \kappa_0 \phi,
				\\
				\nonumber
				\bm{r}(\rho,\psi,\Omega)  
				=& 	 
				\begin{bmatrix}a(\rho) \\ b(\rho)  \end{bmatrix} 
				+ \epsilon 
				\bigg( 
				\bm{Q}(\psi) 
				\begin{bmatrix} 
					\Real (S^1_{\bm{0},\kappa_0})  \\ \Imag (S^1_{\bm{0},\kappa_0})
				\end{bmatrix} 
				\\ 
				&+ 
				\sum_{ \substack{ \bm{m} \in \mathbb{N}^2 \\ \kappa\in \mathbb{Z}^* }} 
				\rho^{|\bm{m}|} \bm{Q}((\kappa/\kappa_0)\psi) 
				\begin{bmatrix} 
					\Real (S^1_{\bm{m},\kappa})  \\ \Imag(S^1_{\bm{m},\kappa})  
				\end{bmatrix} 
				\bigg),\\
				a(\rho)              =& \ \Real \bigg(  \sum_{\bm{m}} R^1_{\bm{m}} \rho^{|\bm{m}|} \bigg),
				\\
				b(\rho,\Omega)       =& \ \Imag \bigg(  \sum_{\bm{m}} R^1_{\bm{m}} \rho^{|\bm{m}|} \bigg) - \kappa_0 \rho \Omega,
				\\
				\bm{Q(\bullet)} :=& \ \begin{bmatrix} \cos \bm{(\bullet)}  &  \sin \bm{(\bullet)} \\  -\sin \bm{(\bullet)}    &\cos \bm{(\bullet)} \end{bmatrix} 
			\end{align}
			and the sum over $\kappa$ and $\bm{m}$ includes the harmonics and spatial multi-indices associated to nonzero reduced dynamics coefficients in a normal-form style of parametrization.
			\item \label{lem:ii_FRC} Forced response curve:
			The fixed points of the SSM-reduced dynamics lie on the forced response curve that satisfies
			\begin{equation}
				\bm{r}(\rho,\psi,\Omega) = \bm{0}.
			\end{equation}
			\item \label{lem:iii_FRC_stab} Stability of the forced response:
			For any hyperbolic fixed point $(\rho,\psi)$ of \eqref{eq:PolarDynamics}, the eigenvalues of the Jacobian 
			\begin{align}\nonumber
				J = 
				\begin{bmatrix} 
					\partial_\rho  a(\rho)    				& 0 \\
					\partial_\rho ( b(\rho)/ \rho)   & 0 
				\end{bmatrix} 
				+
				\epsilon 
				\bigg(
				\begin{bmatrix} 
					0 					&   d_{\bm{0},\kappa_0}(\psi) \\ 
					-\frac{d_{\bm{0},\kappa_0}(\psi) }{\rho^2} & -\frac{c_{\bm{0},\kappa_0}(\psi) }{\rho} \end{bmatrix} 
				\\	
				+
				\sum_{\substack{ \bm{m} \in \mathbb{N}^2 \\ \kappa\in \mathbb{Z}^* }}
				\begin{bmatrix} 
					c_{\bm{m},\kappa}(\psi)   |\bm{m}| \rho^{|\bm{m}|-1}      & (\kappa/\kappa_0) d_{\bm{m},\kappa}(\psi) \rho^{|\bm{m}|}
					\\  
					d_{\bm{m},\kappa}(\psi)   (|\bm{m}|-1) \rho^{|\bm{m}|-2 } & - (\kappa/\kappa_0 ) c_{\bm{m},\kappa}(\psi) \rho^{|\bm{m}|-1}
				\end{bmatrix} 
				\bigg)
			\end{align}
			determine the stability of the fixed point, where
			\begin{align}
				c_{\bm{m},\kappa}(\psi) &:= \Real (S^1_{\bm{m},\kappa})\cos((\kappa/\kappa_0) \psi)  \nonumber \\ & \ \ + \Imag (S^1_{\bm{m},\kappa})\sin((\kappa/\kappa_0) \psi),  
				\\
				d_{\bm{m},\kappa}(\psi) &:= -\Real (S^1_{\bm{m},\kappa})\sin((\kappa/\kappa_0) \psi) \nonumber \\ & \ \  + \Imag (S^1_{\bm{m},\kappa})\cos((\kappa/\kappa_0) \psi).  
			\end{align}
		\end{enumerate}
	\end{lemma}
	\begin{proof}
		See Appendix \ref{App:FRC_Lemma}
	\end{proof}
	Fixed points of the reduced dynamics in polar coordinates correspond to periodic orbits of the full system as the SSM is a time-periodic invariant manifold, i.e., $\bm{z(t)} = \bm{W}_{\epsilon}(\bm{p}(t),\phi)$. As statement (ii) of Lemma~\ref{lem:FRC} only requires the computation of zero level sets, the forced response can be computed independently for each forcing frequency $\{\Omega_1, ... \Omega_z \}$ in the interval of interest. Another consequence is that isolated regions of forced response, the so-called \textit{isolas}, can be computed at the same time. This is in stark contrast to continuation-based methods, which will only find such branches of response from initial conditions that are specified close to isolas. 
	
	We make use Lemma~\ref{lem:FRC} to analyze the forced response of dynamical systems that are subject to direct and parametric excitation simultaneously. In our examples, we will consider systems of the form 
	\begin{align}\label{eq:DS_Mathieu}
		\bm{M} \bm{\ddot{y}} + \bm{C} \bm{\dot{y}}  + \bigg( \bm{K} + \epsilon\mu \bm{Q}  \cos(2\Omega t) \bigg)   \bm{y}  +\bm{f}(\bm{y, \dot{y}}) =\epsilon \bm{g}(\bm{y,\dot{y}},\Omega t),
	\end{align}
	with the function $\bm{g}$ including nonlinear parametric excitation and external excitation. Solving the invariance equation \eqref{eq:InvEq} up to $\mathcal{O}(\epsilon)$ provides a parametrization of the nonautonomous SSM. The computation of these parametrizations and the extraction of the FRCs will be performed with SSMTool 2.4. For a given dynamical system of first or second order, the algorithm automatically solves the invariance equations \eqref{eq:InvEq_aut} and \eqref{eq:InvEq_nonaut} up to any desired order and computes the steady-states using Lemma~\ref{lem:FRC}. 
	
	
	

	\subsection{Examples}
	\label{sec:ExamplesFRC}

	\subsubsection{Self-excited parametric oscillator}\label{sec:FRC_ex_SelfExcite} 
	Self-excitation of mechanical systems is caused by negative linear damping that leads to a destabilization of the trivial fixed point of the linearized system. In contrast, nonlinear positive damping stabilizes the system for large oscillation amplitudes and forms a limit cycle, resulting in behavior described by a Hopf-type normal form~\cite{Lu2017}. Such self-excitations may arise, for instance, from dry friction, vibrations during machine cutting, or nonlinear feedback control in MEMS devices \cite{Chatterjee2007, Malas2015, Fu2011}. 
	
	The example we study here represents a 2-DOF oscillator with nonlinear elasticity and damping, as shown in FIG. \ref{fig:SelfExciteModel}. This system is subject to external excitation along one DOF, which is coupled to the second DOF via a spring with periodically varying stiffness (see, e.g., Warminski and Szabelski\cite{Szabelski1997, Warminski2DOF1995}). The three types of excitation seen in this example (external, parametric and self-excitation) are present in a number of engineered structures, including coupled vibrations of unbalanced shafts with time-varying rigidity, whirling in slide bearings, and metal machine cutting. To study the system dynamics, the cited references construct approximate analytical solutions by assuming a steady state in the form of a sum of trigonometric functions. Assuming that the oscillation amplitudes vary slowly in time and, therefore, dropping higher-order derivatives, the authors arrive at an expression for the amplitudes, which they solve numerically. 
	\begin{figure}[h]
		\centering
		\includegraphics[width=0.45\textwidth]{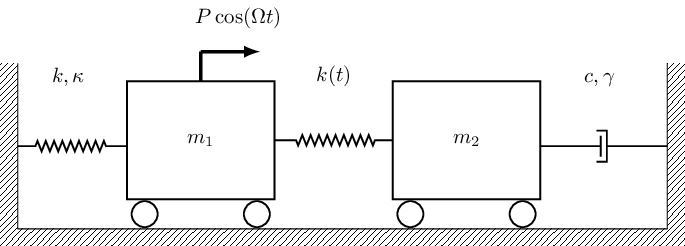}
		\caption{Model of a self-excited oscillator system with external and parametric excitation. The stiffness of the coupling spring varies periodically as $k(t) = 4k(1 + \mu \cos (2 \Omega t))$. The negative linear damping coefficient $c$ introduces self-excitation that is stabilized by the nonlinear damping with coefficient $\gamma$. The left spring is characterized by linear and nonlinear stiffness coefficients $k$ and $\kappa$. } \label{fig:SelfExciteModel}
	\end{figure}
	The equations of motion can be written in the dimensionless form \cite{Szabelski1997}
	\begin{align}
		\label{eq:self-1}
		\ddot{y}_1 + \Tilde{k} ( 1 - \Tilde{\mu} \cos{2 \Omega t}) (y_1 - y_2) + y_1 + \Tilde{\kappa} y_1^3 = q \cos{ \Omega t},
		\\
		\label{eq:self-2}
		\ddot{y}_2 - ( \Tilde{c} - \Tilde{\gamma} \dot{y}_2^2) \dot{y}_2 - \Tilde{k} M (1 - \Tilde{\mu}\cos{2 \Omega t}) (y_1 - y_2) = 0.
	\end{align}
	The parameters values are chosen as $\Tilde{k} = 4, \Tilde{\mu} = 0.02, \Tilde{\kappa} = 0.1,  \Tilde{c} =  0.01, \Tilde{\gamma} = 0.05, M = 0.5 $ and $q = 0.02$ and $\epsilon =1$. Using SSM-theory to obtain the FRC of this dynamical system involves the following steps.
	
	\begin{figure}[htbp]
		\centering
		\subfloat[]{ 
			\includegraphics[width=0.4\textwidth]{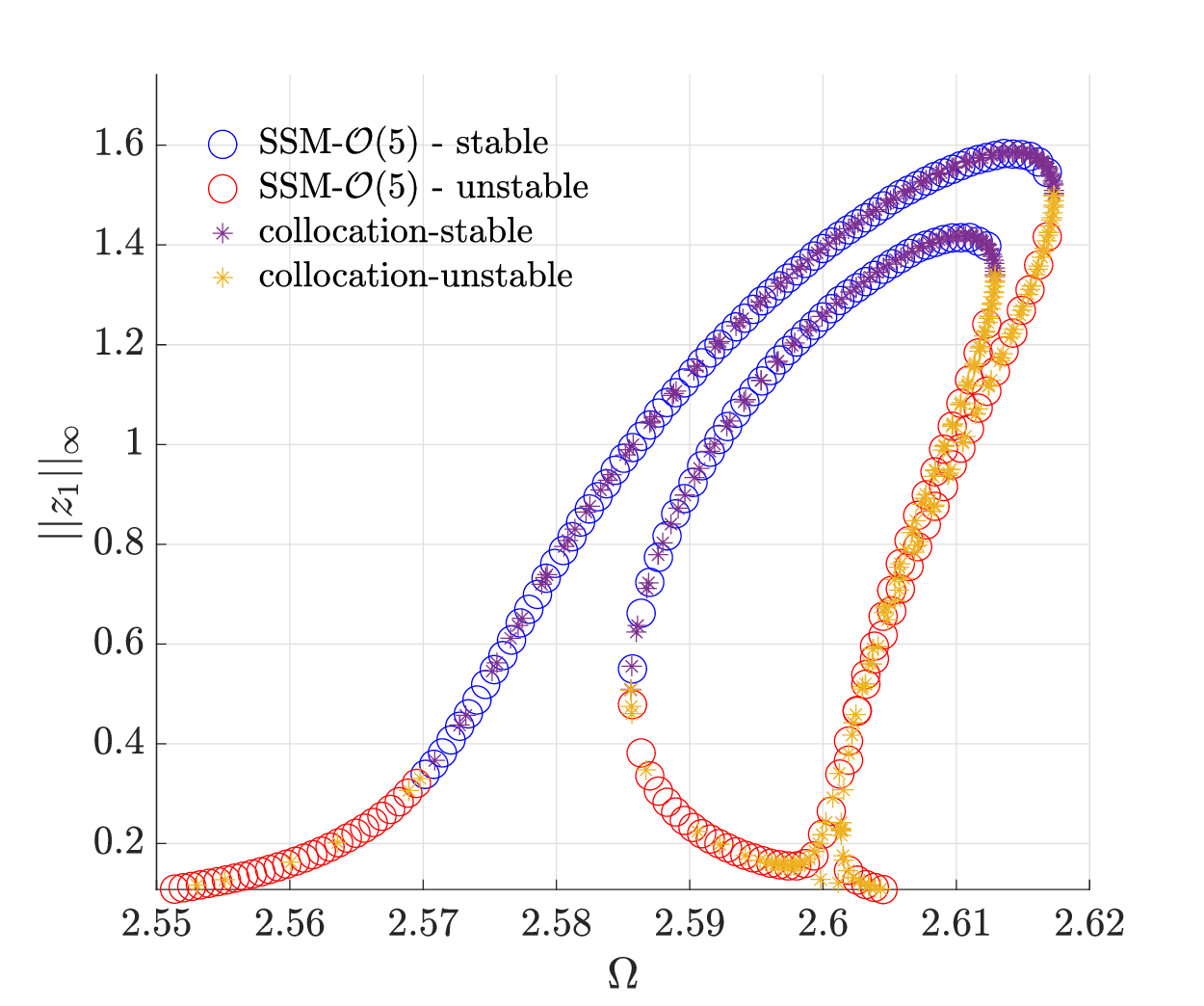}
		}
		\hfill
		\subfloat[]{    \includegraphics[width=0.4\textwidth]{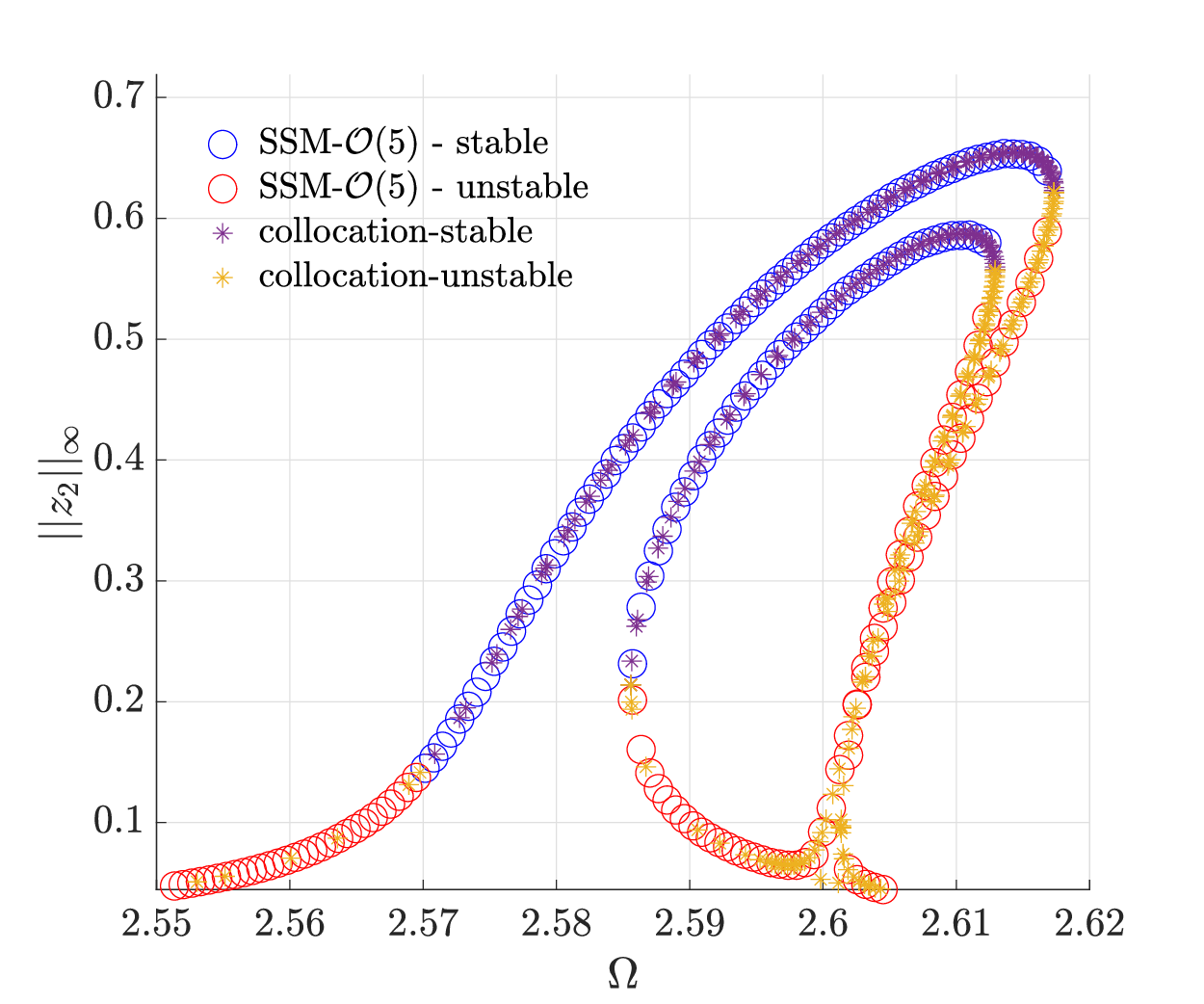}
		}
		\caption{Forced response curve of the self-excited oscillator. The forced response obtained from a fifth-order approximation of the dynamics on the non-autonomous SSM accurately describes the full system response. Results are verified using \texttt{po}-toolbox of \textsc{coco}.} \label{fig:SelfExcite_FRC}
	\end{figure}

	\begin{enumerate}
		\item System setup:
		Using the standardized form~\eqref{eq:DS}, the system~\eqref{eq:self-1}-\eqref{eq:self-2} can be rewritten as 
		\begin{align}
			&    \bm{M} = \begin{pmatrix}1 & 0 \\ 0 & 1   \end{pmatrix}, \quad 
			\bm{K} = \begin{pmatrix} \Tilde{k} +1 & - \Tilde{k} \\ - \Tilde{k} M & \Tilde{k} M \end{pmatrix},\\    
			&    \bm{Q} = \Tilde{k}\begin{pmatrix} - 1 &  1 \\   M & -  M \end{pmatrix},\quad
			\bm{C} = \begin{pmatrix} 0 & 0 \\ 0 & -\Tilde{c}  \end{pmatrix},\\
			&   \bm{f}(\bm{y},\dot{\bm{y}}) = \begin{pmatrix} \Tilde{\kappa} y_1^3 \\ \Tilde{\gamma} \dot{y}_2^3  \end{pmatrix}, \quad
			\bm{g}(\Omega t) 
			= q \begin{pmatrix}  \cos{\Omega t} \\ 0  \end{pmatrix}.
		\end{align}
		To solve the invariance equations, we expand the nonlinearity and forcing in Taylor and Taylor-Fourier series. Details on this procedure can be found in Appendix~\ref{App:ForceExpansion}.
		\item Choose master subspace: Linear analysis of the corresponding first-order ODE~\eqref{eq:DS_lin} provides the eigenvalue pairs $\lambda_{1},\conj{\lambda_1} \approx 0.0013 \pm 2.5887 i$ and $\lambda_{2},\conj{\lambda_2} \approx 0.0037 \pm 0.5463 i  $. We choose to analyze the resonant excitation of the first mode with excitation frequency values $\Omega \approx \Imag{\lambda_1}$ and construct a time-periodic SSM over the spectral subspace $\mathcal{E} = \{\bm{v}_1, \conj{\bm{v}_1} \}$ .
		\item Autonomous SSM and reduced dynamics: The leading order autonomous invariance equation is solved by choosing 
		\begin{align}
			\bm{W}_{\bm{e}_1} = \bm{v}_1 ,\ \bm{W}_{\bm{e}_2} = \conj{\bm{v}_1},
			\\
			\bm{R}_{\bm{e}_1} = \begin{pmatrix}\lambda_1 \\ 0 \end{pmatrix}, \
			\bm{R}_{\bm{e}_1} = \begin{pmatrix} 0 \\ \conj{\lambda_1} \end{pmatrix}.
		\end{align}
		Subsequently, the autonomous invariance equation~\eqref{eq:Hom_Aut_Second} can be solved recursively to the desired order. Near-resonances between the master modes of the type \eqref{eq:ResCond_aut} are detected automatically and the autonomous reduced dynamics are transformed to their normal form accordingly.
		\item Assemble non-autonomous invariance equations: The autonomous SSM and reduced dynamics coefficients do not depend on the external excitation frequency. As this is not the case for non-autonomous coefficients, the nonautonomous invariance equation has to be solved for every forcing frequency $\Omega$. Given a set of forcing frequencies $\{ \Omega_1, \cdots , \Omega_z\}$, the nonautonomous invariance equation is solved up to the desired order independently for each frequency.  Near-resonances in the master subspace of the form \eqref{eq:ResCond_nonaut} are automatically detected at each order of computation, which determines the nontrivial coefficients of the reduced vector field according to a normal form style of parametrization. The coefficients of the reduced dynamics are used to construct the function $\bm{r}(\rho, \psi, \Omega)$ defined in eq.~\eqref{eq:PolarDynamics}, whose zero level set provides us the forced response according to Lemma~\ref{lem:FRC}. The resulting reduced dynamics up to cubic order at frequency $\Omega = 2.6$ rad/s can be written as
		\begin{align} \nonumber
			\dot{p} &= (1328.6 + 25887.3i)p 
			+ (-35.9 + 63.2i)p^2\Bar{p} 
			\nonumber \\ \nonumber
			& \ + (7.4  + 1152.5i) e^{2 i \phi}\Bar{p} 
			+ (0.2 - 0.8i)   e^{-i\phi}p^2
			\\ \label{eq:NormalForm_example} 
			& \ + (-0.4 +1.6i)  e^{i\phi}p\Bar{p}
			+ (1 - 1.4i) e^{- 2 i \phi}p^3, 
		\end{align}
		where we have used a rescaled time variable $s = 10^{-4}t$, and defined the phase variable $\phi = \Omega t$. The dynamics along the second reduced coordinate is simply given by the complex conjugate of~\eqref{eq:NormalForm_example}.
	\end{enumerate}

	FIG.~\ref{fig:SelfExcite_FRC} shows the FRC obtained from this procedure, where we observe a interesting self-intersecting loop in the FRC. Once again, we verified the SSM-based reduced results against the full system using the \texttt{po}-toolbox of \textsc{coco}, where the relevant continuation parameters were set according to TABLE \ref{tab:CocoParametersFRC} in Appendix \ref{App:COCO}. The computation time for this self-excited oscillator using a sequential implementation of the non-autonomous SSM computation is 3 minutes and 48 seconds. Meanwhile, the verification via continuation on the full system takes 3 minutes and 32 seconds. We observe that in this sequential implementation on a low-dimensional example, the ROM does not bring any computational savings relative to the full system in this low-dimensional example.

	Similar oscillator systems featuring self-excitation or parametric amplification have been shown to exhibit complicated forced response curves, including isolated branches of response~\cite{Warminski2DOF1995,Li2020}. To demonstrate the effectiveness of SSM-based ROMs in reproducing such complicated forced response curves, we have implemented further examples of self-excited oscillators with isolas~\cite{Warminski2DOF1995} and parametric amplifiers with nonlinear damping~\cite{Li2020}. These examples are openly available in SSMTool~\cite{SSMTool2}. Next, we discuss finite-element-type examples under parametric and external excitation. 

	\subsubsection{Bernoulli beam with parametric and external excitation}\label{sec:ExamplesFRCBB}
	We now return to the nonlinear Bernoulli beam example with the parameters listed in TABLE~\ref{tab:BB_model}. In addition to the parametric excitation, however, the free end is now forced externally as well. The resulting equations of motion for the discretised variables are now given by
	\begin{align}
		\bm{M}\ddot{\bm{y}}+\bm{C}\dot{\bm{y}}+\bm{K}\bm{y} + \epsilon\mu \cos 2\Omega t\bm{Q}   \bm{y}+ \bm{f}(\bm{y},\dot{\bm{y}})= \epsilon \bm{f}^0_{ext}\cos(\Omega t).
	\end{align}

	Here $\bm{f}(\bm{y},\dot{\bm{y}})$ again includes the nonlinear damping coefficient $\gamma$ of the damper, as well as the nonlinear stiffness $\kappa$ of the spring attached to the tip of the beam. $\bm Q$ denotes a projection onto the tip of the beam along the transverse direction. The harmonic external excitation vector $\bm{f}^0_{ext}$ acts on the tip of the beam, with a unit magnitude in the transverse direction. The damping coefficient is set to $\sigma = 1$. The external excitation amplitude and the relative parametric excitation amplitude are taken to be $\epsilon = 2 \times 10^{-3}$ and $\mu = 45 [N]$, respectively. The system is forced in resonance with the first mode with the eigenvalue pair
	
	\begin{align}
		\lambda_{1},\conj{\lambda_1} \approx -0.1238 \pm 6.9995 i.
	\end{align}
	
	\begin{figure}[h!]
		\centering
		\subfloat[]{
			\includegraphics[width=0.4\textwidth]{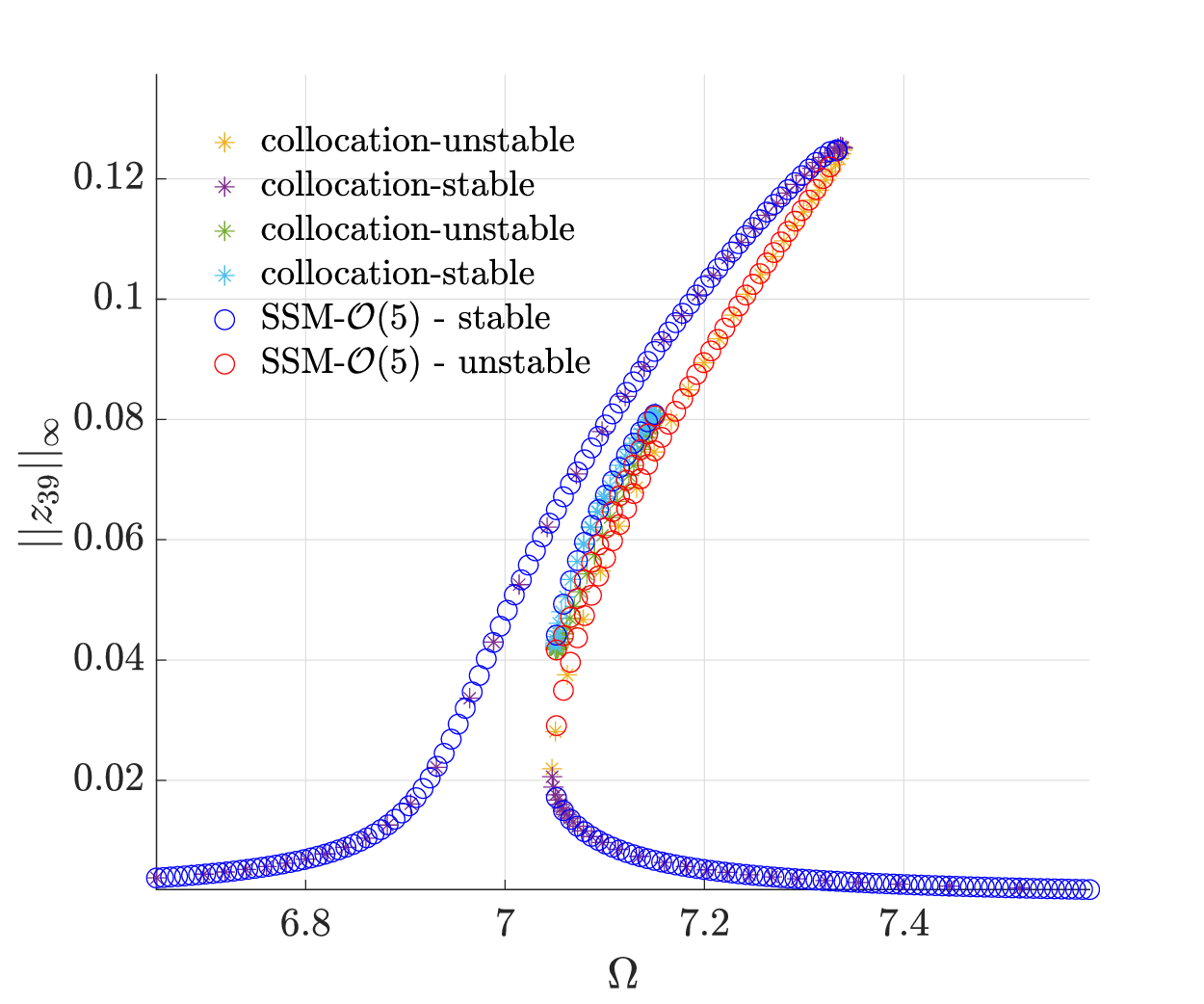}
			\label{fig:BB_external_a}
		}\hfill
		\subfloat[]{
			\includegraphics[width=0.4\textwidth]{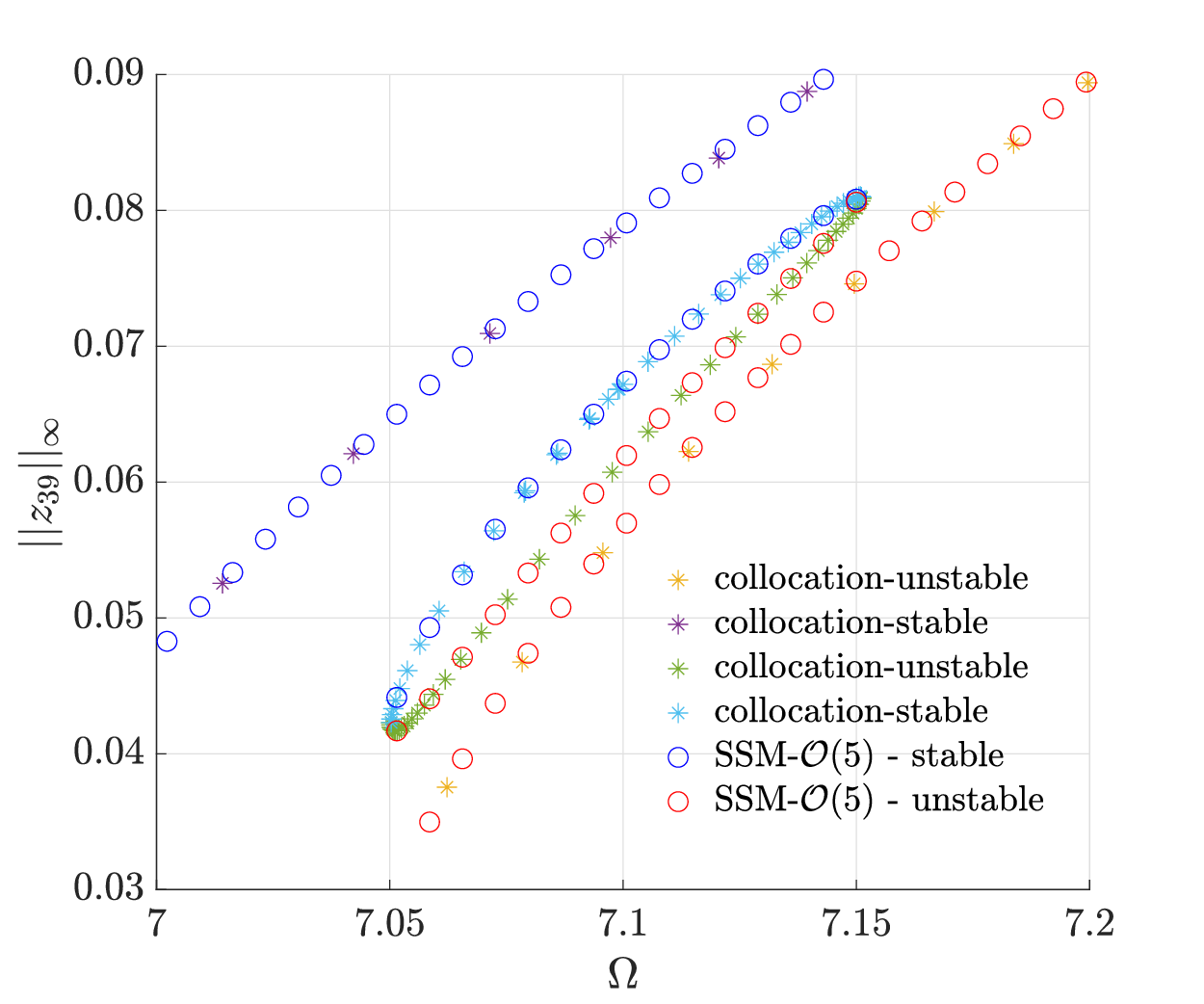}
			\label{fig:BB_external_b}
		} 
		\caption{(a) Forced response obtained for the externally excited Bernoulli beam with periodically varying stiffness actuation at the tip for $n = 40$ DOFs. The nonlinear spring attached to the beam leads to a hardening-type response. Due to the interplay of external and parametric excitation, an isola arises for sufficiently large values of the parametric excitation amplitude $\mu$ ($\mu = 45 [N]$ for this simulation). The isola is independently verified using an initial condition provided by the results of the SSM computation. (b) Close-up view of the isolated response branch verified against the full system in a separate continuation run. } \label{fig:BB_external}
	\end{figure}
	
	The excitation frequencies for parametric and external excitation are assumed to have a $2:1$ relationship. Therefore, the direct forcing is resonant with the first eigenmode of the beam, whereas the parametric term is operated in the regime of principal parametric resonance. FIG.\ref{fig:BB_external_a} shows the forced response for a 20 element model with $n=40$ DOFs, where we observe an isolated branch of forced response under the main branch for the chosen parameter values. The emergence of such an isola with an increasing parametric excitation amplitude $\mu$ has been studied in systems where parametric resonance and external excitation interact\cite{Rhoads2009, Aghamohammadi2022DynamicExcitations}. Indeed, this isola is associated to the coexistence of the two resonances, i.e., the parametric 2:1 resonance and the direct 1:1 resonance of the external excitation with the eigenvalues of the first mode.
	
	We verify the FRC obtained from the SSM-based ROM against the full system using the \texttt{po} toolbox of \textsc{coco} with parameters set according to TABLE \ref{tab:CocoParametersFRC}. The isolated branch of the forced response, which appears inside the main branch of the FRC, is not detected by the continuation run. We use an initial solution on the isola provided by the SSM-based ROM as an input for \textsc{coco} to independently verify the isolated branch for the full system. Verification of the main branch took a total of 4 hours and 40 minutes, while the isolated response needed 15 hours and 17 minutes. At the same time, using Lemma \ref{lem:FRC}, the SSM-based ROM takes only 48 seconds for FRC extraction with a sequential implementation.

	Next, we repeat our analysis for larger system sizes resulting from higher-order discretizations of the beam. The corresponding computation times are depicted in FIG.~\ref{fig:CompTimeExternalBB}. All computations for this example were performed on the Euler Cluster of ETH Zürich due to the long runtimes of the collocation routine. We increased the system size up to 25,000 elements, resulting in up to $n=50000$ DOFs. The sequential computation of the FRC, including the repeated computation of the non-autonomous SSM and its reduced dynamics for each forcing frequency $\{\Omega_1, \cdots , \Omega_z \}$, now takes 1 hour and 19 minutes. 
	\begin{figure}[h!]
		\centering
		\includegraphics[width=0.5\textwidth]{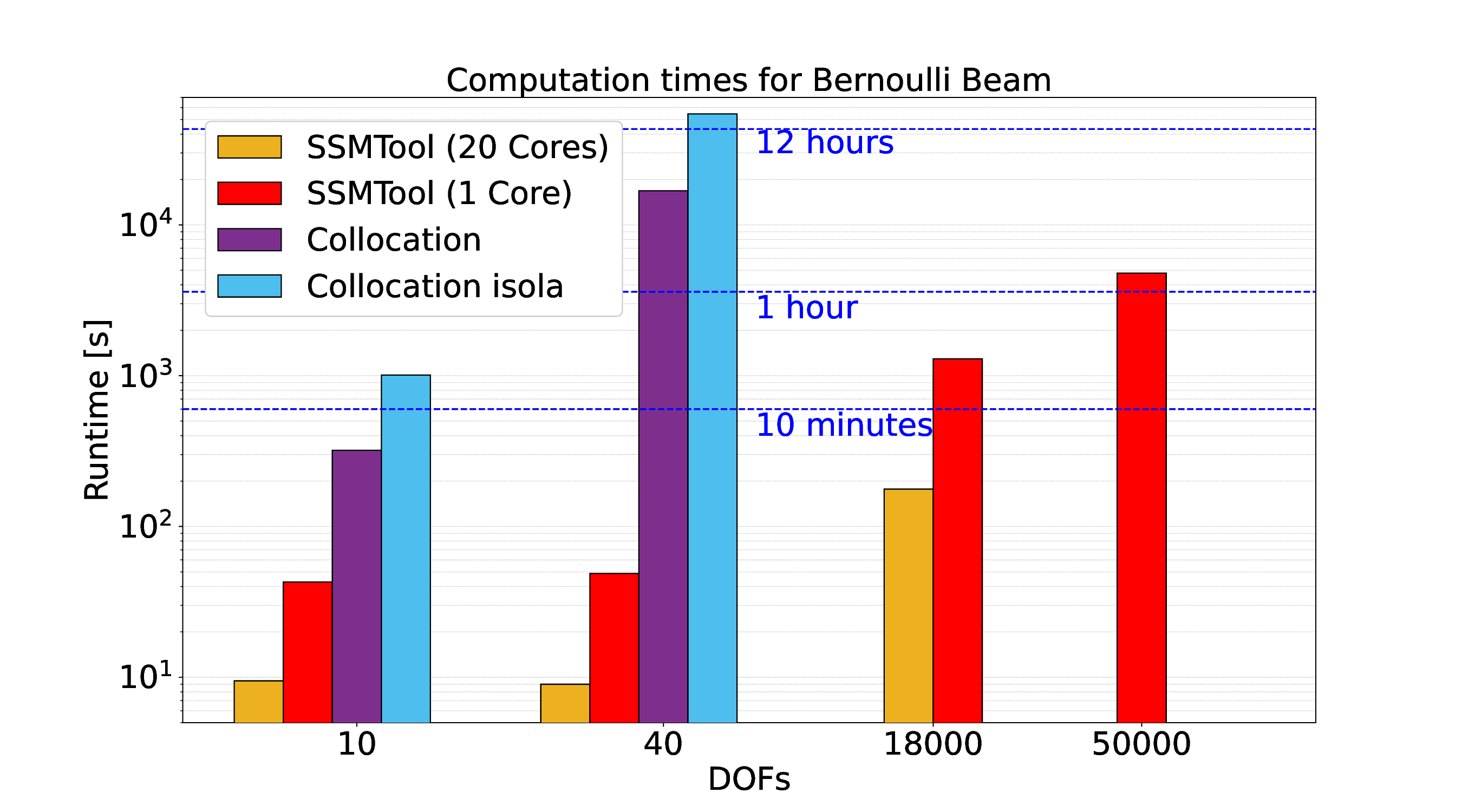}
		\caption{Computation times for the FRC of the externally excited Bernoulli beam example. Results from the SSM-based ROMs were verified against the full system using the \texttt{po}-toolbox of \textsc{coco} (see FIG.~\ref{fig:BB_external}). Computation for isolas in the forced response are reported separately as they were verified in independent continuation runs.} \label{fig:CompTimeExternalBB}
	\end{figure}

	\section{Conclusion}
	\label{sec:conc}
	We have developed an algorithm for computing non-autonomous SSMs and their reduced dynamics up to any desired order of accuracy in externally and parametrically forced systems. In doing so, we have taken into account prior developments and extended them to the case of parametrically excited mechanical systems with nonlinearities of arbitrary order and magnitude, including nonlinear damping. We have put forward automated routines for computing SSMs for second-order mechanical systems in particular and first-order dynamical systems in general. In our derivations and computations, we have made use of the multi-index notation for the expansions and carried out computations in physical coordinates using minimal eigenvectors of the full system. All this optimizes memory requirements and avoids modal transformations that are computationally infeasible for high-dimensional systems. 
	
	Using the SSM-reduced models as the basis for parameter continuation, we have shown how to compute the stability diagrams for the principal resonance of periodic, parametrically excited systems. Furthermore, we have shown that in systems under additional external excitation, the forced response can be obtained by simply locating the zeros of a two-dimensional function, i.e., without numerical continuation. As a result, it is now possible to parallelize the computation of the periodic responses attached to two-dimensional, nonautonomous SSMs for systems subject to periodic external and parametric excitation over any desired range of forcing frequencies. All these algorithms and their applications to our examples are available in the open-source MATLAB package SSMTool 2.4~\cite{SSMTool2}.
	
	With examples treating coupled Mathieu equations, a parametrically excited Bernoulli beam, and an axially forced prismatic beam, we have illustrated how SSM-based model reduction can be used to obtain accurate stability diagrams and FRCs in parametrically excited systems at a significantly reduced computational cost. Indeed, we analyzed a number of typical low-dimensional models studied previously in the literature and showed that SSM-based ROMs provide accurate predictions even if they do not necessarily yield computational savings. In contrast, SSM-reduced models also provide major computational advantages over the full system in more complex examples such as a Bernoulli beam attached to a nonlinear damper and spring, and subject to external and parametric excitation. In that example, we have uncovered an internal isola, which is an effect that has been observed for one-dimensional models subject to parametric resonance and external excitation~\cite{Rhoads2009, Aghamohammadi2022DynamicExcitations}. Ours appears to be the first demonstration of this phenomenon in finite element models. We showcased the power of this mathematically grounded SSM-reduction by increasing system sizes beyond the capabilities of collocation-based continuation methods when applied to the full system. We also demonstrated how parallelization can be used to reduce the computation time by an order of magnitude in SSM-based FRC calculations.
	
	Future work may target systems with even higher dimensionality, where the beneficial effects of model reduction should weigh in even more stunningly. Areas of related applications include various types of parametrically driven MEMS devices, such as comb-driven micromirrors, among others \cite{Frangi2017, Polunin2017, Welte2013}. Additionally, the developments presented here open the door for an in-depth analysis of the interaction of more complex time-dependent phenomena, such as parametric resonance in MEMS systems with internal resonances~\cite{Gobat2021ReducedResonance}.

	\section*{Data availability}
	The code used to generate the numerical results included in this paper is available as part of the open-source MATLAB package SSMTool, available at \url{https://doi.org/10.5281/zenodo.4614201}.
	
	\begin{acknowledgments}
		We declare no conflict of interest. We thank Mingwu Li for his helpful comments and advice on the continuation toolbox \textsc{coco}.
	\end{acknowledgments}

	\appendix

	\section{SSM Computation in Multi-Index format}\label{app_sec:SSMComp}
	In this section of the appendix, we provide the mathematical details for SSM computations in physical coordinates with the use of multi-indices, including the proofs for Lemmas \ref{lem:HomAut} - \ref{Lem:HomNonautVel}. We begin by reviewing the multi-index formalism and its advantages. We then give an in-depth derivation of the expressions that determine the SSM and its reduced dynamics parametrizations. We present these results for second- and first-order ODEs

	\subsection{Multi-Index calculus}\label{app_sec:MathMultiIndex}
	A multi-index of order $m \in \mathbb{N}$ is a vector $\bm{m} \in \mathbb{N}^M$ such that $||\bm{m}||_1 = m$.
	The unit multi-indices with a single nonzero entry in the $j^{\mathrm{th}}$ position are denoted by $\bm{e}_j$. The addition and subtraction of multi-indices are defined element-wise. Using multi-indices, multivariate monomials can be conveniently expressed as 
	\begin{equation}
		\bm{p^m} = p_1^{m_1} \cdots p_M^{m_M}.
	\end{equation}
	Likewise, a time derivative of such monomials are conveniently expressed as 
	\begin{equation}
		\frac{d\bm{p^m}}{dt} 
		=
		\sum_{j=1}^M \partial_{p_j} \bm{p^m}\dot{p}_j
		=
		\sum_{j=1}^M m_j\bm{p}^{\bm{m-e}_j } \dot{p}_j.
	\end{equation}
	\subsubsection*{Composition Coefficients}
	In SSM computations, composition of nonlinear forces with the nonlinear SSM parametrization has to be performed. For this, we need to compute scalar powers $s$ of  power series written in multi-index notation, i.e., $(\sum_{\bm{m}} W^i_{\bm{m}}\bm{p^m})^s$. A convenient method, which offers a recursive definition of the resulting coefficients, uses the concept of radial derivative discussed by Haro et al.~\cite{Haro2016}. Ponsioen et al.~\cite{Ponsioen2020} use this method to express the coefficients $H_{s,\bm{h}}$ in the expansion 
	\begin{equation} 
		\big( \sum_{\bm{m}}W^i_{\bm{m}}\bm{p^m} \big)^s =  \sum_{\bm{h}}H^i_{s,\bm{h}}\bm{p^h}
	\end{equation}
	using the recursive definition
	\begin{equation}
		H^i_{s,\bm{h}} =  \frac{s}{h_j} \sum_{\substack{ \bm{m \leq h} \\}} m_j W^i_{\bm{m}} H^i_{s-1,\bm{h-m}},
	\end{equation}
	where $j$ is the index of the smallest nonzero element in $\bm{h}$.
	To simplify the calculation of this recursion, it is useful to look at some special cases under the assumption that the coefficients $W^i_{\bm{m}}$ are the autonomous coefficients of the $i^{\mathrm{th}}$ coordinate of the SSM expansion. The SSM parametrization does not contain constant terms, which means that all composition coefficients are zero for $|\bm{m}|= 0$. The properties of $H^i_{s,\bm{h}}$ are then
	\begin{itemize}
		\item $H^i_{1,\bm{0}} = 0$,
		\item $H^i_{1,\bm{h}} = W^i_{\bm h}$,
		\item $H^i_{0,\bm{h}} = 0 $ for all $\bm{h} \neq \bm{0}$, 
		\item $H^i_{0,0} = 1 $.
	\end{itemize} 
	We generalize this concept to vector-valued power series by collecting all terms corresponding to the same multi-index $\bm{h} \in \mathbb{N}^M$ and defining a new composition coefficient $H_{\bm{n}, \bm{h}} $ such that
	\begin{align}
		(  \sum_{\bm{m}\in \mathbb{N}^M} \bm{W}_{\bm{m}} \bm{p}^{\bm m} ) ^{\bm n} 
		&= \prod_{i=1}^N ( \sum_{\bm{m}\in \mathbb{N}^M} W^i_{\bm{m}} \bm{p}^{\bm m} )^{n_i}
		\\   &= \prod_{i=1}^N ( \sum_{\bm{u}\in \mathbb{N}^M} H^i_{n_i,\bm{u}} \bm{p}^{\bm u})
		\\   &= \sum_{\bm{h} \in \mathbb{N}^M} H_{\bm{n}, \bm{h}} \bm{p}^{\bm h}.
	\end{align}
	In more detail:
	\begin{align}
		H_{\bm{n},\bm{h}} :=
		\sum_{\bm{u}_1 \in \mathbb{N}^M } 
		\cdots 
		\sum_{\bm{u}_{2n} \in \mathbb{N}^M  }
		H^1_{n_1,\bm{u}_1}
		\cdots 
		H^{2n}_{n_{2n},\bm{u}_{2n}}
		\bigg|_{\sum \bm{u}_i = \bm{h}}.
	\end{align}
	Although a vectorization of this computation provides a very efficient routine for the composition of power series, it also poses a major bottleneck in terms of memory and time requirements in the SSM computation proposed here. 
	\subsubsection*{Redundancy in tensor notation}\label{App:TensorRedundancy}
	We now discuss the computational disadvantages of a tensor-based notation compared to the use of multi-indices. Let $\bm{p} \in \mathbb{C}^2$ and consider the Taylor-expansion of a nonlinear function $\bm{W}(\bm{p})$ in the tensor notation as
	\begin{align}
		\label{eq:tensor_exp}
		\bm{W}(\bm{p}) = \sum_{i \in \mathbb{N}} \bm{W}_i \bm{p}^{\otimes i }.
	\end{align}
	Specifically, the second term in the expansion~\eqref{eq:tensor_exp} can be represented by a $2 \times 4$ matrix $\bm{W}_2$ acting on the Kronecker product $  \bm{p} \otimes \bm{p}$ as 
	\begin{align}
		\bm{W}_2 \bm{p} \otimes \bm{p}
		= 
		\begin{bmatrix}
			W_{11} & W_{12} & W_{13} & W_{14} \\
			W_{21} & W_{22} & W_{23} & W_{24}
		\end{bmatrix}
		\begin{bmatrix}
			p_1 p_1 \\ p_1 p_2 \\ p_2 p_1 \\ p_2 p_2
		\end{bmatrix}
		\\ \label{eq:tensor_unique}
		=
		\begin{bmatrix}
			W_{11} p_1^2 +  (W_{12} + W_{13}) p_1 p_2 + W_{14} p_2^2 \\
			W_{21} p_1^2 +  (W_{22} + W_{23}) p_1 p_2 + W_{24} p_2^2
		\end{bmatrix}
	\end{align}
	Examining Eq. \eqref{eq:tensor_unique}, we note that only 6 of the 8 coefficients contained in the matrix $\bm{W}_2$ are needed to uniquely represent $\bm{W}_2 \bm{p} \otimes \bm{p}$. Indeed, the same term in multi-index notation, 
	\begin{align}
		\bm W ( \bm p ) = \sum_{\bm m \in \mathbb {N}^M} \bm W_{\bm m} \bm p^{\bm m},
	\end{align}
	is characterised by 6 unique multi-index coefficients given as
	\begin{align}
		\bm{W}_{\big(\substack{ 2\\ 0}\big)}=
		\begin{bmatrix}
			W_{11}  \\
			W_{21} 
		\end{bmatrix}, 
		\quad
		\bm{W}_{\big(\substack{ 1\\ 1}\big)}=
		\begin{bmatrix}
			W_{12}+W_{13}  \\
			W_{22}+W_{23} 
		\end{bmatrix},
		\quad
		\bm{W}_{\big(\substack{ 0\\ 2}\big)}=
		\begin{bmatrix}
			W_{13}  \\
			W_{23} 
		\end{bmatrix}
	\end{align}
	This redundancy in the number of coefficients utilized by the tensor notation relative to the multi-index notations grows with the order of expansion. Consider, for instance, the case of a 2-dimensional SSM, where, at any order $\Gamma$, the expansion is characterized by a tensor with $N  2^{\Gamma +1 }$ coefficients. However, we have only $ \Gamma + 1$ unique monomials associated with these coefficients. Hence, while the number of unique monomials grows linearly with the order $\Gamma$, the corresponding number of tensor coefficients grows exponentially in $\Gamma$.
	
	%
	We remark, however, that the tensors can be specially constructed to contain only a minimal number of nonzero coefficients equal to that in the multi-index notation. This is possible by exploiting the inherent symmetries in monomials associated with the expansion, e.g., by using a sparse or symmetric definition of tensor arrays. In practice, the SSM coefficients which are determined numerically by solving invariance equations may still lead to densely populated or assymmetric tensor arrays. At each order of computation, the computed coefficients would, therefore, require recasting to this minimal form, resulting in unnecessary computational overhead. 
	
	Furthermore, it turns out that the invariance equations~\eqref{eq:Hom_Aut_Second}, \eqref{eq:HomNonAut_Second} that need to be solved for SSM expansion coefficients are decoupled for individual multi-indices. 
	For computing a two-dimensional autonomous SSM at order $\Gamma$ in a vectorized tensor-based notation adopted in Ref.~\onlinecite{Jain2021}, a linear system of dimension $N  2^{\Gamma +1 }$ needs to be solved. Using multi-indices instead, only $\Gamma + 1$ linear systems of dimension $n$ each need to be solved. One can also solve the linear systems associated with each tensor index independently, as remarked in Ref.~\onlinecite{Jain2021} and as done for second-order systems in Refs.~\onlinecite{Vizzaccaro2022, Opreni2022}. In that case, a total of $2^{\Gamma +1 }$ systems of linear equations of dimension $n$ must be solved in the tensor-index notation. Hence, for two-dimensional SSMs, the use of multi-indices leads to an exponential decrease in the number of linear systems that needs to be solved relative to the tensor notation as the expansion order increases.
	
	Thus, we have shown the tensor notation result in redundant computational costs relative to the multi-index notation.
	\subsubsection*{Examples of nonlinear force expansion}\label{App:ForceExpansion}
	We expand the excitation and nonlinearity in the multi-index notation as
	\begin{align}
		\label{eq:nonlinearity_exp}
		\bm{f}(\bm{y},\dot{\bm{y}}) 
		&=  \sum_{\bm{a,b}\in \mathbb{N}^n}          
		\bm{f}_{\bm{a,b}} \bm{y^a}\dot{\bm{y}}^{\bm b},
		\\
		\label{eq:force_exp}
		\bm{g}(\bm{y},\dot{\bm{y}},\boldsymbol\phi) 
		&=  \sum_{\bm{a,b}\in \mathbb{N}^n}          
		\sum_{\boldsymbol\kappa \in \mathbb{Z}^k} \bm{g}_{\bm{a,b},\boldsymbol\kappa } e^{i \langle \boldsymbol\kappa , \boldsymbol\phi  \rangle}\bm{y^a}\dot{\bm{y}}^{\bm b},
	\end{align}
	where the coefficients $\bm{f}_{\bm{a,b} },~\bm{g}_{\bm{a,b}, \bm \kappa} \in \mathbb{C}^n$. Note that these coefficients can be recast in a first-order form simply by associating a multi-index $\bm{n} \in \mathbb{N}^N$ to each pair $\bm{a,b}$ such that in the phase space, $\bm{z^n} = \bm{y^a}\dot{\bm{y}}^{\bm b}$. The first-order version of the two expansions~\eqref{eq:nonlinearity_exp}, \eqref{eq:force_exp} can then be written as
	\begin{align}
		\label{eq:expF}
		\bm{F}(\bm{z}) 
		&=  \sum_{\bm{n}\in \mathbb{N}^N}          
		\bm{F}_{\bm{n}} \bm{z}^{\bm n}
		, \quad 
		\bm{F}_{\big(\substack{ \bm a\\ \bm b}\big)} =          \begin{bmatrix}
			\bm{f}_{\bm a ,\bm b}\\ \bm{0}_n
		\end{bmatrix},
		\\ \label{app_eq:Force_expand_FirstOrder}
		\bm{G}(\bm{z},\boldsymbol\phi) 
		&=  \sum_{\bm{n}\in \mathbb{N}^N}          
		\sum_{\boldsymbol\kappa \in \mathbb{Z}^k} \bm{G}_{\bm{n},\boldsymbol\kappa } e^{i \langle \boldsymbol\kappa , \boldsymbol\phi  \rangle}\bm{z}^{\bm n}
		,\quad
		\bm{G}_{\big(\substack{ \bm a\\ \bm b}\big),\boldsymbol\kappa} =          \begin{bmatrix}
			\bm{g}_{\bm{a,b},\boldsymbol\kappa } \\ \bm{0}_n
		\end{bmatrix}.
	\end{align}
	As an example of the above notation, we consider the forcing in the self-excited oscillator system~\eqref{eq:self-1}-\eqref{eq:self-2}. The forcing term contains external and linear parametric excitation given by
	\begin{align}
		\bm{g}(\Omega t, \bm{y}) 
		&= \cos{\Omega t} \begin{pmatrix} q  \\ 0  \end{pmatrix} 
		+  \lambda \cos{2\Omega t} 
		\begin{pmatrix} y_1 - y_2 \\ -M (y_1 - y_2) \end{pmatrix} 
		\\
		&= (e^{i\Omega t} + e^{-i\Omega t}) \begin{pmatrix} q/2  \\ 0  \end{pmatrix} 
		\nonumber \\
		&+  (e^{i2\Omega t} + e^{ -i2\Omega t} ) \frac{\lambda}{2}  
		\begin{pmatrix} y_1 - y_2 \\ -M (y_1 - y_2) \end{pmatrix} .
	\end{align}
	Here, the set of harmonics is given as $K = \{1, -1, 2, -2 \}$. As the system contains only 2 DOFs, the spatial multi-indices are two dimensional and the vectors $\bm{g}_{\bm{n},\kappa}$ that characterise the expansion~\eqref{eq:force_exp} are given as
	\begin{align}
		\bm{g}_{\bm{0},1} &= \bm{g}_{\bm{0},-1} =\begin{pmatrix}   q/2  \\ 0 \end{pmatrix},
		\\
		\bm{g}_{\bm{e}_1,2} &= 
		\bm{g}_{\bm{e}_1,-2} = \begin{pmatrix} \lambda/2  \\ - M \lambda /2 \end{pmatrix},
		\\
		\bm{g}_{\bm{e}_2,2} &= 
		\bm{g}_{\bm{e}_2,-2} = \begin{pmatrix} -\lambda/2  \\  M \lambda /2 \end{pmatrix}.
	\end{align}
	\subsection{Second order computation}\label{app_sec:SSMComp_2ndorder}
	\subsubsection*{Autonomous SSM Computation}
	In order to compute the SSMs and their reduced dynamics parametrizations for second-order mechanical systems, we express them as Taylor series~\eqref{eq:AutSSMParam},\eqref{eq:AutSSMParamVel} and \eqref{eq:AutRDParam}. Substituting these expansions into the autonomous invariance equation \eqref{eq:InvEq_aut} and collecting unique monomial terms, we obtain the system~\eqref{eq:Hom_Aut_Second} of the linear equations for each multi-index $\bm m$ as we detail below. At the leading order, i.e, for monomial multi-index $\bm{p}^{\bm{e}_i}$, we obtain 
	\begin{align}\label{app_eq:SecondOrderSystemLinearInvEq}
		\begin{bmatrix} \bm{C} & \bm{M} \\ \bm{M} & \bm{0} \end{bmatrix}
		\sum_{j=1}^M
		\begin{bmatrix}
			\bm{w}_{\bm{e}_j} \\ \bm{\dot{w}}_{\bm{e}_j}
		\end{bmatrix}
		R^j_{\bm{e}_i}
		=			\begin{bmatrix} -\bm{K} & \bm{0} \\ \bm{0} & \bm{M} \end{bmatrix}
		\begin{bmatrix}
			\bm{w}_{\bm{e}_i} \\ \bm{\dot{w}}_{\bm{e}_i}
		\end{bmatrix}.
	\end{align}
	Eliminating the velocity variables $\bm{\dot{w}}_{\bm{e}_i}$ from system~\eqref{app_eq:2ndOrderSys_AutHighOrderInvEq}, we obtain
	\begin{align}
		\label{eq:leading_eig}
		\bm{M} \sum_{j=1}^M \sum_{k=1}^M \bm{w}_{\bm{e}_k} R^k_{\bm{e}_j} R^j_{\bm{e}_i}
		+ \bm{C} \sum_{j=1}^M \bm{w}_{\bm{e}_j}  R^j_{\bm{e}_i}
		+ \bm{K} \bm{w}_{\bm{e}_i}
		= \bm{0}.
	\end{align}
	From the eigenvalue problem~\eqref{eq:EigenprobSecondOrder}, we note that the choice $\bm{w}_{\bm{e}_j} = \bm  \phi_j$ and $R^j_{\bm{e}_j} = \delta_{ij} \lambda_i$ solves equation~\eqref{eq:leading_eig}. Together with eq.~\eqref{app_eq:SecondOrderSystemLinearInvEq} this solution choice for the leading-order displacement parametrization implies that $\bm{\dot{w}}_{\bm{e}_i} = \bm  \phi_i \lambda_i$.  Higher-order terms of the invariance equation~\eqref{eq:InvEq_aut} for a multi-index ${\bm m}$ are given as
	\begin{align}\label{app_eq:2ndOrderSys_AutHighOrderInvEq}
		\bigg(
		\begin{bmatrix} -\bm{K} & \bm{0} \\ \bm{0} & \bm{M} \end{bmatrix}
		& -
		\Lambda_{\bm m}
		\begin{bmatrix} \bm{C} & \bm{M} \\ \bm{M} & \bm{0} \end{bmatrix}
		\bigg) 
		\begin{bmatrix}
			\bm{w}_{\bm m} \\ \bm{\dot{w}}_{\bm m}
		\end{bmatrix}
		\nonumber \\ &=
		\sum_{j=1}^M 
		\begin{bmatrix} \bm{C} & \bm{M} \\ \bm{M} & \bm{0} \end{bmatrix} 
		\begin{bmatrix}
			\bm  \phi_j 
			\\
			\bm  \phi_j \lambda_{j}
		\end{bmatrix} 
		R^j_{\bm m}
		+
		\begin{bmatrix} \bm{1} & \bm{0} \\ \bm{0} & \bm{M} \end{bmatrix}
		\begin{bmatrix}
			{\bm Y}_{\bm{m}} \\ {\bm V}_{\bm m}
		\end{bmatrix}, 
	\end{align}
	where we define 
	\begin{align}
		\Lambda_{\bm m} &:= \bm  \Lambda \cdot \bm{m},\\
		{\bm Y}_{\bm m} 
		&:= 
		\sum_{j=1}^M  \sum_{\substack{ \bm{u,k} \in \mathbb{N}^M \\ 1<u<m \\ \bm{u+k} - \bm{e}_j = \bm{m}   }}
		\big(
		\bm{M} \bm{\dot{w}}_{\bm u}
		\big)
		u_j R^j_{\bm k}
		+\bm{C}\bm{V}_{\bm m}
		- \sum_{\bm{n}\in \mathbb{N}^N} \bm{f}_{\bm n} H_{\bm{n},\bm{m}},
		\\
		{\bm V}_{\bm m} &:= 
		\sum_{j=1}^M  \sum_{\substack{ \bm{u,k} \in \mathbb{N}^M \\ 1<u<m \\ \bm{u+k} - \bm{e}_j = \bm{m}   }}
		\bm{w}_{\bm u} u_j R^j_{\bm k}.
	\end{align}
	We rewrite the second set of equations in system~\eqref{app_eq:2ndOrderSys_AutHighOrderInvEq} to obtain an explicit expression for the velocity parametrization as 
	\begin{align}
		\dot{\bm{w}}_{\bm m}
		= 
		\Lambda_{\bm m} \bm{w}_{\bm m}
		+
		\sum_{j=1}^M  \bm  \phi_j R^j_{\bm m}
		+  {\bm V}_{\bm m},
		\label{app_eq:Relation_Aut}
	\end{align}
	The velocity parametrization depends purely on multiplicative terms of the displacement parametrization and the reduced dynamics. Indeed, the expression~\eqref{app_eq:Relation_Aut} is identical to the general result of Lemma~\ref{lem:HomAutVel}, which we prove later.
	Now we utilize relation~\eqref{app_eq:Relation_Aut} to rewrite $\bm Y_ {\bm m}$ in terms of the displacement parametrization and substitute the resulting expression into the first set of equations in system~\eqref{app_eq:2ndOrderSys_AutHighOrderInvEq} to obtain
	\begin{align}\label{app_eq:Hom_aut}
		&\overbrace{  
			\bigg(
			\bm{K}  
			+ \Lambda_{\bm m} \bm{C} 
			+ (\Lambda_{\bm m})^2 \bm{M}
			\bigg)
		}^{=: \bm{L}_{\bm m}}
		\bm{w}_{\bm m}
		=
		\bm{D}_{\bm m}\bm{R}_{\bm m} + 
		\bm{C_m},
	\end{align}
	where $\bm{C_m} :=-\Lambda_{\bm m} \bm{M} {\bm V}_{\bm m}- {\bm Y}_{\bm m}$ and $\bm{D}_{\bm m}\in \mathbb{C}^{n\times M}$ is a matrix whose $j^{\mathrm{th}}$ column is given by $(\bm{D}_{\bm m})_{j} := -((\Lambda_{\bm m} + \lambda_j) \bm{M} +\bm{C})\bm  \phi_j$ for every $ j=1,\dots, M$.  This concludes the proof of Lemma~\ref{lem:HomAut}.
	
	Now we prove Lemma~\ref{lem:HomAutVel} by applying the chain rule to compute the velocity parametrization and using the fact that the autonomous reduced dynamics trajectories satisfy $\dot{\bm{p}} = \bm{R(p)}$. Finally, we collect the terms that sum up to the same multi-index $\bm m$ as
	\begin{align}
		\dot{\bm w}_{\bm m} = \frac{d \bm w}{dt} \Bigg|_{\bm m} = [D\bm w(\bm p) \dot{\bm p}]_{\bm m} 
		&= [D\bm w(\bm p) \bm R (\bm p)]_{\bm m} 
		\\
		\label{eq:aut_vel}&= \sum_{j=1}^M  \sum_{\substack{ \bm{u,k} \in \mathbb{N}^M \\  \bm{u+k} - \bm{e}_j = \bm{m}   }}
		\bm{w}_{\bm u} u_j R^j_{\bm k}
	\end{align}
	Indeed, equation~\eqref{eq:aut_vel} can be verified against the standard invariance equation given in eq.~\eqref{app_eq:Relation_Aut}.
	\paragraph*{Reduced dynamics}
	The linear system~\eqref{eq:Hom_Aut_Second} contains $n$ equations for each multi-index $\bm m$ with $n$ unknowns associated with the displacement parametrization $\bm w_{\bm m}$ and additional $M$ unknowns associated with the reduced dynamics parametrization $\bm R_{\bm m}$. This underdeterminacy of system~\eqref{eq:Hom_Aut_Second} allows different solution choices for the reduced dynamics parametrization. Specifically, a normal-form style parametrization is appealing when (near)-resonances are present among the master subspace eigenvalues. Indeed, for an eigenvalue $\lambda_i$ associated to the master subspace, a near-resonance of the form
	\begin{align}\label{app_eq:NearResonance}
		\lambda_i \approx \Lambda_{\bm m}
	\end{align}
	results in the coefficient matrix $\bm {L_m}$ being nearly singular. To ensure (robust) solvability of system~\eqref{eq:Hom_Aut_Second}, the coefficients $\bm R_{\bm m}$ must be chosen such that the right-hand side of system~\eqref{eq:Hom_Aut_Second} is orthogonal to the (near) kernel of $\bm {L_m}$. In particular, we require (see Jain and Haller\cite{Jain2021}) 
	\begin{align} 
		\bm  \theta^*_i
		\bigg(
		\bm{K}  
		+ \Lambda_{\bm m} \bm{C} 
		+ (\Lambda_{\bm m})^2 \bm{M}
		\bigg) \bm{w}_{\bm m}
		=
		0,
	\end{align}
	which directly provides us the essentially nonzero reduced dynamics coefficients $R^i_{\bm{m}}$ in the normal-form parametrization style as 
	\begin{align}\label{app_eq:RD_Aut_Second}
		\bm  \theta^*_i
		(  (\Lambda_{\bm m} + \lambda_i) \bm{M}   
		+\bm{C}  ) \bm\phi_i  
		R^i_{\bm m}
		= 
		-\bm  \theta^*_i 
		(
		\Lambda_{\bm m} \bm{M} {\bm V}_{\bm m}
		+ {\bm Y}_{\bm m}
		).
	\end{align}
	
	\subsubsection*{Non-autonomous SSM}
	Collecting terms at $\mathcal{O}(\epsilon)$ in the invariance equation~\eqref{eq:InvEq}  corresponding to a given monomial multi-index $\bm{m}$ and Fourier harmonic $\bm{\kappa}$, we obtain
	\begin{align}\label{app_eq:2ndOrderSytemNonAutinvEq}
		\bigg(
		\begin{bmatrix} -\bm{K} & \bm{0} \\ \bm{0} & \bm{M} \end{bmatrix}
		&-
		\Lambda_{\bm{m}, \bm  \kappa}
		\begin{bmatrix} \bm{C} & \bm{M} \\ \bm{M} & \bm{0} \end{bmatrix}
		\bigg) 
		\begin{bmatrix}
			\bm{x}_{\bm{m}, \bm  \kappa } \\ \bm{\dot{x}}_{\bm{m},\bm \kappa}
		\end{bmatrix}
		\nonumber \\ &=
		\sum_{j=1}^M
		\begin{bmatrix} \bm{C} & \bm{M} \\ \bm{M} & \bm{0} \end{bmatrix}
		\begin{bmatrix}
			\bm  \phi_j 
			\\
			\bm  \phi_j \lambda_j
		\end{bmatrix}
		S^j_{\bm{m},\bm \kappa}
		+
		\begin{bmatrix} \bm{1} & \bm{0} \\ \bm{0} & \bm{M} \end{bmatrix}
		\begin{bmatrix}
			{\bm Y}_{\bm{m},\bm \kappa} \\ {\bm V}_{\bm{m},\bm \kappa}
		\end{bmatrix},
	\end{align}
	where were we define 
	\begin{align}
		&\Lambda_{\bm{m}, \bm  \kappa} := \Lambda_{\bm m}  +i \langle \bm  \Omega , \bm  \kappa \rangle,\\
		&{\bm Y}_{\bm{m}, \bm  \kappa} 
		:=
		\bm{M}\sum_{j=1}^M \sum_{\substack{\bm{u}, \bm{k} \in           \mathbb{N}^M \\ \bm{u} + \bm{k} - \bm{e}_j =      \bm{m} \\ |\bm{k}| < m}}
		\big(    	   
		\bm{\dot{w}}_{\bm{u}}  u_j S^j_{\bm{k},\bm  \kappa}
		+     \bm{\dot{x}}_{\bm{k},\bm \kappa} k_j R^j_{\bm{u}}
		\big) 
		+ \bm{C}\bm{V}_{\bm{m}, \bm \kappa}   
		\nonumber \\
		&- \sum_{\bm{n}\in \mathbb{N}^N}          
		\bm{f}_{\bm{n}} 
		\sum_{j=1}^{2n}  n_j 
		\bigg(
		\sum_{\substack{\bm{u}, \bm{h} \in \mathbb{N}^M \\ \bm{u} + \bm{h} = \bm{m}}} 
		H_{\bm{n- \bm{e}_j}, \bm{h}} X^j_{\bm{u}, \bm  \kappa }
		\bigg)-
		\sum_{\bm{n}\in \mathbb{N}^N}
		\bm{g}_{\bm{n},\bm \kappa} 
		H_{\bm{n}, \bm{m}},
		\\
		&
		{\bm V}_{\bm{m}, \bm  \kappa} 
		:=
		\sum_{j=1}^M 
		\sum_{\substack{\bm{u}, \bm{k} \in \mathbb{N}^M \\ \bm{u} + \bm{k} - \bm{e}_j = \bm{m} \\ |\bm{k}| < m}}
		\bm{w}_{\bm{u}}
		u_j S^j_{\bm{k},\bm  \kappa}   
		+  
		\bm{x}_{\bm{k},\bm \kappa} k_j R^j_{\bm{u}}.
	\end{align}
	Analogous to the autonomous case (see eq.~\eqref{app_eq:Relation_Aut}), an explicit relation between displacement and velocity parametrization can be established from the invariance equation~\eqref{app_eq:2ndOrderSytemNonAutinvEq} as
	\begin{align} \label{app_eq:Hom_eq_nonaut_withvelocity}
		\bm{\dot{x}}_{\bm{m},\bm \kappa} 
		=
		\Lambda_{\bm{m}, \bm  \kappa}
		\bm{x}_{\bm{m}, \bm  \kappa }
		+
		\sum_{j=1}^M 
		\bm  \phi_j 
		S^j_{\bm{m},\bm \kappa}
		+  {\bm V}_{\bm{m}, \bm  \kappa},
	\end{align}
	which is identical to the general result of Lemma~\ref{Lem:HomNonautVel}, as we show later. We utilize relation~\eqref{app_eq:Hom_eq_nonaut_withvelocity} to rewrite ${\bm Y}_{\bm{m}, \bm \kappa}$ in terms of the displacement parametrization and substitute the resulting expression into the first set of equations in system~\eqref{app_eq:2ndOrderSytemNonAutinvEq} to obtain
	\begin{align}\label{app_eq:HomolEq_nonaut_2ndorder}
		&\overbrace{
			\big(
			\bm{K} 
			+
			\Lambda_{\bm{m}, \bm  \kappa}
			\bm{C} 
			+
			(\Lambda_{\bm{m}, \bm  \kappa} )^2
			\bm{M}
			\big)
		}^{\bm{L}_{\bm{m},\bm  \kappa}}
		\bm{x}_{\bm{m}, \bm  \kappa }
		=
		\bm{D}_{\bm{m},\bm  \kappa}\bm{S}_{\bm{m},\bm \kappa}
		+ 
		\bm{C}_{\bm{m}, \bm  \kappa},
	\end{align}
	where $\bm{C}_{\bm{m}, \bm  \kappa}=- \Lambda_{\bm{m}, \bm  \kappa}\bm{M} {\bm V}_{\bm{m}, \bm  \kappa} - {\bm Y}_{\bm{m}, \bm \kappa}$
	and $\bm{D}_{\bm{m},\bm  \kappa}\in \mathbb{C}^{n\times M}$ is a matrix whose $j^{\mathrm{th}}$ column is given by $(\bm{D}_{\bm{m},\bm  \kappa})_{j} :=-(\bm{C}  +\bm{M}(\lambda_j +\Lambda_{\bm{m}, \bm  \kappa}) )\bm  \phi_j$ for every $ j=1,\dots, M$.  This concludes the proof of Lemma~\ref{Lem:HomNonaut}.
	
	For proving Lemma~\ref{Lem:HomNonautVel}, we use the chain rule using the fact that the non-autonomous reduced dynamics trajectories satisfy $\dot{\bm{p}} = \bm{R}_{\epsilon}(\bm p,\phi)$. To obtain an expression for $\dot{\bm{x}}$, we collect $\mathcal{O}(\epsilon)$ terms in the time-derivative of the displacement parametrization of the SSM as 
	\begin{align}
		\dot{\bm{x}}(\bm{p},\boldsymbol\phi) &= \frac{d \bm{w}_{\epsilon}}{dt}  \Bigg|_{\epsilon} =  
		D\bm{w}(\bm{p}) \cdot \bm S +  D_{\bm{p}}\bm x \cdot \bm R + D_{\bm{\phi}}\bm x \cdot \bm \Omega .
	\end{align}    
	Finally, we restrict this expression to the multi-index $\bm{m}$ and harmonic $\bm{\kappa}$ to obtain
	\begin{align}
		\label{eq:vel_exp}
		\dot{\bm x}_{\bm{m}, \bm  \kappa} 
		=
		\sum_{j=1}^M 
		\sum_{\substack{\bm{u}, \bm{k} \in \mathbb{N}^M \\ \bm{u} + \bm{k} - \bm{e}_j = \bm{m} }}
		\bm{w}_{\bm{u}}
		u_j S^j_{\bm{k},\bm  \kappa}   
		+  
		\bm{x}_{\bm{k},\bm \kappa} k_j R^j_{\bm{u}}.
	\end{align}
	Indeed, the expression in equation~\eqref{eq:vel_exp} is identical to the one based on invariance equation computation given in \eqref{app_eq:Hom_eq_nonaut_withvelocity}.

	\paragraph*{Nonautonomous reduced dynamics}
	Similarly to the autonomous case, the underdeterminacy of the linear system~\eqref{app_eq:HomolEq_nonaut_2ndorder} allows for different solution choices for the reduced dynamics parametrization.
	Specifically, a normal-form style parametrization is appealing when (near)-resonances are present among the master subspace eigenvalues and a given harmonic of the forcing frequency. Indeed, for an eigenvalue $\lambda_i$ associated to the master subspace, a near-resonance of the form
	\begin{align}\label{app_eq:ResCond_nonaut}
		\lambda_i \approx \Lambda_{\bm{m},\bm  \kappa}
	\end{align}
	results in the coefficient matrix ${\bm{L}_{\bm{m},\bm  \kappa}}$ being nearly singular. To ensure (robust) solvability of system~\eqref{app_eq:HomolEq_nonaut_2ndorder}, the coefficients $\bm{S}_{\bm{m},\bm \kappa}$ must be chosen such that right-hand side of system~\eqref{app_eq:HomolEq_nonaut_2ndorder} is orthogonal to the (near) kernel of ${\bm{L}_{\bm{m},\bm  \kappa}}$. In particular, we require (cf.~Jain and Haller\cite{Jain2021}) 
	\begin{align}
		(\bm  \theta_i)^*
		\big(
		\bm{K} 
		+
		\Lambda_{\bm{m}, \bm  \kappa}
		\bm{C} 
		+
		(\Lambda_{\bm{m}, \bm  \kappa} )^2
		\bm{M}
		\big)
		\bm{x}_{\bm{m}, \bm  \kappa }
		= 0,
	\end{align}
	which directly provides the essentially non-zero reduced dynamics coefficients $S^i_{\bm{m},\bm \kappa}$ in the normal-form parametrization style as
	\begin{align}
		(\bm  \theta_i)^*
		(
		\bm{C}  
		&+
		\bm{M}
		(\lambda_i 
		+
		\Lambda_{\bm{m}, \bm  \kappa}
		)
		)
		\bm  \phi_i S^i_{\bm{m},\bm \kappa}
		\nonumber \\ \label{eq:normal_form_nonauto}&= 
		- (\bm  \theta_i)^* (
		\Lambda_{\bm{m}, \bm  \kappa}
		\bm{M} {\bm V}_{\bm{m}, \bm  \kappa} 		
		+ {\bm Y}_{\bm{m}, \bm  \kappa}
		)
	\end{align}
	\subsection{First-order computation}\label{app_sec:SSMComp_1storder}
	For general first-order dynamical systems the equations do not exhibit the internal structure~\eqref{eq:SSM_param_struct} seen for second-order mechanical systems (Lemmas~\ref{lem:HomAutVel} \& \ref{Lem:HomNonautVel}). Hence, we develop general expressions for SSM computation in the first-order system~\eqref{eq:DS_FirstOrder}. Once again, we operate in physical coordinates using the multi-index notation.  
	
	\subsubsection*{Autonomous SSM}
	We substitute the autonomous first-order SSM expansion
	\begin{align}\label{app_eq:AutSSMParam}
		\bm{W}(\bm{p}) 
		&=  \sum_{\bm{m}\in \mathbb{N}^M} \bm{W}_{\bm{m}} \bm{p}^{\bm m},
	\end{align}
	where $\bm{W}_{\bm m} \in \mathbb{C}^N$, and its reduced dynamics expansion~\eqref{eq:AutRDParam} into the autonomous invariance equation~\eqref{eq:InvEq_aut} to obtain
	\begin{align}
		\bm{B} (\text{D}_{\bm p} \sum_{\bm{u}\in \mathbb{N}^M} \bm{W}_{\bm{u}} &\bm{p}^{\bm u} ) 
		\sum_{\bm{k}\in \mathbb{N}^M}\bm{R}_{\bm{k}} \bm{p}^{\bm k} 
		\nonumber = 
		\bm{A}\sum_{\bm{m}\in \mathbb{N}^M}\bm{W}_{\bm{m}} \bm{p}^{\bm m} 
		\\ \label{eq:inv_exp_first} 
		&+ \sum_{\bm{n}\in \mathbb{N}^N}  \bm{F}_{\bm{n}} (\sum_{\bm{u}\in \mathbb{N}^M}\bm{W}_{\bm{u}} \bm{p}^{\bm u})^{\bm n},
	\end{align}
	where $\bm{F}_{\bm{n}}\in\mathbb{R}^N$ are the coefficients in the nonlinearity expansion~\eqref{eq:expF} in the first-order form. Collecting coefficients of every unique monomial in system~\eqref{eq:inv_exp_first}, we obtain a linear system for each monomial multi-index $\bm m$, as we detail below.
	At the leading-order, i.e, for monomial multi-index $\bm{p}^{\bm{e}_i}$, we obtain
	\begin{align}\label{eq:leading_eig_first}
		\bm{B} (\sum_{j=1}^M \bm{W}_{\bm{e}_j} ) 
		\bm{R}_{\bm{e}_i} 
		=
		\bm{A} \bm{W}_{\bm{e}_i}.
	\end{align}
	From the eigenvalue problem~\eqref{eq:LeftEigenvalue}, we see that the choice $\bm{W}_{\bm{e}_i} = \bm{v}_i$ and $R^j_{\bm{e}_j} = \delta_{ij} \lambda_i$ solves equation~\eqref{eq:leading_eig_first}. It is noteworthy that due to this diagonal choice of the leading-order autonomous reduced dynamics, the equations~\eqref{eq:inv_exp_first} get decoupled between distinct multi-indices at each order. Hence, collecting the coefficients of higher-order terms in system~\eqref{eq:inv_exp_first} for a multi-index ${\bm m}$, we obtain
	\begin{align}\label{app_eq:Hom_Aut_First}
		\underbrace{
			\bigg(\bm{A}
			-
			\Lambda_{\bm{m}} \bm{B}\bigg)
		}_{:= \bm{\mathcal{L}}_{\bm{m}}}
		\bm{W}_{\bm m}
		=
		\sum_{j=1}^M \bm{B} \bm{v}_j R^j_{\bm m}
		+
		\bm{\mathcal{C}}_{\bm m},
	\end{align}
	where we define 
	\begin{align}
		\bm{\mathcal{C}}_{\bm m} &:=  
		\sum_{j=1}^M  \sum_{\substack{\bm{u}, \bm{k} \in \mathbb{N}^M \\ \bm{u} + \bm{k} - \bm{e}_j = \bm{m} \\ 1 < |\bm{u}| < m }} \bm{B}\bm{W}_{\bm{u}} u_j R^j_{\bm{k}}
		-\sum_{\bm{n}\in \mathbb{N}^N}  \bm{F}_{\bm{n}}  H_{\bm{n}, \bm{m}}.
	\end{align}
	%
	%
	
	\paragraph*{Autonomous reduced dynamics}
	Analogous to the discussion for second-order systems, we observe that linear system~\eqref{app_eq:Hom_Aut_First} is underdetermined, which allows for choices of parametrizations for the reduced dynamics. Indeed, for an eigenvalue $\lambda_i$ associated to the master subspace, a near-resonance of the form~\eqref{app_eq:NearResonance}
	results in the coefficient matrix $\bm {\mathcal{L}_m}$ being nearly singular. To ensure the solvability of the system~\eqref{app_eq:Hom_Aut_First}, we require (cf. Jain and Haller\cite{Jain2021}) 
	\begin{align}
		(\bm{u}_i)^* \bm{\mathcal{L}}_{\bm m} \bm{W}_{\bm m} = \bm{0},
	\end{align}
	which provides the nonzero reduced dynamics coefficients $\bm{R}^i_{\bm{m}}$ in the normal-form parametrization style as 
	\begin{align}
		R^i_{\bm m} = -(\bm{u}_i)^* \bm{\mathcal{C}}_{\bm m},
	\end{align}
	where we have chosen a normalization for the eigenvectors that satisfies $(\bm{u}_i)^* \bm{B} \bm{v}_j = \delta_{ij}$ for every $i,j = 1,\dots,M$.
	\subsubsection*{Non-autonomous SSM}
	To derive the non-autonomous invariance equation in first-order form, we expand the $\mathcal{O}(\epsilon)$-terms in the non-autonomous SSM parametrization~\eqref{eq:ExpandW_e1} as
	\begin{align}\label{eq:NonAutSSMexp}
		\bm{X}(\bm{p},\bm\phi) 
		&=  \sum_{\bm{m}\in \mathbb{N}^M} \sum_{\bm\kappa \in \mathbb{Z}^k} \bm{X}_{\bm{m},\bm\kappa } e^{i \langle \bm\kappa , \bm\phi  \rangle} \bm{p}^{\bm m},
	\end{align}
	where $\bm{X}_{\bm{m},\bm\kappa } \in \mathbb{C}^N$. Substituting the expansions~\eqref{eq:NonAutSSMexp} and~\eqref{eq:expandNonAutRD} into the nonautonomous invariance equation~\eqref{eq:InvEq_nonaut}, and collecting terms corresponding to any given Fourier-harmonic $\bm  \kappa$, we obtain
	\begin{widetext}
		\begin{align} \nonumber
			\bm{B}
			\bigg( 
			\text{D}_{\bm p}( \sum_{\bm{u}\in \mathbb{N}^M} \bm{W}_{\bm{u}} \bm{p}^{\bm u})\sum_{\bm{k}\in \mathbb{N}^M} \bm{S}_{\bm{k},\bm  \kappa} \bm{p}^{\bm k}  + 
			(\partial_{\bm p} \sum_{\bm{u}\in \mathbb{N}^M} \bm{X}_{\bm{u},\bm  \kappa}  \bm{p}^{\bm u}) \sum_{\bm{k}\in \mathbb{N}^M} \bm{R}_{\bm{k}} \bm{p}^{\bm k}
			+
			i \langle \bm  \kappa, \bm{\Omega } \rangle\sum_{\bm{m}\in \mathbb{N}^M} \bm{X}_{\bm{m},\bm  \kappa} \bm{p}^{\bm m}  
			\bigg)
			=
			\\ \label{eq:nonAuto_Kappa}
			\bm{A}\sum_{\bm{m}\in \mathbb{N}^M} \bm{X}_{\bm{m},\bm  \kappa}\bm{p}^{\bm m} 
			+	
			\big[\text{D}_{\bm z} \bm{F}_{\bm{n}} (\sum_{\bm{u}\in \mathbb{N}^M}\bm{W}_{\bm{u}} \bm{p}^{\bm u})^{\bm n}  \big]\sum_{\bm{m}\in \mathbb{N}^M} \bm{X}_{\bm{m},\bm  \kappa} \bm{p}^{\bm m}
			+
			\sum_{\bm{n}\in \mathbb{N}^N}
			\sum_{\bm{m} \in \mathbb{N}^M}
			\bm{G}_{\bm{n} ,\bm  \kappa}
			H_{\bm{n}, \bm{m}},
		\end{align}
	\end{widetext}
	where $\bm{G}_{\bm{n} ,\bm  \kappa}\in\mathbb{C}^N$ are the coefficients in the forcing expansion~\eqref{app_eq:Force_expand_FirstOrder} in the first-order form.  
	Finally, collecting the coefficients corresponding to any given monomial multi-index $\bm{m}$ in system~\eqref{eq:nonAuto_Kappa}, we obtain (cf. Poinsioen et al.\cite{Ponsioen2020})
	%
	%
	
	%
	\begin{align}
		\label{app_eq:Hom_NonAut_First}
		\underbrace{
			\bigg( 
			\bm{A}
			-
			( \Lambda_{\bm m} +  i \langle \bm  \kappa, \bm  \Omega \rangle)\bm{B} 
			\bigg)
		}_{:=\bm{\mathcal{L}}_{\bm{m},\bm \kappa} }
		\bm{X}_{\bm{m},\bm  \kappa} 
		=
		\sum_{j=1}^M 
		\bm{B} 
		\bm{v}_j S^j_{\bm{m},\bm  \kappa}  
		+
		\bm{\mathcal{C}}_{\bm{m}, \bm  \kappa},
	\end{align}
	where
	\begin{align}
		\bm{\mathcal{C}}_{\bm{m}, \bm  \kappa} &:=
		\bm{B} 
		\bigg( 
		\sum_{j=1}^M \sum_{\substack{\bm{u}, \bm{k} \in \mathbb{N}^M \\ \bm{u} + \bm{k} - \bm{e}_j = \bm{m} \\ |\bm{k}| < m}}
		\bm{W}_{\bm{u}}                       u_j S^j_{\bm{k},\bm  \kappa}   
		\nonumber \\ &+  
		\sum_{j=1}^M \sum_{\substack{\bm{u}, \bm{k} \in \mathbb{N}^M \\ \bm{u} + \bm{k} - \bm{e}_j = \bm{m} \\ |\bm{u}| < m }}
		\bm{X}_{\bm{u},\bm  \kappa}    u_j R^j_{\bm{k}} 
		\bigg)
		-
		\sum_{\bm{n}\in \mathbb{N}^N}
		\bm{G}_{\bm{n} ,\bm  \kappa}
		H_{\bm{n}, \bm{m}} 
		\nonumber \\
		&-  \sum_{\bm{n}\in \mathbb{N}^N}          
		\bm{F}_{\bm{n}}
		\sum_{j=1}^{2n}  n_j 
		\bigg(
		\sum_{\substack{\bm{u}, \bm{h} \in \mathbb{N}^M \\ \bm{u} + \bm{h} = \bm{m}}} 
		H_{\bm{n- \bm{e}_j}, \bm{h}} X^j_{\bm{u}, \bm  \kappa }
		\bigg).
	\end{align}
	Recursively solving the linear system~\eqref{app_eq:Hom_NonAut_First} provides us the coefficients $\bm{X}_{\bm{m},\bm  \kappa},\bm{S}_{\bm{m},\bm  \kappa}$ defining the $\mathcal{O}(\epsilon)$-terms in the parametrizations for the nonautonomous SSM and its reduced dynamics. At the leading order, system~\eqref{app_eq:Hom_NonAut_First} simplifies as (cf. Jain and Haller~\cite{Jain2021})
	\begin{align}\label{app_eq:InvEqOrder0}
		\bigg( 
		\bm{A}
		- i \langle \bm  \kappa, \bm  \Omega \rangle\bm{B} 
		\bigg)
		\bm{X}_{\bm{0},\bm  \kappa} 
		=
		\sum_{j=1}^M 
		\bm{B} 
		\bm{v}_j S^j_{\bm{0},\bm  \kappa}  
		+
		\bm{G}_{\bm{0}, \bm  \kappa}.
	\end{align}
	At any higher order $ m=|\bm{m}|>1$, we have $z_m = \binom{m+M-1}{M-1}$ distinct multi-indices. An important consequence of choosing the linear autonomous reduced dynamics to be diagonal, i.e., $R^i_{\bm{e}_j} = \delta_{ij} \lambda_j$, is that the system~\eqref{app_eq:Hom_NonAut_First} gets decoupled between distinct multi-indices at any given order $m$ and for any given harmonic $\bm \kappa$. Thus, instead of solving $z_m \times N$ coupled linear equations in general, we need to solve only $z_m$ linear systems of dimension $N$ each.
	
	%
	%
	
	%
	\paragraph*{Reduced Dynamics}
	Analogous to the discussion for second-order nonautonomous setting, we observe that the linear system~\eqref{app_eq:Hom_NonAut_First} is underdetermined, which allows for choices of parametrizations for the reduced dynamics. Indeed, for an eigenvalue $\lambda_i$ associated with the master subspace, a near-resonance of the form~\eqref{eq:ResCond_nonaut}
	results in the coefficient matrix ${\bm{\mathcal{L}}_{\bm{m},\bm \kappa}}$ being nearly singular.  
	
	To ensure (robust) solvability of system~\eqref{app_eq:Hom_NonAut_First}, the coefficients coefficients $\bm{S}_{\bm{m},\bm \kappa}$ must be chosen such that right-hand side of system~\eqref{app_eq:HomolEq_nonaut_2ndorder} is orthogonal to the (near) kernel of ${\bm{\mathcal{L}}_{\bm{m},\bm  \kappa}}$. In particular, we require 
	\begin{align}
		(\bm{u}_i)^* \bm{\mathcal{L}}_{\bm{m},\bm \kappa} \bm{X}_{\bm{m},\bm  \kappa} = \bm{0},
	\end{align}
	which directly provides us the essentially non-zero reduced dynamics coefficients $S^i_{\bm{m},\bm \kappa}$ in the normal-form parametrization style as 
	\begin{align}
		S^i_{\bm{m},\bm  \kappa} = (\bm{u}_i)^* \bm{\mathcal{C}}_{\bm{m},\bm  \kappa}.
	\end{align}
	Specifically, at the leading-order, we have $S^i_{\bm{0},\bm  \kappa} = (\bm{u}_i)^* \bm{G}_{\bm{0}, \bm  \kappa}$ (cf. Jain and Haller\cite{Jain2021}) . 
	\section{Proof of Lemma \ref{lem:FRC} } \label{App:FRC_Lemma}
	In this section we provide the proofs for the statements in Lemma \ref{lem:FRC}. We use a polar transformation of the parametrization coordinates $\bm{p} = [p, \conj{p}] = [\rho e^{i\theta}, \rho e^{-i\theta}]$. Thus, $\dot{p} = (\dot{\rho} + \rho i \dot{\theta}) e^{i\theta}$. Specifically, the reduced dynamics~\eqref{eq:red_dyn_expansion} over the two-dimensional nonautonomous SSM ($M=2$) transforms as
	\begin{align}
		(\dot{\rho} + \rho i \dot{\theta})
		&= 
		\sum_{\bm{m}\in \mathbb{N}^l} R^1_{\bm{m}}    (\rho e^{i\theta})^{m_1} (\rho e^{-i\theta})^{m_2} e^{-i\theta}      
		\nonumber \\    
		& \ \ \ \ \ +   \epsilon 
		\sum_{\substack{\bm{m}\in \mathbb{N}^l \\ \kappa \in \mathbb{Z} }}  S^1_{\bm{m},\kappa} e^{i \kappa \phi}  (\rho e^{i\theta})^{m_1} (\rho e^{-i\theta})^{m_2} e^{-i\theta}
		\\ \label{eq:red_polar2}
		&= 
		\sum_{\bm{m}\in \mathbb{N}^l} R^1_{\bm{m}}    \rho^{|\bm{m}|} e^{i\theta(m_1-m_2-1)}
		\nonumber \\     
		& \ \ \ \ \ + 
		\epsilon 
		\sum_{\substack{\bm{m}\in \mathbb{N}^l \\ \kappa \in \mathbb{Z}^* }}  S^1_{\bm{m},\kappa }   \rho^{|\bm{m}|} e^{i	\kappa \phi}  e^{i\theta(m_1-m_2-1)}.
	\end{align}
	Choosing a normal-form style of parametrization~\eqref{app_eq:RD_Aut_Second} for the autonomous reduced dynamics, we obtain nonzero coefficients $R^1_{\bm m}$ only for multi-indices that satisfy the near resonance condition~\eqref{eq:ResCond_aut}. Hence, we have $m_1 - m_2 -1 = 0 $ for every nonzero $R^1_{\bm m}$. Thus, the angular part vanishes for all autonomous terms in eq.~\ref{eq:red_polar2}.
	
	Furthermore, choosing a normal-form style of parametrization~\eqref{eq:normal_form_nonauto} for the nonautonomous reduced dynamics, we obtain nonzero coefficients $S^1_{\bm{m}, {\kappa}}$ only for the pairs $(\bm{m},{\kappa})$ that satisfy the near-resonance condition~\eqref{eq:ResCond_nonaut}. But under the external near-resonance condition~\eqref{eq:ResCond_FRC}, for any pair ($\bm{m}$,$\kappa$) that fulfills the near-resonance conditions~\eqref{eq:ResCond_nonaut}, we have $\kappa = (1 - m_1 + m_2)\kappa_0$. 
	This allows us to rewrite $e^{i\kappa \phi} e^{i\theta(m_1-m_2-1)} = 
	e^{i(\kappa/\kappa_0) (\kappa_0\phi - \theta)}$ in eq.~\eqref{eq:red_polar2}. Now, introducing the phase shift $\psi = \theta - \kappa_0 \phi$ in eq.~\eqref{eq:red_polar2} and using the fact that $\dot{\phi} = \Omega$, we obtain 
	\begin{align}
		(\dot{\rho} + \rho i (\dot{\psi} + \kappa_0\Omega ))
		=
		\sum_{\bm{m}\in \mathbb{N}^l} R^1_{\bm{m}} &   \rho^{|\bm{m}|}
		\nonumber \\ \label{eq:polar_red_complex}
		+ 
		\epsilon 
		S^1_{\bm{0},\kappa_0 }  e^{-i\psi}
		&+ 
		\epsilon 
		\sum_{\substack{\bm{m}\in \mathbb{N}^l \\ \kappa \in \mathbb{Z}^* }}  S^1_{\bm{m},\kappa }  e^{-i (\kappa/\kappa_0) \psi}.
	\end{align}
	Finally, separating the real and imaginary parts of eq.~\eqref{eq:polar_red_complex} results in Eq.~\eqref{eq:PolarDynamics}, which concludes the proof of statement~\ref{lem:i_PolDyn}.
	The zeros of the function
	\begin{align}
		\bm r (\rho, \psi, \Omega) = \bm 0     
	\end{align}
	correspond to fixed points of the polar dynamics~\eqref{eq:PolarDynamics}. For the original non-autonomous reduced dynamics~\eqref{eq:red_dyn_expansion}, such a fixed point represents a trajectory with a constant phase shift $\psi$ relative to the forcing as well as a constant polar amplitude $\rho$. Hence, this fixed point describes a periodic orbit of the reduced dynamics. As the SSM parametrization maps the orbits of the reduced system onto the orbits of the full system, we obtain a periodic steady-state for the full system which concludes the proof of statement~\ref{lem:ii_FRC}.
	For computing the stability of these fixed points, we rewrite the polar dynamics~\eqref{eq:polar_red_complex} in the form 
	\begin{align}
		\begin{bmatrix}
			\dot{\rho} \\ \dot{\psi}
		\end{bmatrix}
		=
		\begin{bmatrix}
			a(\rho) \\ b(\rho) / \rho
		\end{bmatrix}
		&+ 
		\epsilon \bigg( 
		\frac{\bm{Q}(\psi) }{\rho}
		\begin{bmatrix} \Real (S^1_{\bm{0},\kappa_0})  \\ \Imag (S^1_{\bm{0},\kappa_0})  \end{bmatrix} 
		\nonumber \\ \label{eq:jacobian_polar}
		+ \sum_{ \substack{ \bm{m} \in \mathbb{N}^2 \\ \kappa\in \mathbb{Z} }} 
		&\rho^{|\bm{m}|-1} \bm{Q}((\kappa/\kappa_0)\psi) 
		\begin{bmatrix} \Real (S^1_{\bm{m},\kappa})  \\ \Imag(S^1_{\bm{m},\kappa})  \end{bmatrix} 
		\bigg).
	\end{align}
	Evaluating the Jacobian of the right-hand side of eq.~\eqref{eq:jacobian_polar} at a hyperbolic fixed point $(\rho, \psi)$ concludes its stability via linearization. This proves statement~\ref{lem:iii_FRC_stab}.

	\section{Details on the continuation routine}\label{App:COCO}
	
	Within SSMTool 2.4 as well as for the full system analysis, we make use of the \texttt{po} toolbox provided by the continuation toolbox \textsc{coco}~\cite{Dankowicz2013}. In this section, we give details on performing parameter continuation over the SSM-based ROMs. We also list the continuation parameters used for each example in this paper. 
	\subsection{Sensitivity coefficients }  
	For an efficient numerical continuation of periodic orbits in \textsc{coco}, the Jacobian of the reduced vector field with respect to the continuation parameters $\epsilon$ and $\Omega$ and with respect to the parametrization coordinates $\bm p$ should be provided.
	%
	%
	The Jacobian with respect to the parametrization coordinates is straightforward, as the vector field is polynomial in $\bm p$. For the derivative with respect to the forcing amplitude $\epsilon$, only the $\mathcal{O}(\epsilon)$-terms need to be considered, which yields 
	\begin{align}
		\partial_\epsilon \bm{R}_\epsilon (\bm{p},  \phi) =  \bm{S} (\bm{p}, \phi).
	\end{align}
	As the nonautonomous coefficients are obtained by solving system~\eqref{eq:HomNonAut_Second}, they implicitly depend on the forcing frequency, i.e.,  $\bm{S}_{\mathbf{m},\kappa } = \bm{S}_{\mathbf{m},\kappa }(\Omega)$. Accounting for this frequency-dependence of the reduced dynamics coefficients, we write the derivative with respect to $\Omega$ as
	\begin{align}\label{app_eq:ContOmDepRD}
		\partial_\Omega \bm{R}_\epsilon (\bm{p}, \boldsymbol \phi) 
		&=
		\epsilon
		\sum_{\bm{m}\in \mathbb{N}^M} 
		\sum_{\kappa \in \mathbb{Z}}
		(\partial_\Omega\bm{S}_{\bm{m},\kappa } e^{i \kappa \Omega t
		}
		+
		i \kappa t \bm{S}_{\bm{m},\kappa } e^{i \kappa \Omega t
		}
		)
		{\bm p}^{\bm m},
	\end{align}
	where the computation of $\partial_\Omega \bm{S}_{\bm{m},\kappa }$ is not straightforward and we choose to ignore these coefficients in Eq.~\eqref{app_eq:ContOmDepRD}. This choice may lead to an inaccurate computation of the Jacobian, which may reduce the rate of convergence to the solution during numerical continuation. However, this choice does not compromise the solution accuracy, which is controlled by separate numerical tolerances (see Table~\ref{tab:CocoParametersDescription}). These tolerances are related to the solution of the invariance equations and to the residual in the Newton iteration, which we evaluate up to machine precision. Opreni et al.~\cite{Opreni2022} extend this choice to the SSM and its reduced dynamics parametrizations by assuming the coefficients $\bm{W}_{\bm{m},\kappa}, \bm{S}_{\bm{m},\kappa}$ to be independent of the forcing frequency $\Omega$, and computing them only once across the entire frequency range. Such an assumption cannot guarantee solution accuracy, but offers additional computational advantages. Hence, in contrast to Opreni et al.~\cite{Opreni2022}, we choose to solve the nonautonomous invariance equations for each value of the forcing frequency here.
	\subsection{Period-doubling bifurcations and stability boundary}
	The \texttt{po}-toolbox of \textsc{coco}, which is used for periodic-orbit continuation in our SSM-based ROMs as well as in the full system, uses Floquet multipliers to determine the stability of a periodic orbit and detect bifurcations (if any). Specifically, a change in the stability of a periodic orbit occurs when a Floquet multiplier crosses the unit circle in the complex plane, resulting in a period-doubling bifurcation~\cite{Dankowicz2013,MingwuPart2}. In a system subject to periodic, parametric excitation, switching to a branch of period-doubling bifurcations and performing continuation along this branch using \textsc{coco} then provides us the stability boundary for the trivial periodic response. 
	%
	%
	\subsection{Parameters for collocation}
	
	The parameters for the continuation of the SSM-based ROMs are presented in TABLE~\ref{tab:CocoParametersFRCSSM}. TABLE~\ref{tab:CocoParametersFRC} lists the collocation and continuation parameters used for various simulations of the full system. Note that since the level set method of SSMTool 2.4 does not require continuation and relies solely on the evaluation of the zero level set of a function, there are no continuation parameters listed for the examples in Section~\ref{sec:ExamplesFRC}. TABLE~\ref{tab:CocoParametersDescription} gives a short description of each parameter along with its default value in \textsc{coco}.
	\begin{table}
		\centering
		\caption{Parameters for continuation with \texttt{po} \cite{Dankowicz2013}}
		\begin{tabular}{ l r l } 
			\hline
			Parameter & Default & Description \\ [0.5ex] 
			\hline\hline
			\texttt{NAdapt}      & 0   & Period for refinement of coll. mesh  \\
			\texttt{h\_0}       & 0.1 & Initial continuation stepsize \\
			\texttt{h\_max}     & 0.5 & Max. continuation stepsize  \\
			\texttt{h\_min}      & 0.01& Min. continuation stepsize  \\
			\texttt{h\_fac\_max} & 2   & Max. factor for apdapting stepsize  \\
			\texttt{h\_fac\_min} & 0.5 & Min. factor for adapting stepsize  \\ 
			\texttt{MaxRes}      & 0.1 & Maximal residual norm  \\
			\texttt{bi\_direct}  & 1   & Bidirectional continuation (boolean)\\
			\texttt{PtMX}        & 100 & Max. number of continuation steps\\
			\texttt{al\_max}     & 7   & Max. angle btw. consecutive tangents \\
			\texttt{ItMX}        & 10  & Iterations before adapting stepsize \\
			\texttt{TOL}         & 1e-6& Tolerance for Newton Iteration  \\
			\texttt{NTST}        & 10  & Sub-intervals in coll. discretization\\
			\texttt{NCOL}        & 4   & Number of collocation points  \\
			\hline
			\label{tab:CocoParametersDescription}
		\end{tabular}
	\end{table}
	\begin{table}
		\centering
		\caption{Collocation and continuation parameters set in the \texttt{po}-toolbox of \textsc{coco} for the SSM-based ROMs in different examples.}
		\begin{tabular}{l r r r r r r} 
			\hline
			Figure & \ref{fig:StabDiag_2DMathieu_SD} & \ref{fig:StabDiag_2DMathieu_FRC} & \ref{fig:BB_SD} & \ref{fig:BB_FRC} & \ref{fig:PB_SD} & \ref{fig:PB_FRC} \\ 
			\hline
			Parameter & Value  \\ [0.5ex] 
			\hline\hline
			\texttt{NAdapt}      & 0   & 0   & 0   & 0   & 0   & 1    \\
			\texttt{h\_0}        & 0.1 & 1e-4& 0.1 & 1e-4& 0.1 & 1e-4  \\
			\texttt{h\_max}      & 0.5 & 0.5 & 0.5 & 0.5 & 0.5 & 50  \\
			\texttt{h\_min}      & 0.01& 1e-4& 0.01& 1e-4& 0.01& 1e-4 \\
			\texttt{bi\_direct}  & 1   & 0   & 1   & 0   & 1   & 0    \\
			\texttt{PtMX}        & 50  & 35  & 50  & 40  & 50  & 100  \\
			\hline
			\label{tab:CocoParametersFRCSSM}
		\end{tabular}
	\end{table}
	\begin{table}
		\centering
		\caption{Collocation and continuation parameters set in the \texttt{po}-toolbox of \textsc{coco} for the full system in different examples.}
		\begin{tabular}{l r r r r r r r r r r r r r r} 
			\hline
			Figure & \ref{fig:StabDiag_2DMathieu_SD} & \ref{fig:StabDiag_2DMathieu_FRC} & \ref{fig:BB_SD} & \ref{fig:BB_FRC} & \ref{fig:PB_SD} & \ref{fig:PB_FRC} & \ref{fig:SelfExcite_FRC} & \ref{fig:BB_external} & \ref{fig:BB_external} Isola \\ 
			\hline
			Parameter& Value    \\ [0.5ex] 
			\hline\hline
			\texttt{NAdapt}      & 0   & 0   & 0   & 1   & 0   & 1   & 1               & 1   & 1   \\
			\texttt{h\_0}        & 0.1 & 1e-4& 0.1 & 0.1 & 0.1 & 1e-4& 0.1           & 0.1 & 0.1 \\
			\texttt{h\_max}      & 0.5 & 0.5 & 0.5 & 0.5 & 0.5 & 50  & 0.5           & 0.5 & 0.5 \\
			\texttt{MaxRes}      & 0.1 & 0.1 & 0.1 & 0.1 & 0.1 & 0.1 & 0.1          & 0.1 & 0.1\\
			\texttt{bi\_direct}  & 1   & 0   & 1   & 0   & 1   & 0   & 1               & 1   & 0   \\
			\texttt{PtMX}        & 10  & 60  & 100 & 60  & 100 & 500 & 350            & 200 & 90   \\
			\texttt{ItMX}        & 10  & 15  & 10  & 15  & 10  & 15  & 10              & 10  & 10   \\
			\texttt{NTST}        & 10  & 30  & 10  & 30  & 10  & 70  & 70             & 20  & 80   \\
			\hline
			\label{tab:CocoParametersFRC}
		\end{tabular}
	\end{table}
	\section{Modal equations for the prismatic beam}\label{App:PrismaticBeam}
	
	In this section, we derive the ODEs governing the evolution of the modal coordinates of a clamped-hinged prismatic beam as a function of the physical parameters, i.e., elastic modulus $E$, density $\rho$, radius of gyration $r$, area $A$, and moment of inertia of the cross section $I$, and a characteristic length $L$. The characteristic length may be the length of the beam or a typical wavelength for oscillatory modes of the beam. We generally assume that the beam length is obtained by multiplying the characteristic length $L$ by a dimensionless constant $l$ as $lL$. At the hinged end, the beam is subject to an axial load $p_a (t)$. The PDE governing the behaviour of the transverse displacement~$w$ of the beam in a nondimensional form is given as~\cite{Nayfeh1974NonlinearElements, Barari2011NonlinearBeams}
	\begin{align}\label{eq:PDEPB}
		\frac{\partial^4 w}{\partial y^4} + \frac{\partial^2 w}{\partial t^2}
		&=  +\epsilon 
		\big(H \frac{\partial^2w}{\partial y^2}  - 2c \frac{\partial w}{\partial t}
		p_t(y,t)
		-
		p_a(t) \frac{\partial^2w}{\partial y^2}   
		\big),\\
		w(0) &= w''(0) = w(l) = w'(l) = 0,
	\end{align}
	where the following dimensionless quantities are used~\cite{Nayfeh1974NonlinearElements}
	\begin{align}
		y &= \frac{y'}{L}, &  t &= \sqrt{\frac{Er^2}{\rho L^4}}t',
		&
		w &= \frac{L}{r^2}w',     
		\\  p &= \frac{L^7}{r6EA}p',  &
		c &= \frac{L^4}{2r^3 \sqrt{\rho E}} \frac{r^2}{L^2} c',
	\end{align}
	with $(\bullet)'$ denoting the corresponding quantities in physical units. Specifically, $p'$ denotes the distributed load applied along the beam; $c'$ is the distributed damping coefficient; $y', t', w'$ represent the axial coordinate, the time, and the transverse displacement. The axial stretching forces due to the nonlinear bending-stretching coupling are accounted by the term $H$ defined as 
	\begin{align}
		H = \frac{r^2}{2L^2} \int_0^l \big( \frac{\partial w}{\partial y}\big)^2 dy.
	\end{align}
	In this example, we choose the slenderness ratio to be $\frac{r^2}{L^2} = 1 \times 10^{-4}$ in line with the original treatment of Nayfeh et al.\cite{Nayfeh1974NonlinearElements}, who analyze the prismatic beam under weak nonlinearities. A modal expansion of the displacement field $w$ yields
	\begin{align}\label{eq:modal_exp}
		w(y,t) = \sum_{i} \psi_i(y) z_i(t),
	\end{align}
	where $\psi_i(y)$ denotes the $i^{\mathrm{th}}$ eigenfunction associated to the eigenfrequency $\omega_i$ of the linearized version of the PDE~\eqref{eq:PDEPB} around the trivial solution. These eigenfunctions are solutions of the eigenvalue problem 
	\begin{align}\label{eq:PB_lin_modal_eqs}
		\frac{ \partial^4 \psi_i}{\partial y^4} &= \omega_i^2 \psi_i,\\
		\psi(0) &= \psi''(0) = \psi(l) = \psi'(l) = 0,
	\end{align}
	where we choose an orthonormal representation for these eigenfunctions, i.e.,
	\begin{align}
		\int_{0}^l  \psi_i(y) \psi_j(y) = \delta_{ij}dy.
	\end{align}
	Nayfeh et al.~\cite{Nayfeh1974NonlinearElements} have computed explicit expressions for these eigenfunctions in terms of trigonometric and hyperbolic functions. Substituting the modal expansion~\eqref{eq:modal_exp} into system~\eqref{eq:PDEPB}, and performing a Galerkin projection, we obtain the set of ODEs governing the evolution of the modal coordinates $z_j(t)$ as
	\begin{align}\label{eq:modal_eqs}
		\omega_j^2 z_j + \Ddot{z}_j +  2c \epsilon \dot{z}_j= \sum_{i,k,s} \epsilon \alpha_{jiks} z_i z_k z_s + \epsilon   \sum_i p_a(t) z_i a_{ji},
	\end{align}
	where
	\begin{align}
		a_{ji} &= \int_0^l  \psi_j \ \psi''_i dy,
		\\
		\alpha_{jiks} &= \bigg(\frac{r^2}{2L^2} \int_0^l \psi'_s \psi'_k dy \bigg)\bigg(\int_0^l  \psi_j  \psi''_i dy\bigg).
	\end{align}
	We have assumed the external load to be applied at the hinged tip of the beam and set the transversal excitation to zero. Hence, the axial forcing effectively leads to a parametric excitation of the transverse modal coordinates in system~\eqref{eq:modal_eqs}.

	\section*{References}
	
	\bibliographystyle{naturemag}
	\bibliography{Library/references}

\begin{thebibliography}{10}
\expandafter\ifx\csname url\endcsname\relax
  \def\url#1{\texttt{#1}}\fi
\expandafter\ifx\csname urlprefix\endcsname\relax\def\urlprefix{URL }\fi
\providecommand{\bibinfo}[2]{#2}
\providecommand{\eprint}[2][]{\url{#2}}

\bibitem{Champneys2013}
\bibinfo{author}{Champneys, A.}
\newblock \bibinfo{title}{Dynamics of parametric excitation}.
\newblock In \emph{\bibinfo{booktitle}{Encyclopedia of Complexity and Systems
  Science}} (\bibinfo{publisher}{Springer, New York, NY},
  \bibinfo{year}{2013}).
\newblock \urlprefix\url{https://doi.org/10.1007/978-3-642-27737-5_144-3}.

\bibitem{Sanches2012}
\bibinfo{author}{Sanches, L.}, \bibinfo{author}{Michon, G.},
  \bibinfo{author}{Berlioz, A.} \& \bibinfo{author}{Alazard, D.}
\newblock \bibinfo{title}{{Parametrically excited helicopter ground resonance
  dynamics with high blade asymmetries}}.
\newblock \emph{\bibinfo{journal}{Journal of Sound and Vibration}}
  \textbf{\bibinfo{volume}{331}}, \bibinfo{pages}{3897--3913}
  (\bibinfo{year}{2012}).

\bibitem{Coleman1956}
\bibinfo{author}{Coleman, R.~P.} \& \bibinfo{author}{Feingold, A.~M.}
\newblock \bibinfo{title}{{Theory of self-excited mechanical oscillations of
  helicopter rotors with hinged blades}} (\bibinfo{year}{1956}).
\newblock
  \urlprefix\url{https://digital.library.unt.edu/ark:/67531/metadc60767/m1/3/}.

\bibitem{Ecker2011}
\bibinfo{author}{Ecker, H.}
\newblock \bibinfo{title}{{Beneficial Effects of Parametric Excitation in Rotor
  Systems}}.
\newblock In \emph{\bibinfo{booktitle}{IUTAM Bookseries}},
  vol.~\bibinfo{volume}{25}, \bibinfo{pages}{361--371}
  (\bibinfo{publisher}{Springer, Dordrecht}, \bibinfo{year}{2011}).
\newblock
  \urlprefix\url{https://link.springer.com/chapter/10.1007/978-94-007-0020-8_31}.

\bibitem{Tehrani2015}
\bibinfo{author}{Tehrani, M.~G.}, \bibinfo{author}{Bonisoli, E.} \&
  \bibinfo{author}{Scapolan, M.}
\newblock \bibinfo{title}{{Energy Harvesting Perspectives from Parametric
  Resonant Systems}}.
\newblock In \emph{\bibinfo{booktitle}{Conference Proceedings of the Society
  for Experimental Mechanics Series}}, vol.~\bibinfo{volume}{9},
  \bibinfo{pages}{223--232} (\bibinfo{publisher}{Springer, Cham},
  \bibinfo{year}{2015}).
\newblock
  \urlprefix\url{https://link.springer.com/chapter/10.1007/978-3-319-15233-2_23}.

\bibitem{Rhoads2006_MEMS}
\bibinfo{author}{Rhoads, J.~F.} \emph{et~al.}
\newblock \bibinfo{title}{{Generalized parametric resonance in
  electrostatically actuated microelectromechanical oscillators}}.
\newblock \emph{\bibinfo{journal}{Journal of Sound and Vibration}}
  \textbf{\bibinfo{volume}{296}}, \bibinfo{pages}{797--829}
  (\bibinfo{year}{2006}).

\bibitem{Li2014}
\bibinfo{author}{Li, L.}, \bibinfo{author}{Hiller, T.},
  \bibinfo{author}{Bamieh, B.} \& \bibinfo{author}{Turner, K.}
\newblock \bibinfo{title}{{Amplitude control of parametric resonances for mass
  sensing}}.
\newblock In \emph{\bibinfo{booktitle}{Proceedings of IEEE Sensors}}, vol.
  \bibinfo{volume}{2014}, \bibinfo{pages}{198--201}
  (\bibinfo{publisher}{Institute of Electrical and Electronics Engineers Inc.},
  \bibinfo{year}{2014}).

\bibitem{Welte2013}
\bibinfo{author}{Welte, J.}, \bibinfo{author}{Kniffka, T.~J.} \&
  \bibinfo{author}{Ecker, H.}
\newblock \bibinfo{title}{{Parametric excitation in a two degree of freedom
  MEMS system}}.
\newblock \emph{\bibinfo{journal}{Shock and Vibration}}
  \textbf{\bibinfo{volume}{20}}, \bibinfo{pages}{1113--1124}
  (\bibinfo{year}{2013}).

\bibitem{Polunin2017}
\bibinfo{author}{Polunin, P.~M.} \& \bibinfo{author}{Shaw, S.~W.}
\newblock \bibinfo{title}{{Self-induced parametric amplification in ring
  resonating gyroscopes}}.
\newblock \emph{\bibinfo{journal}{International Journal of Non-Linear
  Mechanics}} \textbf{\bibinfo{volume}{94}}, \bibinfo{pages}{300--308}
  (\bibinfo{year}{2017}).

\bibitem{Sharma2012}
\bibinfo{author}{Sharma, M.}, \bibinfo{author}{Sarraf, E.~H.},
  \bibinfo{author}{Baskaran, R.} \& \bibinfo{author}{Cretu, E.}
\newblock \bibinfo{title}{{Parametric resonance: Amplification and damping in
  MEMS gyroscopes}}.
\newblock \emph{\bibinfo{journal}{Sensors and Actuators A: Physical}}
  \textbf{\bibinfo{volume}{177}}, \bibinfo{pages}{79--86}
  (\bibinfo{year}{2012}).

\bibitem{Frangi2017}
\bibinfo{author}{Frangi, A.}, \bibinfo{author}{Guerrieri, A.},
  \bibinfo{author}{Carminati, R.} \& \bibinfo{author}{Mendicino, G.}
\newblock \bibinfo{title}{{Parametric Resonance in Electrostatically Actuated
  Micromirrors}}.
\newblock \emph{\bibinfo{journal}{IEEE Transactions on Industrial Electronics}}
  \textbf{\bibinfo{volume}{64}}, \bibinfo{pages}{1544--1551}
  (\bibinfo{year}{2017}).

\bibitem{Rhoads2009}
\bibinfo{author}{Rhoads, J.~F.} \& \bibinfo{author}{Shaw, S.~W.}
\newblock \bibinfo{title}{{The Effects of Nonlinearity on Parametric
  Amplifiers}}.
\newblock In \emph{\bibinfo{booktitle}{Proceedings of the ASME Design
  Engineering Technical Conference}}, vol.~\bibinfo{volume}{4},
  \bibinfo{pages}{593--597} (\bibinfo{publisher}{American Society of Mechanical
  Engineers Digital Collection}, \bibinfo{year}{2009}).

\bibitem{Verhulst2002}
\bibinfo{author}{Verhulst, F.}
\newblock \bibinfo{title}{{Parametric and Autoparametric Resonance}}.
\newblock \emph{\bibinfo{journal}{Acta Applicandae Mathematica}}
  \textbf{\bibinfo{volume}{70}}, \bibinfo{pages}{231--264}
  (\bibinfo{year}{2002}).
\newblock
  \urlprefix\url{https://link.springer.com/article/10.1023/A:1013934501001}.

\bibitem{Verhulst2009}
\bibinfo{author}{Verhulst, F.}
\newblock \bibinfo{title}{{Perturbation Analysis of Parametric Resonance}}
  (\bibinfo{year}{2009}).
\newblock
  \urlprefix\url{https://link.springer.com/referenceworkentry/10.1007/978-0-387-30440-3_393}.

\bibitem{Jia2016}
\bibinfo{author}{Jia, Y.}, \bibinfo{author}{Du, S.} \& \bibinfo{author}{Seshia,
  A.~A.}
\newblock \bibinfo{title}{{Twenty-Eight Orders of Parametric Resonance in a
  Microelectromechanical Device for Multi-band Vibration Energy Harvesting}}.
\newblock \emph{\bibinfo{journal}{Scientific Reports}}
  \textbf{\bibinfo{volume}{6}}, \bibinfo{pages}{1--8} (\bibinfo{year}{2016}).
\newblock \urlprefix\url{https://www.nature.com/articles/srep30167}.

\bibitem{Verhulst1996}
\bibinfo{author}{Verhulst, F.}
\newblock \emph{\bibinfo{title}{{Nonlinear Differential Equations and Dynamical
  Systems}}} (\bibinfo{publisher}{Springer Berlin Heidelberg},
  \bibinfo{year}{1996}).

\bibitem{Warminski2005}
\bibinfo{author}{Warminski, J.}
\newblock \bibinfo{title}{{Regular and Chaotic Vibrations of a Parametrically
  and Self-Excited System Under Internal Resonance Condition}}.
\newblock \emph{\bibinfo{journal}{Meccanica}} \textbf{\bibinfo{volume}{40}},
  \bibinfo{pages}{181--202} (\bibinfo{year}{2005}).
\newblock
  \urlprefix\url{https://link.springer.com/article/10.1007/s11012-005-3306-4}.

\bibitem{Rhoads2006_Microbeam}
\bibinfo{author}{Rhoads, J.~F.}, \bibinfo{author}{Shaw, S.~W.} \&
  \bibinfo{author}{Turner, K.~L.}
\newblock \bibinfo{title}{{The nonlinear response of resonant microbeam systems
  with purely-parametric electrostatic actuation}}.
\newblock \emph{\bibinfo{journal}{Journal of Micromechanics and
  Microengineering}} \textbf{\bibinfo{volume}{16}}, \bibinfo{pages}{890}
  (\bibinfo{year}{2006}).
\newblock
  \urlprefix\url{https://iopscience.iop.org/article/10.1088/0960-1317/16/5/003
  https://iopscience.iop.org/article/10.1088/0960-1317/16/5/003/meta}.

\bibitem{Li2020}
\bibinfo{author}{Li, D.} \& \bibinfo{author}{Shaw, S.~W.}
\newblock \bibinfo{title}{{The effects of nonlinear damping on degenerate
  parametric amplification}}.
\newblock \emph{\bibinfo{journal}{Nonlinear Dynamics}}
  \textbf{\bibinfo{volume}{102}}, \bibinfo{pages}{2433--2452}
  (\bibinfo{year}{2020}).
\newblock
  \urlprefix\url{https://link.springer.com/article/10.1007/s11071-020-06090-8}.

\bibitem{Zaghari2016}
\bibinfo{author}{Zaghari, B.}, \bibinfo{author}{Rustighi, E.} \&
  \bibinfo{author}{Tehrani, M.~G.}
\newblock \bibinfo{title}{{Dynamic response of a nonlinear parametrically
  excited system subject to harmonic base excitation}}.
\newblock In \emph{\bibinfo{booktitle}{Journal of Physics: Conference Series}},
  vol. \bibinfo{volume}{744}, \bibinfo{pages}{012125} (\bibinfo{publisher}{IOP
  Publishing}, \bibinfo{year}{2016}).
\newblock
  \urlprefix\url{https://iopscience.iop.org/article/10.1088/1742-6596/744/1/012125
  https://iopscience.iop.org/article/10.1088/1742-6596/744/1/012125/meta}.

\bibitem{Szabelski1995}
\bibinfo{author}{Szabelski, K.} \& \bibinfo{author}{Warminski, J.}
\newblock \bibinfo{title}{{Self-excited System Vibrations with Parametric and
  External Excitations}}.
\newblock \emph{\bibinfo{journal}{Journal of Sound and Vibration}}
  \textbf{\bibinfo{volume}{187}}, \bibinfo{pages}{595--607}
  (\bibinfo{year}{1995}).
\newblock
  \urlprefix\url{https://www.sciencedirect.com/science/article/pii/S0022460X85705470}.

\bibitem{Szabelski1997}
\bibinfo{author}{Szabelski, K.} \& \bibinfo{author}{Warmi{\'{n}}ski, J.}
\newblock \bibinfo{title}{{Vibration of a Non-Linear Self-Excited System with
  Two Degrees of Freedom under External and Parametric Excitation}}.
\newblock \emph{\bibinfo{journal}{Nonlinear Dynamics}}
  \textbf{\bibinfo{volume}{14}}, \bibinfo{pages}{23--36}
  (\bibinfo{year}{1997}).
\newblock
  \urlprefix\url{https://link.springer.com/article/10.1023/A:1008227315259}.

\bibitem{Sorokin2015VibrationModulation}
\bibinfo{author}{Sorokin, V.~S.} \& \bibinfo{author}{Thomsen, J.~J.}
\newblock \bibinfo{title}{{Vibration suppression for strings with distributed
  loading using spatial cross-section modulation}}.
\newblock \emph{\bibinfo{journal}{Journal of Sound and Vibration}}
  \textbf{\bibinfo{volume}{335}}, \bibinfo{pages}{66--77}
  (\bibinfo{year}{2015}).

\bibitem{Aghamohammadi2022DynamicExcitations}
\bibinfo{author}{Aghamohammadi, M.}, \bibinfo{author}{Sorokin, V.} \&
  \bibinfo{author}{Mace, B.}
\newblock \bibinfo{title}{{Dynamic analysis of the response of Duffing-type
  oscillators subject to interacting parametric and external excitations}}.
\newblock \emph{\bibinfo{journal}{Nonlinear Dynamics}}
  \textbf{\bibinfo{volume}{107}}, \bibinfo{pages}{99--120}
  (\bibinfo{year}{2022}).
\newblock
  \urlprefix\url{https://link.springer.com/article/10.1007/s11071-021-06972-5}.

\bibitem{Warminski2003}
\bibinfo{author}{Warminski, J.} \& \bibinfo{author}{Balthazar, J.~M.}
\newblock \bibinfo{title}{{Vibrations of a Parametrically and Self-Excited
  System with Ideal and Non-Ideal Energy Sources}}.
\newblock \emph{\bibinfo{journal}{Journal of the Brazilian Society of
  Mechanical Sciences and Engineering}} \textbf{\bibinfo{volume}{25}},
  \bibinfo{pages}{413--420} (\bibinfo{year}{2003}).
\newblock
  \urlprefix\url{http://www.scielo.br/j/jbsmse/a/sckSpJbBpHPcYhn8xr59pFg/?lang=en}.

\bibitem{AUTO}
\bibinfo{author}{Doedel, E.} \& \bibinfo{author}{Oldeman, B.}
\newblock \bibinfo{title}{{AUTO: Software for Continuation and Bifurcation
  Problems in Ordinary Differential Equations}} (\bibinfo{year}{2007}).
\newblock \urlprefix\url{http://indy.cs.concordia.ca/auto/}.

\bibitem{Dankowicz2013}
\bibinfo{author}{Dankowicz, H.} \& \bibinfo{author}{Schilder, F.}
\newblock \emph{\bibinfo{title}{{Recipes for Continuation}}}
  (\bibinfo{publisher}{Society for Industrial and Applied Mathematics},
  \bibinfo{year}{2013}).
\newblock \urlprefix\url{https://epubs.siam.org/page/terms}.

\bibitem{Krack2019}
\bibinfo{author}{Krack, M.} \& \bibinfo{author}{Gross, J.}
\newblock \emph{\bibinfo{title}{{Harmonic balance for nonlinear vibration
  problems}}} (\bibinfo{publisher}{Springer International Publishing},
  \bibinfo{year}{2019}).

\bibitem{Jain2021}
\bibinfo{author}{Jain, S.} \& \bibinfo{author}{Haller, G.}
\newblock \bibinfo{title}{{How to compute invariant manifolds and their reduced
  dynamics in high-dimensional finite element models}}.
\newblock \emph{\bibinfo{journal}{Nonlinear Dynamics}}
  \textbf{\bibinfo{volume}{107}}, \bibinfo{pages}{1417--1459}
  (\bibinfo{year}{2021}).
\newblock
  \urlprefix\url{https://link.springer.com/article/10.1007/s11071-021-06957-4}.

\bibitem{Hansen1985}
\bibinfo{author}{Hansen, J.}
\newblock \bibinfo{title}{{Stability diagrams for coupled Mathieu-equations}}.
\newblock In \emph{\bibinfo{booktitle}{Ingenieur-Archiv}},
  vol.~\bibinfo{volume}{55}, \bibinfo{pages}{463--473}
  (\bibinfo{publisher}{Springer}, \bibinfo{year}{1985}).
\newblock \urlprefix\url{https://link.springer.com/article/10.1007/BF00537654}.

\bibitem{Lindh1970}
\bibinfo{author}{Lindh, K.~G.} \& \bibinfo{author}{Likins, P.~W.}
\newblock \bibinfo{title}{{Infinite determinant methods for stability analysis
  of periodic-coefficient differential equations}}.
\newblock \emph{\bibinfo{journal}{AIAA Journal}} \textbf{\bibinfo{volume}{8}},
  \bibinfo{pages}{680--686} (\bibinfo{year}{1970}).
\newblock \urlprefix\url{https://arc.aiaa.org/doi/abs/10.2514/3.5741}.

\bibitem{Shaw1991}
\bibinfo{author}{Shaw, S.} \& \bibinfo{author}{Pierre, C.}
\newblock \bibinfo{title}{{Non-linear normal modes and invariant manifolds}}.
\newblock \emph{\bibinfo{journal}{Journal of Sound and Vibration}}
  \textbf{\bibinfo{volume}{150}}, \bibinfo{pages}{170--173}
  (\bibinfo{year}{1991}).
\newblock \urlprefix\url{https://hal.archives-ouvertes.fr/hal-01310674
  https://hal.archives-ouvertes.fr/hal-01310674/document}.

\bibitem{Warminski2012}
\bibinfo{author}{Warminski, J.}, \bibinfo{author}{Kecik, K.},
  \bibinfo{author}{Mitura, A.} \& \bibinfo{author}{Bochenski, M.}
\newblock \bibinfo{title}{{Nonlinear phenomena in mechanical system dynamics}}.
\newblock In \emph{\bibinfo{booktitle}{Journal of Physics: Conference Series}},
  vol. \bibinfo{volume}{382}, \bibinfo{pages}{012004} (\bibinfo{publisher}{IOP
  Publishing}, \bibinfo{year}{2012}).
\newblock
  \urlprefix\url{https://iopscience.iop.org/article/10.1088/1742-6596/382/1/012004
  https://iopscience.iop.org/article/10.1088/1742-6596/382/1/012004/meta}.

\bibitem{Sinha2005}
\bibinfo{author}{Sinha, S.~C.}, \bibinfo{author}{Redkar, S.},
  \bibinfo{author}{Deshmukh, V.} \& \bibinfo{author}{Butcher, E.~A.}
\newblock \bibinfo{title}{{Order reduction of parametrically excited nonlinear
  systems: Techniques and applications}}.
\newblock \emph{\bibinfo{journal}{Nonlinear Dynamics}}
  \textbf{\bibinfo{volume}{41}}, \bibinfo{pages}{237--273}
  (\bibinfo{year}{2005}).

\bibitem{SinhaManifolds}
\bibinfo{author}{Sinha, S.~C.}, \bibinfo{author}{Redkar, S.} \&
  \bibinfo{author}{Butcher, E.~A.}
\newblock \bibinfo{title}{{Order reduction of nonlinear systems with time
  periodic coefficients using invariant manifolds}}.
\newblock \emph{\bibinfo{journal}{Journal of Sound and Vibration}}
  \textbf{\bibinfo{volume}{284}}, \bibinfo{pages}{985--1002}
  (\bibinfo{year}{2005}).

\bibitem{ssmexist}
\bibinfo{author}{Haller, G.} \& \bibinfo{author}{Ponsioen, S.}
\newblock \bibinfo{title}{{Nonlinear normal modes and spectral submanifolds:
  existence, uniqueness and use in model reduction}}.
\newblock \emph{\bibinfo{journal}{Nonlinear Dynamics}}
  \textbf{\bibinfo{volume}{86}}, \bibinfo{pages}{1493--1534}
  (\bibinfo{year}{2016}).
\newblock
  \urlprefix\url{https://link.springer.com/article/10.1007/s11071-016-2974-z}.

\bibitem{SSMTool}
\bibinfo{author}{Ponsioen, S.}, \bibinfo{author}{Pedergnana, T.} \&
  \bibinfo{author}{Haller, G.}
\newblock \bibinfo{title}{{Automated computation of autonomous spectral
  submanifolds for nonlinear modal analysis}}.
\newblock \emph{\bibinfo{journal}{Journal of Sound and Vibration}}
  \textbf{\bibinfo{volume}{420}}, \bibinfo{pages}{269--295}
  (\bibinfo{year}{2018}).
\newblock
  \urlprefix\url{https://www.sciencedirect.com/science/article/pii/S0022460X18300701}.

\bibitem{isolatedresponsesten}
\bibinfo{author}{Ponsioen, S.}, \bibinfo{author}{Pedergnana, T.} \&
  \bibinfo{author}{Haller, G.}
\newblock \bibinfo{title}{{Analytic prediction of isolated forced response
  curves from spectral submanifolds}}.
\newblock \emph{\bibinfo{journal}{Nonlinear Dynamics}}
  \textbf{\bibinfo{volume}{98}}, \bibinfo{pages}{2755--2773}
  (\bibinfo{year}{2019}).
\newblock
  \urlprefix\url{https://link.springer.com/article/10.1007/s11071-019-05023-4}.

\bibitem{MingwuPart1}
\bibinfo{author}{Li, M.}, \bibinfo{author}{Jain, S.} \&
  \bibinfo{author}{Haller, G.}
\newblock \bibinfo{title}{{Nonlinear analysis of forced mechanical systemswith
  internal resonance using spectral submanifolds, Part I: Periodic response and
  forced response curve}}.
\newblock \emph{\bibinfo{journal}{Nonlinear Dynamics}}
  \textbf{\bibinfo{volume}{110}}, \bibinfo{pages}{1005--1043}
  (\bibinfo{year}{2022}).
\newblock
  \urlprefix\url{https://link.springer.com/article/10.1007/s11071-022-07714-x}.

\bibitem{MingwuPart2}
\bibinfo{author}{Li, M.} \& \bibinfo{author}{Haller, G.}
\newblock \bibinfo{title}{{Nonlinear analysis of forced mechanical systems with
  internal resonance using spectral submanifolds, Part II: Bifurcation and
  quasi-periodic response}}.
\newblock \emph{\bibinfo{journal}{Nonlinear Dynamics}}
  \textbf{\bibinfo{volume}{110}}, \bibinfo{pages}{1045--1080}
  (\bibinfo{year}{2022}).
\newblock
  \urlprefix\url{https://link.springer.com/article/10.1007/s11071-022-07476-6}.

\bibitem{Cenedese2022Data-drivenSystems}
\bibinfo{author}{Cenedese, M.}, \bibinfo{author}{Axas, J.},
  \bibinfo{author}{Yang, H.}, \bibinfo{author}{Eriten, M.} \&
  \bibinfo{author}{Haller, G.}
\newblock \bibinfo{title}{{Data-driven nonlinear model reduction to spectral
  submanifolds in mechanical systems}}.
\newblock \emph{\bibinfo{journal}{Philosophical Transactions of the Royal
  Society A}} \textbf{\bibinfo{volume}{380}}, \bibinfo{pages}{--}
  (\bibinfo{year}{2022}).
\newblock
  \urlprefix\url{https://royalsocietypublishing.org/doi/10.1098/rsta.2021.0194}.

\bibitem{Mahlkneckt2022}
\bibinfo{author}{Mahlknecht, F.} \emph{et~al.}
\newblock \bibinfo{title}{Using spectral submanifolds for nonlinear periodic
  control}.
\newblock In \emph{\bibinfo{booktitle}{2022 IEEE 61st Conference on Decision
  and Control (CDC)}}, \bibinfo{pages}{6548--6555} (\bibinfo{year}{2022}).

\bibitem{Alora2022Data-DrivenRobots}
\bibinfo{author}{Alora, J.~I.}, \bibinfo{author}{Cenedese, M.},
  \bibinfo{author}{Schmerling, E.}, \bibinfo{author}{Haller, G.} \&
  \bibinfo{author}{Pavone, M.}
\newblock \bibinfo{title}{{Data-Driven Spectral Submanifold Reduction for
  Nonlinear Optimal Control of High-Dimensional Robots}}.
\newblock \emph{\bibinfo{journal}{Arxiv-preprint}}  (\bibinfo{year}{2022}).
\newblock \urlprefix\url{https://arxiv.org/abs/2209.05712v3}.

\bibitem{Li2022ModelSubmanifolds}
\bibinfo{author}{Li, M.}, \bibinfo{author}{Jain, S.} \&
  \bibinfo{author}{Haller, G.}
\newblock \bibinfo{title}{{Model reduction for constrained mechanical systems
  via spectral submanifolds}}.
\newblock \emph{\bibinfo{journal}{Nonlinear Dynamics}}  (\bibinfo{year}{2022}).
\newblock \urlprefix\url{http://arxiv.org/abs/2208.03119
  http://dx.doi.org/10.1007/s11071-023-08300-5}.

\bibitem{Ponsioen2020}
\bibinfo{author}{Ponsioen, S.}, \bibinfo{author}{Jain, S.} \&
  \bibinfo{author}{Haller, G.}
\newblock \bibinfo{title}{{Model reduction to spectral submanifolds and
  forced-response calculation in high-dimensional mechanical systems}}.
\newblock \emph{\bibinfo{journal}{Journal of Sound and Vibration}}
  \textbf{\bibinfo{volume}{488}}, \bibinfo{pages}{115640}
  (\bibinfo{year}{2020}).

\bibitem{Opreni2022}
\bibinfo{author}{Opreni, A.}, \bibinfo{author}{Vizzaccaro, A.},
  \bibinfo{author}{Touz{\'{e}}, C.} \& \bibinfo{author}{Frangi, A.}
\newblock \bibinfo{title}{{High-order direct parametrisation of invariant
  manifolds for model order reduction of finite element structures: application
  to generic forcing terms and parametrically excited systems}}.
\newblock \emph{\bibinfo{journal}{Nonlinear Dynamics}}
  \textbf{\bibinfo{volume}{111}}, \bibinfo{pages}{5401--5447}
  (\bibinfo{year}{2022}).
\newblock
  \urlprefix\url{https://link.springer.com/article/10.1007/s11071-022-07978-3}.

\bibitem{Vizzaccaro2022}
\bibinfo{author}{Vizzaccaro, A.}, \bibinfo{author}{Opreni, A.},
  \bibinfo{author}{Salles, L.}, \bibinfo{author}{Frangi, A.} \&
  \bibinfo{author}{Touz{\'{e}}, C.}
\newblock \bibinfo{title}{{High order direct parametrisation of invariant
  manifolds for model order reduction of finite element structures: application
  to large amplitude vibrations and uncovering of a folding point}}.
\newblock \emph{\bibinfo{journal}{Nonlinear Dynamics}}
  \textbf{\bibinfo{volume}{110}}, \bibinfo{pages}{525--571}
  (\bibinfo{year}{2022}).
\newblock
  \urlprefix\url{https://link.springer.com/article/10.1007/s11071-022-07651-9}.

\bibitem{SSMTool2}
\bibinfo{author}{Jain, S.}, \bibinfo{author}{Thurnher, T.},
  \bibinfo{author}{Li, M.} \& \bibinfo{author}{Haller, G.}
\newblock \bibinfo{title}{{SSMTool 2.4: Computation of invariant manifolds in
  high-dimensional mechanics problems}} (\bibinfo{year}{2023}).
\newblock \urlprefix\url{https://doi.org/10.5281/zenodo.4614201}.

\bibitem{Golub2013MatrixComputations}
\bibinfo{author}{Golub, G.~H.} \& \bibinfo{author}{Loan, C. F.~V.}
\newblock \emph{\bibinfo{title}{{Matrix Computations}}}
  (\bibinfo{publisher}{Johns Hopkins University Press}, \bibinfo{year}{2013}).
\newblock
  \urlprefix\url{https://jhupbooks.press.jhu.edu/title/matrix-computations}.

\bibitem{Guckenheimer1983}
\bibinfo{author}{Guckenheimer, J.} \& \bibinfo{author}{Holmes, P.}
\newblock \emph{\bibinfo{title}{{Nonlinear Oscillations, Dynamical Systems, and
  Bifurcations of Vector Fields}}}, vol.~\bibinfo{volume}{42} of
  \emph{\bibinfo{series}{Applied Mathematical Sciences}}
  (\bibinfo{publisher}{Springer New York}, \bibinfo{address}{New York, NY},
  \bibinfo{year}{1983}).
\newblock \urlprefix\url{http://link.springer.com/10.1007/978-1-4612-1140-2}.

\bibitem{Haro2016}
\bibinfo{author}{Haro, A.}, \bibinfo{author}{Canadell, M.},
  \bibinfo{author}{Figueras, J.-L.}, \bibinfo{author}{Luque, A.} \&
  \bibinfo{author}{Mondelo, J.~M.}
\newblock \emph{\bibinfo{title}{{The Parameterization Method for Invariant
  Manifolds}}}, vol.~\bibinfo{volume}{1} of \emph{\bibinfo{series}{Applied
  Mathematical Sciences}} (\bibinfo{publisher}{Springer International
  Publishing}, \bibinfo{address}{Cham}, \bibinfo{year}{2016}).
\newblock \urlprefix\url{http://link.springer.com/10.1007/978-3-319-29662-3}.

\bibitem{Chen2001}
\bibinfo{author}{Chen, C.~C.} \& \bibinfo{author}{Yeh, M.~K.}
\newblock \bibinfo{title}{{Parametric instability of a beam under
  electromagnetic excitation}}.
\newblock \emph{\bibinfo{journal}{Journal of Sound and Vibration}}
  \textbf{\bibinfo{volume}{240}}, \bibinfo{pages}{747--764}
  (\bibinfo{year}{2001}).

\bibitem{Nayfeh1974NonlinearElements}
\bibinfo{author}{Nayfeh, A.~H.}, \bibinfo{author}{Mook, D.~T.} \&
  \bibinfo{author}{Sridhar, S.}
\newblock \bibinfo{title}{{Nonlinear analysis of the forced response of
  structural elements}}.
\newblock \emph{\bibinfo{journal}{Journal of the Acoustical Society of
  America}} \textbf{\bibinfo{volume}{55}}, \bibinfo{pages}{281--291}
  (\bibinfo{year}{1974}).

\bibitem{Barari2011NonlinearBeams}
\bibinfo{author}{Barari, A.}, \bibinfo{author}{Kaliji, H.~D.},
  \bibinfo{author}{Ghadimi, M.} \& \bibinfo{author}{Domairry, G.}
\newblock \bibinfo{title}{{Nonlinear Vibration of Euler-Bernoulli Beams}}.
\newblock \emph{\bibinfo{journal}{Latin American Journal of Solids and
  Structures}} \textbf{\bibinfo{volume}{8}}, \bibinfo{pages}{139}
  (\bibinfo{year}{2011}).
\newblock
  \urlprefix\url{https://www.lajss.org/index.php/LAJSS/article/view/308}.

\bibitem{Lu2017}
\bibinfo{author}{Lu, Z.~Q.}, \bibinfo{author}{Li, J.~M.},
  \bibinfo{author}{Ding, H.} \& \bibinfo{author}{Chen, L.~Q.}
\newblock \bibinfo{title}{{Analysis and suppression of a self-excitation
  vibration via internal stiffness and damping nonlinearity}}.
\newblock \emph{\bibinfo{journal}{Advances in Mechanical Engineering}}
  \textbf{\bibinfo{volume}{9}}, \bibinfo{pages}{1687814017744024}
  (\bibinfo{year}{2017}).
\newblock
  \urlprefix\url{https://journals.sagepub.com/doi/full/10.1177/1687814017744024}.

\bibitem{Chatterjee2007}
\bibinfo{author}{Chatterjee, S.}
\newblock \bibinfo{title}{{Non-linear control of friction-induced self-excited
  vibration}}.
\newblock \emph{\bibinfo{journal}{International Journal of Non-Linear
  Mechanics}} \textbf{\bibinfo{volume}{42}}, \bibinfo{pages}{459--469}
  (\bibinfo{year}{2007}).

\bibitem{Malas2015}
\bibinfo{author}{Malas, A.} \& \bibinfo{author}{Chatterjee, S.}
\newblock \bibinfo{title}{{Analysis and synthesis of modal and non-modal
  self-excited oscillations in a class of mechanical systems with nonlinear
  velocity feedback}}.
\newblock \emph{\bibinfo{journal}{Journal of Sound and Vibration}}
  \textbf{\bibinfo{volume}{334}}, \bibinfo{pages}{296--318}
  (\bibinfo{year}{2015}).

\bibitem{Fu2011}
\bibinfo{author}{Fu, H.}, \bibinfo{author}{Liu, C.}, \bibinfo{author}{Liu, Y.},
  \bibinfo{author}{Chu, J.} \& \bibinfo{author}{Cao, G.}
\newblock \bibinfo{title}{{Selective photothermal self-excitation of mechanical
  modes of a micro-cantilever for force microscopy}}.
\newblock \emph{\bibinfo{journal}{Applied Physics Letters}}
  \textbf{\bibinfo{volume}{99}}, \bibinfo{pages}{173501}
  (\bibinfo{year}{2011}).
\newblock \urlprefix\url{https://aip.scitation.org/doi/abs/10.1063/1.3655333}.

\bibitem{Warminski2DOF1995}
\bibinfo{author}{Warminski, J.} \& \bibinfo{author}{Szabelski, K.}
\newblock \bibinfo{title}{{The non-linear vibrations of parametrically
  self-excited system with two degrees of freedom}}.
\newblock \emph{\bibinfo{journal}{Journal of Theoretical and Applied
  Mechanics}} \textbf{\bibinfo{volume}{33}}, \bibinfo{pages}{643--665}
  (\bibinfo{year}{1995}).
\newblock
  \urlprefix\url{http://www.ptmts.org.pl/jtam/index.php/jtam/article/view/v33n3p643}.

\bibitem{Gobat2021ReducedResonance}
\bibinfo{author}{Gobat, G.} \emph{et~al.}
\newblock \bibinfo{title}{{Reduced order modelling and experimental validation
  of a MEMS gyroscope test-structure exhibiting 1:2 internal resonance}}.
\newblock \emph{\bibinfo{journal}{Scientific Reports}}
  \textbf{\bibinfo{volume}{11}} (\bibinfo{year}{2021}).

\end{thebibliography}
	
\end{document}